\documentclass[preprint,12pt]{elsarticle}
\graphicspath{{figures/}}

\usepackage{amssymb}
\usepackage{subfigure}
\usepackage{amsmath}
\usepackage{color}
\usepackage{hyperref}

\usepackage{amsthm}
\usepackage{enumitem}
\usepackage{pgfplotstable}
\usepackage{pgfplots}
\pgfplotsset{compat=1.14}
\pgfplotsset{
	colormap={my basis colormap}{
		rgb255=(0, 114, 189);
		rgb255=(54, 106, 148);
		rgb255=(108, 98, 107);
		rgb255=(163, 91, 66);
		rgb255=(217, 83, 25);
	}
}
\pgfplotsset{
	colormap={my parula}{
		rgb255=(53.0655, 42.406, 134.9460);
		rgb255=(20.2336, 132.6061, 211.9511);
		rgb255=(55.5463, 184.8859, 157.9118);
		rgb255=(208.7187, 186.8470,  89.1966);
		rgb255=(248.9565, 250.6905, 13.7190);
	}
}
\usepackage{tikz}
\usepgfmodule{oo}
\usetikzlibrary{calc}
\usetikzlibrary{shapes}
\usetikzlibrary{plotmarks}
\usetikzlibrary{backgrounds}
\usetikzlibrary{decorations.pathmorphing}
\usetikzlibrary{external}
\usetikzlibrary{arrows,fit,matrix,positioning,shapes.geometric}

\usepackage{titlesec}

\titleformat{\paragraph}
{\normalfont\normalsize\it}{\theparagraph}{1em}{}
\titlespacing*{\paragraph}
{0pt}{3.25ex plus 1ex minus .2ex}{1.5ex plus .2ex}

\journal{Computer Methods in Applied Mechanics and Engineering}

\def\mathbf#1{\boldsymbol{#1}}

\def\vec#1{\boldsymbol{#1}}

\theoremstyle{plain}
\newtheorem{theorem}{Theorem}[section]

\newtheorem{lemma}[theorem]{Lemma}

\theoremstyle{remark}

\theoremstyle{definition}

\definecolor{myBlue}{rgb} {0,0.4470,0.7410}
\definecolor{myRed}{rgb} {0.8500,0.3250,0.0980}

\usepackage{marginnote}

\begin{document}

\begin{frontmatter}



 
  \title{Analysis-suitable unstructured T-splines: Multiple extraordinary points per face}


\author[label2]{Xiaodong Wei\corref{cor1}}
\ead{xiaodong.wei@epfl.ch}
\author[label1]{Xin Li}
\author[label3]{Kuanren Qian}
\author[label4]{Thomas J.R. Hughes}
\author[label3]{Yongjie Jessica Zhang}
\author[label5]{Hugo Casquero\corref{cor1}}
\ead{casquero@umich.edu}
\address[label1]{School of Mathematical Science, USTC, Hefei, China.}
\address[label2]{Institute of Mathematics, \'Ecole Polytechnique F\'ed\'erale de Lausanne, 1015 Lausanne, Switzerland.}
\address[label3]{Department of Mechanical Engineering, Carnegie Mellon University, Pittsburgh, PA 15213, U.S.A.}
\address[label4]{Oden Institute for Computational Engineering and Sciences, 201 East 24th Street, C0200, Austin, TX 78712-1229, U.S.A.}
\address[label5]{Department of Mechanical Engineering, University of Michigan – Dearborn, 4901 Evergreen Road, Dearborn, MI 48128-1491, U.S.A.}

\cortext[cor1]{Corresponding authors.}

\begin{abstract}

Analysis-suitable T-splines (AST-splines)  are a promising candidate to achieve a seamless integration between the design and the analysis of thin-walled structures in industrial settings. In this work, we generalize AST-splines to allow multiple extraordinary points within the same face. This generalization drastically increases the flexibility to build geometries using AST-splines; e.g., much coarser meshes can be generated to represent a certain geometry. The AST-spline spaces detailed in this work have $C^1$ inter-element continuity near extraordinary points and $C^2$ inter-element continuity elsewhere. We mathematically show that AST-splines with multiple extraordinary points per face are linearly independent and their polynomial basis functions form a non-negative partition of unity. We numerically show that AST-splines with multiple extraordinary points per face lead to optimal convergence rates for second- and fourth-order linear elliptic problems. To illustrate a possible isogeometric framework that is already available, we design the B-pillar and the side outer panel of a car using T-splines with the commercial software Autodesk Fusion360, import the control nets into our in-house code to build AST-splines, and import the B\'ezier extraction information into the commercial software LS-DYNA to solve eigenvalue problems. The results are compared with conventional finite elements. Good agreement is found, but conventional finite elements require significantly more degrees of freedom to reach a converged solution than AST-splines.



\end{abstract}


\end{frontmatter}


\renewcommand{\thefootnote}{\fnsymbol{footnote}}

\section{Introduction}

In computer-aided-design (CAD) programs, there are two major paradigms to represent surfaces with arbitrary topological genus, namely, trimmed NURBS and spline constructions that handle extraordinary points (vertices around which the mesh topology is unstructured). Trimmed NURBS representations are the more widespread paradigm, available in all CAD programs. In this case, hundreds (or thousands) of tensor-product NURBS patches \cite{Rogers2001}, each with its own parameterization, are used to represent the surface. In addition, most of the NURBS patches are trimmed by NURBS curves. When a NURBS patch is trimmed, the parameterization and the control points of the NURBS patch are not changed in any way. Instead, computer-graphics techniques are used to show only one of the two parts in which the trimming curve divides the NURBS patch. Therefore, trimmed NURBS representations lack a conforming parameterization, i.e., a parameterization associated with the geometry of the final surface. Furthermore, when trimmed NURBS patches are joined together, small gaps between patches are often inevitable. By contrast, spline constructions that handle extraordinary points (also known as star points) have a conforming parameterization for the whole surface. In addition, watertight surfaces, i.e., surfaces without small gaps or overlaps, are always obtained. Subdivision surfaces (SubD) \cite{catmull1978recursively} and T-splines \cite{Sederberg2003, Sederberg2004, Sederberg2008} are the two spline constructions that handle extraordinary points (EPs) which have had more success in CAD programs thus far. SubD are available in the commercial software Rhinoceros 3D and T-splines are available in the commercial software Autodesk Fusion360. SubD are also available in most programs for computer animation (Pixar, Autodesk 3ds Max, Autodesk Maya, Zbrush, Blender, among others).

CAD surfaces are the main input for many downstream applications such as finite element analysis (FEA) and computer-aided manufacturing (CAM). Getting a bilinear quadrilateral mesh for conventional FEA from a trimmed NURBS representation is a non-trivial task. Even though there are several methods that aim at automating this process \cite{bommes2009mixed, ebke2013qex, bommes2013quad, liao2014structure, hiemstra2020towards, chen2019quadrilateral, lei2020quadrilateral, shepherd2020quad}, coming up with an algorithm that automatically delivers a non-distorted bilinear quadrilateral mesh suitable for FEA from any trimmed NURBS representation remains an open problem. This problem is particularly hard to automate due to the geometric imperfections (gaps, overlaps, etc.) often found in trimmed NURBS representations. As a result, in complex engineering applications, transitioning from a trimmed NURBS representation to a high-quality bilinear quadrilateral mesh is often the task that requires more manpower in the whole design-through-analysis cycle \cite{hardwick2005dart, cottrell2009isogeometric}.

Isogeometric analysis (IGA) \cite{1003.000} aims at achieving a seamless integration between CAD and FEA programs. IGA performs numerical simulations using different types of splines as basis functions as opposed to using Lagrange polynomials as in conventional FEA. An untrimmed NURBS patch was shown to be a suitable basis for analysis in \cite{bazilevs2006isogeometric, hughes2008duality, dof, da2011some}. The lack of a conforming parameterization in trimmed NURBS representations rules out the use of standard boundary-fitted methods. To circumvent this issue, non-boundary-fitted methods have been developed in recent years \cite{nagy2015numerical, breitenberger2015analysis, leidinger2019explicit, buffa2020minimal, antolin2019overlapping, wei2020immersed}\footnote{Non-boundary-fitted methods were generalized to deal with volumes instead of surfaces in \cite{antolin2019isogeometric}.}. SubD and T-splines have a conforming parameterization, but they are not directly analysis suitable. SubD uses infinite recursion formulas in the faces around EPs to reach global $C^1$ continuity. This infinite recursion is not amenable to numerical integration \cite{WAWRZINEK201660, juttler2016numerical} and limits convergence rates \cite{arden2001approximation, juttler2016numerical}. Variants of SubD have been proposed in the literature that recover optimal convergence rates for second-order linear elliptic problems \cite{XiaodongSubD2021}. Besides handling EPs, T-splines allow the presence of T-junctions to perform local refinement. T-junctions may result in lack of linear independence and/or partition of unity of the polynomial basis functions \cite{buffa2010linear}. To remedy this, topological constraints on the distribution of T-junctions were developed that guarantee linear independence and partition of unity of the polynomial basis functions \cite{Li2012, Scott2012, Li2014, DaVeiga2013, Veiga2012, bbs2015, Casquero2016, wei2017truncated}. This subset of T-splines is called analysis-suitable T-splines (AST-splines). Regarding EPs, T-splines have been combined with various EP constructions over the years \cite{sederberg1998non, peters2000patching, loop2004second, Scott2013}. In \cite{toshniwal2017smooth}, the subset of AST-splines was extended to handle EPs by using the D-patch framework \cite{reif1997refineable}. This was the first smooth EP construction with proof of linear independence and that led to optimal convergence rates in the context of T-splines. In \cite{CASQUERO2021109872}, this technology was used to represent capsules in simulations of fluid-structure interaction. In \cite{casquero2020seamless}, the combination of AST-splines with the D-patch framework was further improved and streamlined through the use of truncation \cite{Giannelli2012, Wei2018}. In addition, complex geometries were built to show the potential of AST-splines in industrial applications. Apart from SubD and T-splines, there have been other spline constructions that handle EPs proposed in the literature, e.g., manifold splines \cite{grimm1995modeling, navau2000modeling, ying2004simple, gu2005manifold, tosun2011manifold} and PHT-splines \cite{deng2008polynomial, li2010polynomial, kang2015new}. However, these alternatives have not made it into CAD software since control points do not behave as geometric shape handles. Manifold splines achieve optimal convergence rates although require a high number of quadrature points to do that \cite{majeed2017isogeometric, zhang2020manifold}. PHT-splines, when combined with the D-patch framework, also achieve optimal convergence rates \cite{nguyen2016refinable}.

In this work, we extend the subset of AST-splines to allow multiple EPs per face, even the four vertices of a face can be EPs. In previous works \cite{toshniwal2017smooth,casquero2020seamless}, AST-splines require EPs to be at least three rings apart from each other. When EPs are at least three rings apart from each other, both the mathematical proof of linear independence and the implementation of the algorithms are significantly simplified. However, allowing multiples EPs per face radically increases the topological flexibility to build a certain geometry using AST-splines; e.g., larger element sizes around holes can be used. This work shows that AST-splines with multiple EPs per face are linearly independent, their polynomial basis functions form a non-negative partition of unity, and have optimal convergence rates for second- and fourth-order linear elliptic problems. To show the potential of this technology, we combine the basis functions of AST-splines with control nets from the commercial software Autodesk Fusion360 to design the B-pillar and the side outer panel of a car from scratch. We thicken these complex AST-spline surfaces to obtain AST-spline volumes using B-splines in the thickness direction. These globally $C^1$-continuous AST-spline volumes are imported in the commercial software LS-DYNA using B\'ezier extraction so as to solve eigenvalue problems and perform comparisons with conventional finite elements.

The paper is outlined as follows. Section 2 describes how AST-splines with multiple EPs per face are built. Section 3 contains the proofs of linear independence and non-negative partition of unity of the polynomial basis functions. Section 4 computes the convergence rates of AST-splines in second- and fourth-order linear elliptic problems. Section 5 exemplifies the potential of this technology in the automotive industry. Conclusions are drawn in Section 6.

\section{Analysis-suitable T-splines}

In this section, we explain how to construct bi-cubic AST-spline surfaces with multiple EPs per face satisfying the following properties:

\begin{enumerate}
\item[(1)] Linear independence of the blending functions; that is, the blending functions constitute a basis.
\item[(2)] Partition of unity of the polynomial basis functions.
\item[(3)] Each basis function is pointwise non-negative.
\item[(4)] At least $C^1$ continuity everywhere while having a finite representation.
\item[(5)] Local support of the basis functions.
\item[(6)] Local $h$-refinement capabilities.
\item[(7)] Optimal convergence rates with respect to both the mesh size $h$ and the square root of the number of degrees of freedom when solving second- and fourth-order linear elliptic partial differential equations. 
\end{enumerate}

Property (2) implies that an affine transformation of an AST-spline surface is obtained by applying the transformation to its control points. Properties (2) and (3) guarantee that AST-spline surfaces satisfy the convex-hull property. In this section, we assume that the reader is familiar with NURBS \cite{Rogers2001, piegl2012nurbs} and B\'ezier extraction \cite{Borden2011, Scott2011}.

\subsection{T-mesh}

The \textit{T-mesh} determines the connectivity among the different components of a T-spline surface. Fig. \ref{tmesh} (a) shows an example of a T-mesh. The T-mesh \textit{vertices} are indicated with circles in Fig. \ref{tmesh} (a). The T-mesh \textit{edges} are closed line segments that connect two vertices without passing through any other vertex. The T-mesh \textit{faces} are four-sided regions delimited by edges that do not have any interior vertex or edge. Note that a face side may include more than one edge. In Fig. \ref{tmesh} (a), edges and faces are represented by solid black lines and white regions, respectively. The \textit{valence} of a vertex, denoted by $\mu$, is the number of edges that emanate from that vertex. A T-junction is a vertex located in the interior of a face side. T-junctions are vertices with valence 3. T-junctions are marked with blue circles in Fig. \ref{tmesh} (a). \textit{Extraordinary points (EPs)} are either interior vertices with $\mu \neq 4$ that are not T-junctions or boundary vertices with $\mu > 3$. EPs are marked with red circles in Fig. \ref{tmesh} (a). The edges emanating from an EP are called \textit{spoke edges}.

\begin{figure} [t!] 
 \centering
 \subfigure[T-mesh]{\includegraphics[scale=.45]{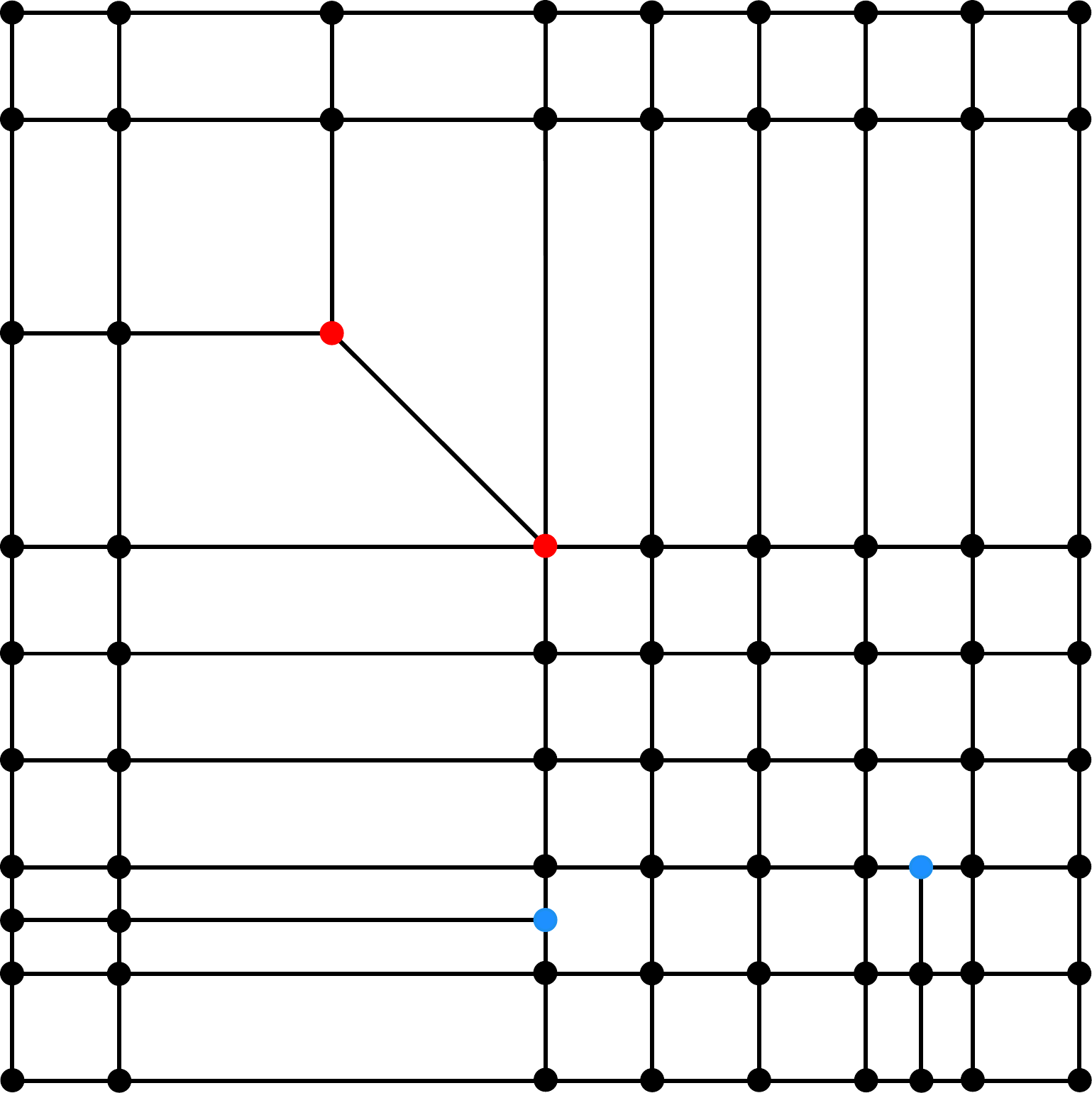}} \hspace*{+6mm}
 \subfigure[Extended T-mesh]{\includegraphics[scale=.45]{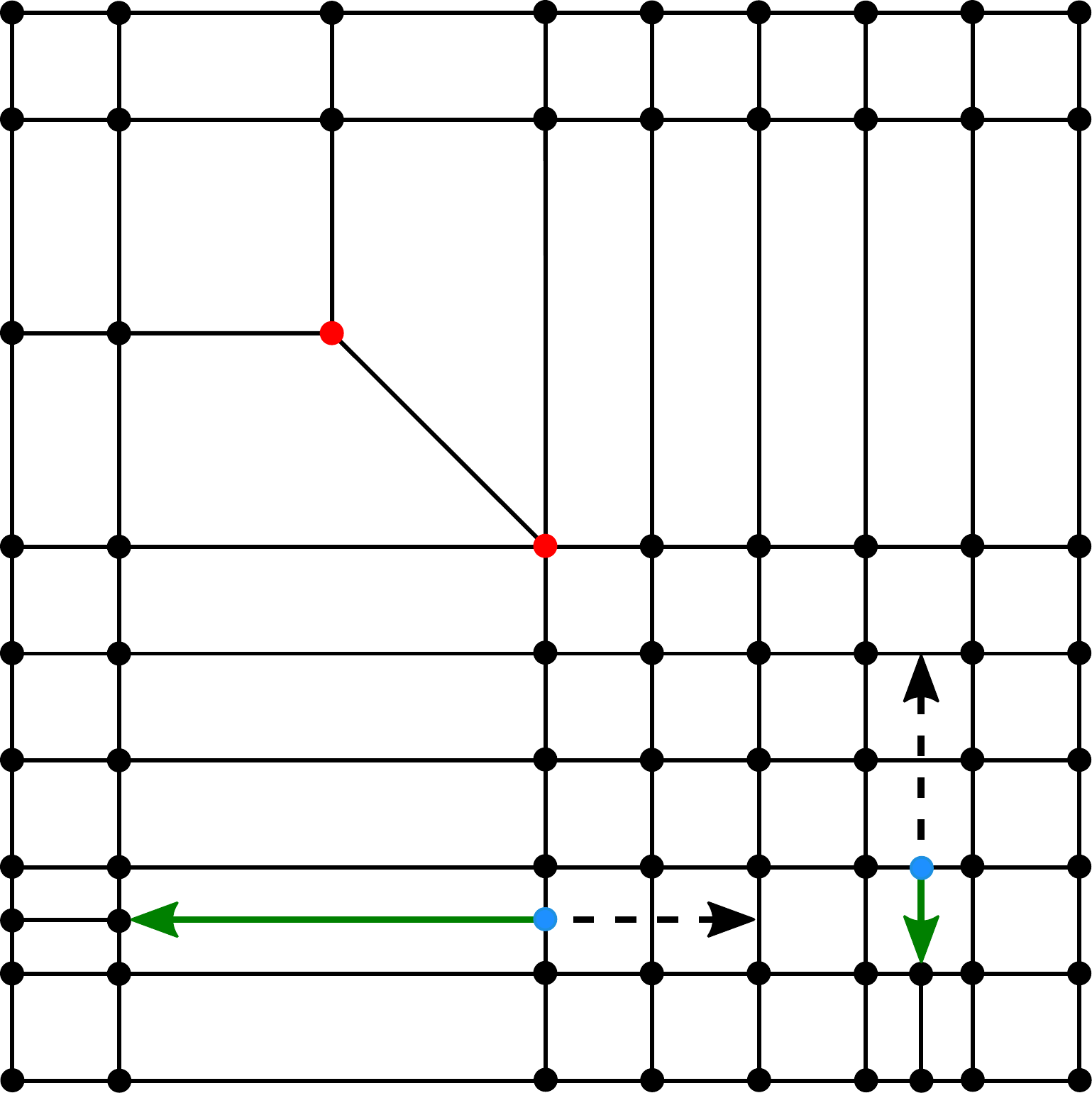}} \\
\caption{(Color online) (a) T-mesh with two EPs and two T-junctions. The EPs and T-junctions are marked with red and blue circles, respectively. (b) Extended T-mesh. Face extensions and edge extensions are represented with dashed black lines and green solid lines, respectively.}
\label{tmesh}
\end{figure}

The \textit{1-ring faces} of an EP are the faces that are in contact with the EP. For $m > 1$, the \textit{m-ring faces} of an EP are all faces that touch the ($m$-1)-ring faces and are not a part of the ($m$-2)-ring faces. The \textit{m-disk faces} of an EP are defined as the set containing all its 1-, 2-, ... , $m$-ring faces. The set of \textit{0-ring vertices} of an EP contains only the EP itself. For $m > 0$, the
\textit{m-ring vertices} of an EP contain all the vertices that lie on the $m$-ring faces and are not a part of the ($m$-1)-ring vertices. The \textit{m-disk vertices} of an EP are the union of all its 0-, 1-, ... , $m$-ring vertices.

A \textit{face extension} is a closed directed line segment that originates at a T-junction and moves in the direction of the missing edge until two orthogonal edges are encountered. A \textit{one-bay} face extension is the part of a face extension that lies on the face adjacent to the T-junction. An \textit{edge extension} is a closed directed line segment that originates at a T-junction and moves in the opposite direction of the face extension until one orthogonal edge is encountered. A \textit{T-junction extension} is composed of a face and an edge extension. Since T-junction extensions are closed line segments, a T-junction extension can intersect with other T-junction extension either in its interior or at its endpoints. The \textit{extended T-mesh} is obtained adding the T-junction extensions to the T-mesh. Fig. \ref{tmesh} (b) plots the extended T-mesh associated with the T-mesh shown in Fig. \ref{tmesh} (a).

The set of \textit{1-layer} faces around the T-mesh boundary contains the faces that are in contact with the T-mesh boundary. For $m > 1$, the set of \textit{m-layer} faces around the T-mesh boundary are all faces that touch the ($m$-1)-layer faces and are not a part of the ($m$-2)-layer faces. The set of \textit{0-layer vertices} around the T-mesh boundary contains the vertices at the T-mesh boundary. For $m > 0$, the
\textit{$m$-layer vertices} around the T-mesh boundary contain all the vertices that lie on the $m$-layer faces but are not a part of the ($m$-1)-layer vertices.

In this work, a T-mesh is \textit{admissible} when it satisfies the following conditions:

\begin{itemize}
\item No one-bay face extension subdivides a 3-disk face of an EP.
\item No perpendicular T-junction extensions intersect.
\item No EP belongs to the 0- and 1-layer vertices around the T-mesh boundary.
\item No T-junction parallel to the boundary belongs to the 0- and 1-layer vertices around the T-mesh boundary.
\end{itemize}

AST-splines are T-splines defined over an admissible T-mesh. In preceding works \cite{Scott2013,toshniwal2017smooth,casquero2020seamless}, the subset of AST-splines requires that no EP belongs to the 3-disk vertices of any other EP. In this paper, we remove this condition; namely, we allow even all the vertices within a face to be EPs. In addition, we also allow EPs to be one layer closer to the T-mesh boundary in comparison with previous works \cite{Scott2013,toshniwal2017smooth,casquero2020seamless}. Generalizing AST-splines in this manner drastically increases the flexibility of building geometries using AST-splines.


\begin{figure} [t!] 
 \centering
 \subfigure[Knot spans]{\includegraphics[scale=.4]{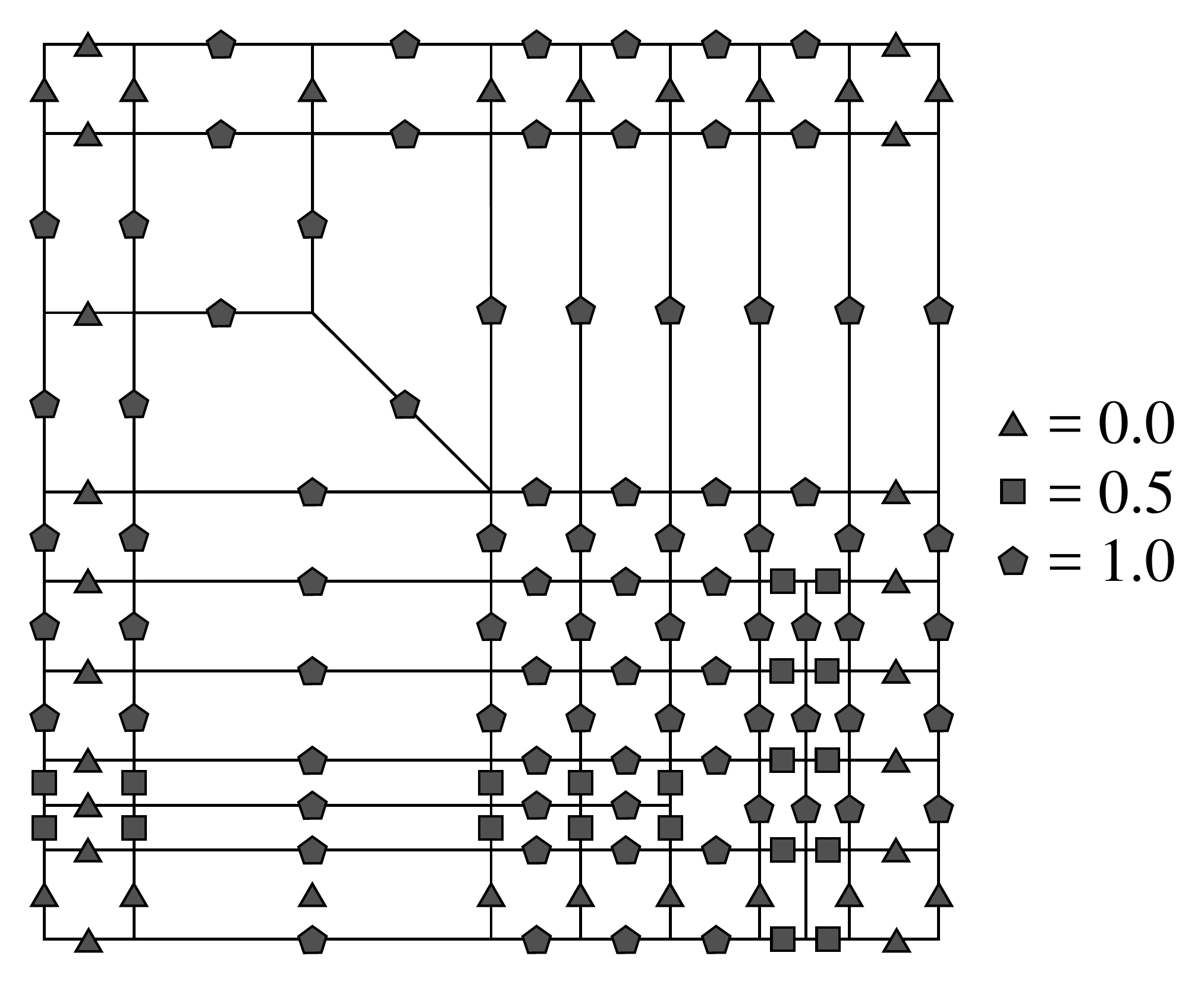}} \hspace*{+6mm}
 \subfigure[Elemental T-mesh]{\includegraphics[scale=.5]{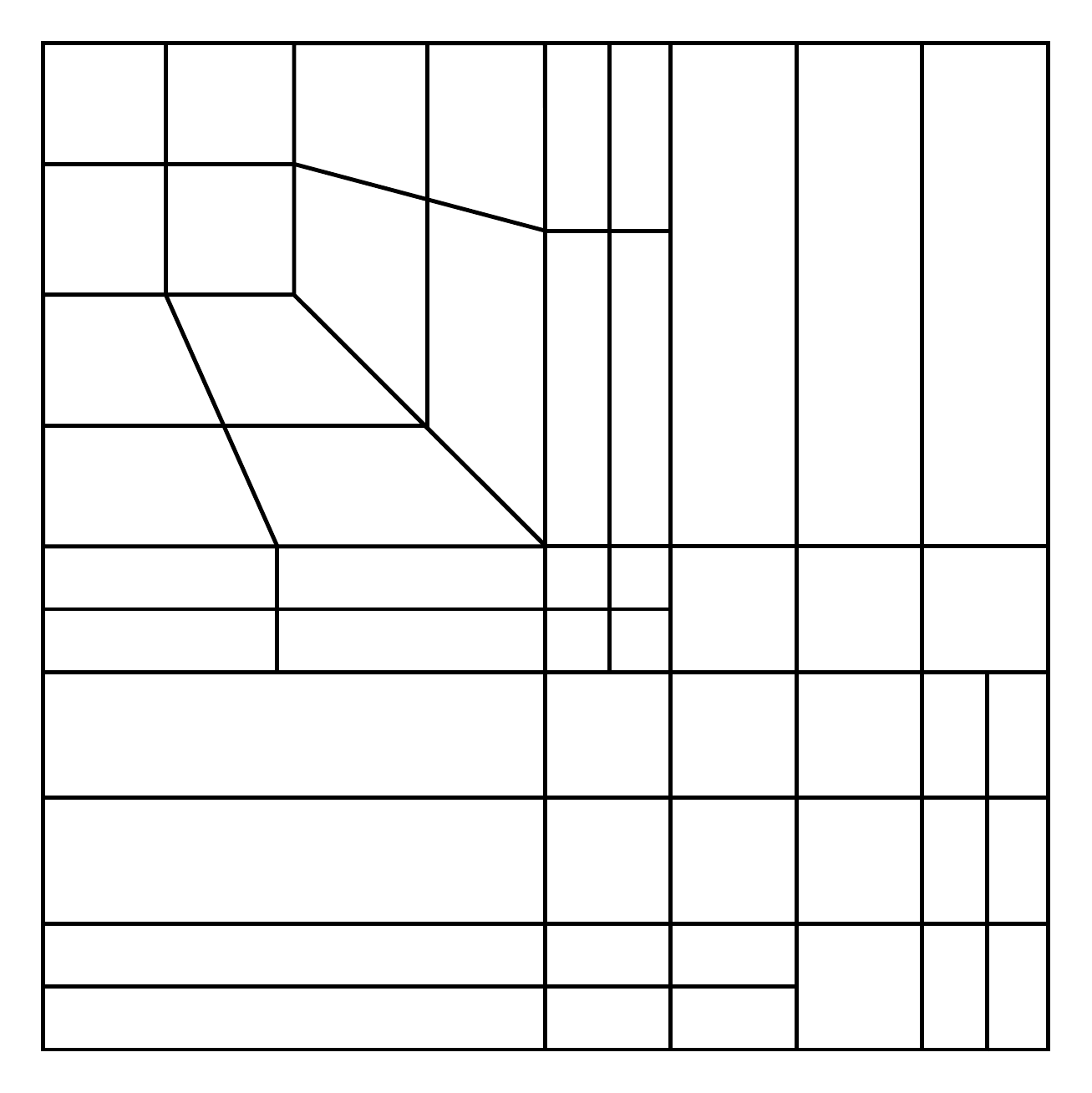}} \\
\caption{(a) A possible knot span configuration for the T-mesh represented in Fig. \ref{tmesh} (a). The pentagons, squares, and triangles correspond to knot spans with values 1, 1/2, and 0, respectively. (b) Elemental T-mesh associated with the T-mesh and the knot span configuration represented in Fig. \ref{tmesh} (a) and Fig. \ref{knotspans} (a), respectively.}
\label{knotspans}
\end{figure}

\subsection{Knot spans}

\textit{Knot spans} determine the parameterization of a T-spline surface. Each T-mesh edge has a knot span assigned. Knot spans are non-negative real numbers. A knot span configuration is valid when it satisfies the following conditions:

\begin{itemize}
\item Knot spans on opposite sides of every face are required to sum to the same value.
\item All the edges that emanate from the T-mesh boundary are assigned with zero knot spans.
\end{itemize}

A possible knot span configuration for our T-mesh example is plotted in Fig. \ref{knotspans} (a). In this work, the spoke edges of a given EP will be assigned with the same nonzero knot span.

The \textit{elemental T-mesh} determines the elements of a T-spline surface, i.e., the regions in which all basis functions are $C^\infty$. Starting from the T-mesh faces, the elements of the elemental T-mesh are obtained by modifying the T-mesh as follows:


\begin{itemize}
\item Adding the face extension to each T-junction, which subdivides faces into two elements.
\item Subdividing each face within the 1-ring faces of an EP into four elements.
\item Eliminating faces with zero parametric measure.
\end{itemize}

 Fig. \ref{knotspans} (b) plots the elemental T-mesh associated with the T-mesh and the knot span configuration shown in Fig. \ref{tmesh} (a) and Fig. \ref{knotspans} (a), respectively. The number of elements in the elemental T-mesh is denoted by $n_{el}$.

\subsection{B\'ezier extraction}

Since basis functions within each element of the elemental T-mesh are bi-cubic polynomials, a basis function $N_A$ restricted to an element $e$ can be represented as a linear combination of the 16 bi-cubic tensor-product Bernstein polynomials, viz.,

\begin{equation}    
N_A|_e = N^e_a \left( \mathbf{\xi} \right) = \sum_{j=1}^{16} C^e_{aj} b_j \left( \mathbf{\xi} \right),  \quad   \mathbf{\xi}   \in   \square  \text{,}
\end{equation}

\noindent where $\square$ is the parent element domain, $b_j$ is the $j$-th Bernstein polynomial, $A$ is a global basis function index, and $a$ is a local-to-element basis function index. Following \cite{Hughes2012}, we use the array $\text{IEN}$ to establish a correspondence between local and global numbering of basis functions, namely, $A = \text{IEN}(a,e)$. 

Collecting all the basis functions with support on element $e$ and the 16 Bernstein polynomials in column vectors $\mathbf{N}^e= ( N_1^e, N_2^e, ..., N_{n^e}^e )^T$ and $\mathbf{b}= ( b_1, b_2, ..., b_{16})^T$, respectively, the \textit{spline extraction operator} $\mathbf{C}^e$ is a matrix of dimension $n^e \times 16$ that relates the basis functions with the Bernstein polynomials as follows
\begin{equation}
\mathbf{N}^e \left(  \mathbf{\xi} \right)=\mathbf{C}^e \mathbf{b} \left(  \mathbf{\xi} \right),\quad\vec\xi\in\square  \text{,}
\end{equation}
where $n^e$ is the number of spline basis functions with support on element $e$. Within the subset of AST-splines, for elements affected by T-junctions, $n^e = 16$ as for the case of NURBS, but for elements affected by EPs, $n^e$ may be higher than 16. 

In an analogous way to how Bernstein polynomials can be related to spline basis functions, spline control points $\mathbf{P}^e=( \mathbf{P}_1^e, \mathbf{P}_2^e, ..., \mathbf{P}_{n^e}^e )^T$ can be related to B\'ezier control points $\mathbf{B}^e= ( \mathbf{B}^e_1, \mathbf{B}^e_2, ..., \mathbf{B}^e_{16})^T$ as follows
\begin{equation}\label{bextractioncoef}
\mathbf{B}^e  =  \left( \mathbf{C}^e \right)^T   \mathbf{P}^e  \text{,}
\end{equation}
where $\mathbf{P}^e$ and $\mathbf{B}^e$ are matrices of dimension $n^e \times 3$ and $16 \times 3$, respectively. $\mathbf{E}^e = (\mathbf{C}^e)^T$ is the \textit{B\'ezier extraction operator}.

\subsection{Basis functions}

In order to define basis functions, we classify the faces and vertices of the T-mesh as follows:

\begin{itemize}
\item \textit{Irregular faces} are the 1-ring faces of EPs. \textit{Transition faces} are the 2-ring faces of EPs. The remaining faces are \textit{regular faces}.

\item \textit{Irregular vertices} are the 0-ring vertices of EPs. \textit{Transition vertices} are the 1-ring vertices of EPs. The remaining vertices are \textit{regular vertices}. 
\end{itemize}

As proposed in \cite{toshniwal2017smooth}, two different T-spline spaces are defined:

\begin{itemize}
\item A design space ($\mathbb{S}^1_D$) for CAD.  In $\mathbb{S}^1_D$, a basis function is assigned to each vertex. The spline basis functions associated with regular vertices are globally $C^2$-continuous while the spline basis functions associated with irregular and transition vertices are globally $C^1$-continuous. From now on, the basis functions of $\mathbb{S}^1_D$ will be denoted by $N_L$, where $L \in \{ 1,...,n\}$ and $n$ is the number of T-mesh vertices.

\item An analysis space ($\mathbb{S}^1_A$) for CAE.  A basis function is assigned to each vertex that is not an irregular vertex or a transition vertex whose 1-ring faces are either irregular faces or transition faces with zero parametric measure. The spline basis functions associated with regular vertices are globally $C^2$-continuous while the spline basis functions associated with transition vertices are globally $C^1$-continuous. Four basis functions are assigned to each irregular face. The face-based basis functions are globally $C^1$-continuous. From now on, the vertex-based basis functions of $\mathbb{S}^1_A$, the face-based basis functions of $\mathbb{S}^1_A$, and all the basis functions of  $\mathbb{S}^1_A$ will be denoted by $\hat{M}_V$, $\tilde{M}_F$, and $M_B$, respectively, where $F \in \{ 1,...,n-n_{ep}-n_{ti}\}$, $V \in \{ 1,...,\sum_{j=1}^{n_{ep}} 4\mu_j \}$, $B \in \{ 1,..., n_{b}\}$, $n_{ep}$ is the number of EPs, $n_{ti}$ is the number of transition vertices whose 1-ring faces are either irregular faces or transition faces with zero parametric measure, $\mu_j$ is the valence of the $j$th EP, and $n_{b} = n - n_{ep} -n_{ti} + \sum_{j=1}^{n_{ep}} 4\mu_j$ is the number of basis functions in $\mathbb{S}_A^1$.
\end{itemize}

The spaces $\mathbb{S}^1_D$ and $\mathbb{S}^1_A$ are constructed in such a way that $\mathbb{S}^1_D \subseteq \mathbb{S}^1_A$. Both $\mathbb{S}^1_D$ and $\mathbb{S}^1_A$ satisfy properties (1)-(6). In addition,  $\mathbb{S}^1_A$ satisfies property (7).

The basis functions are specified by the extraction operators defined in each element of the elemental T-mesh. The extraction operators of elements that are within regular faces are the same for $\mathbb{S}^1_D$ and $\mathbb{S}^1_A$, but for irregular and transition faces are different. Since we have not changed how we deal with T-junctions in regular faces, we refer to our previous work for details \cite{casquero2020seamless}. We focus here on how to deal with irregular and transition faces when multiple EPs per face are allowed.

\begin{figure} [t!] 
 \centering
 \subfigure[]{\includegraphics[scale=0.65]{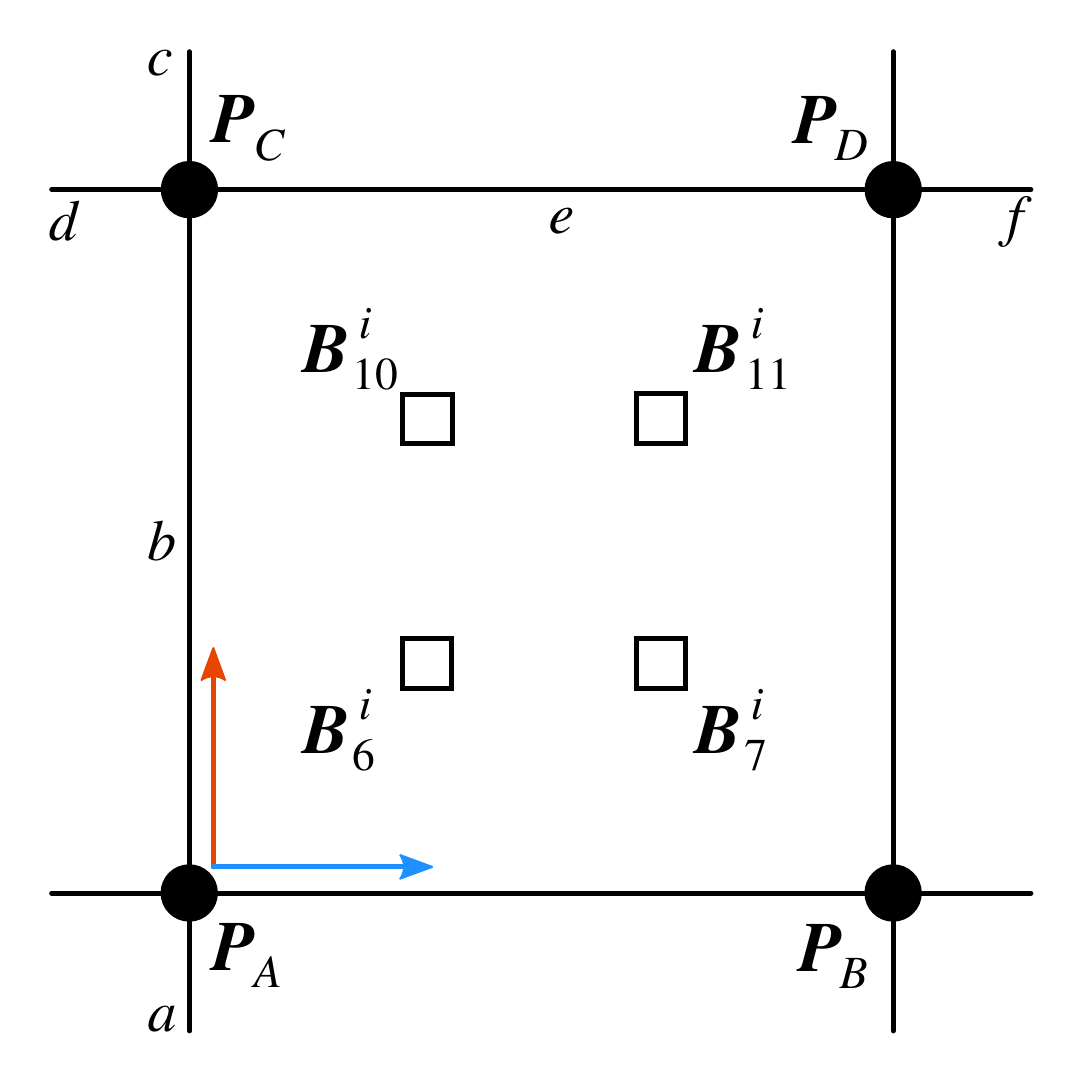}}
 \subfigure[]{\includegraphics[scale=0.65]{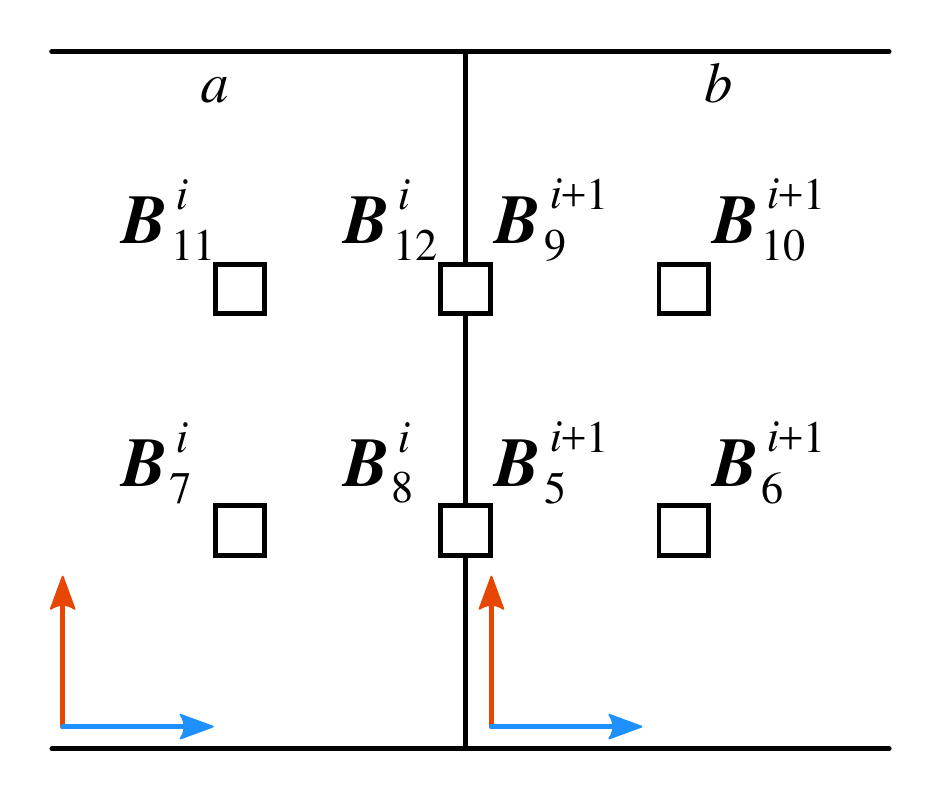}} \\
\caption{(a) Face B\'ezier control points are defined in terms of spline control points using Eqs. \eqref{firsteq}-\eqref{4theq}. (b) Edge B\'ezier control points that are not located at the boundary of the elemental T-mesh are defined in terms of adjacent face B\'ezier control points using Eqs. \eqref{5theq}-\eqref{6theq}.}
\label{cpunstructured}
\end{figure}

\begin{figure} [t!] 
 \centering
 \subfigure[]{\includegraphics[scale=0.65]{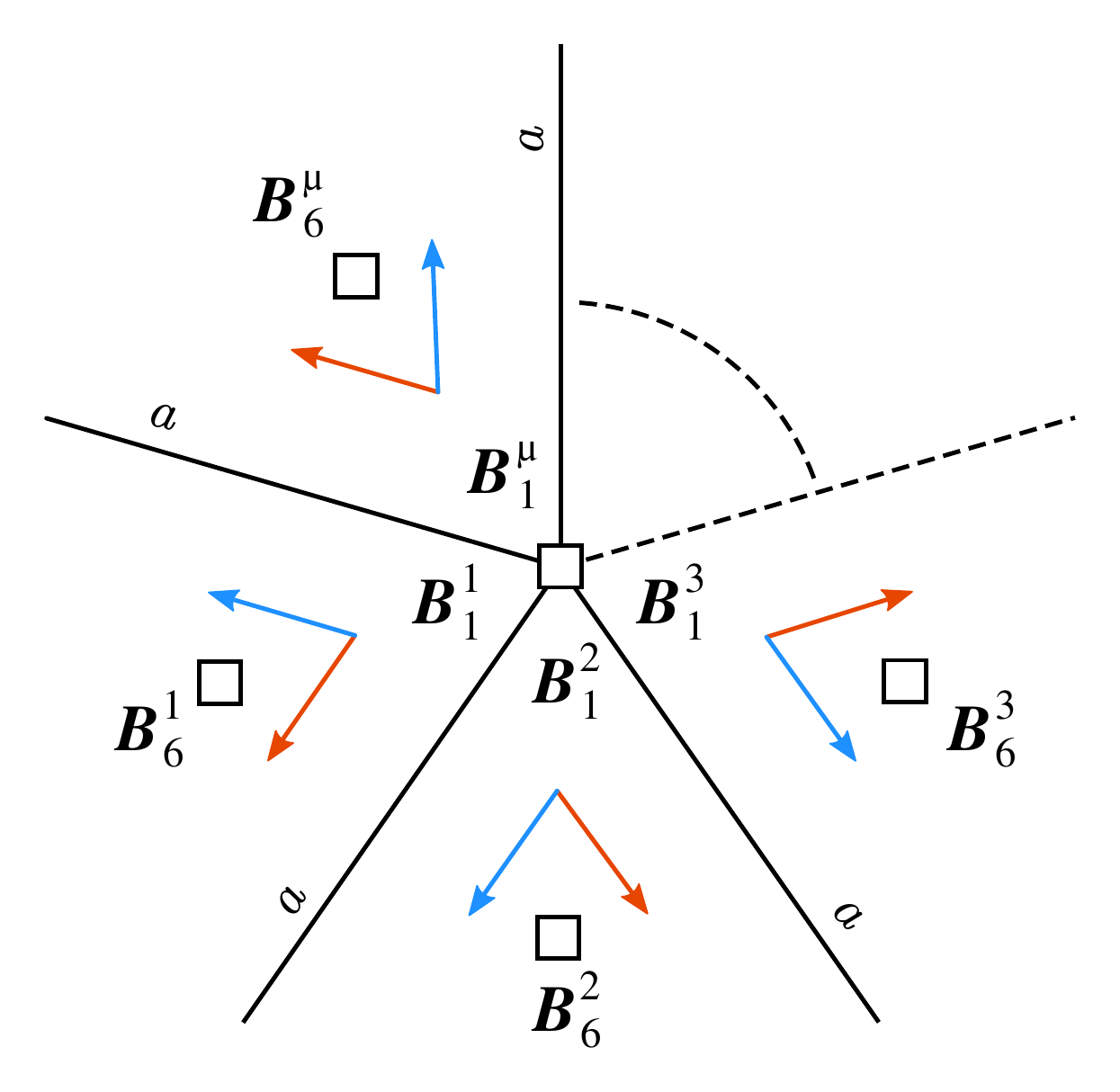}} 
 \subfigure[]{\includegraphics[scale=0.65]{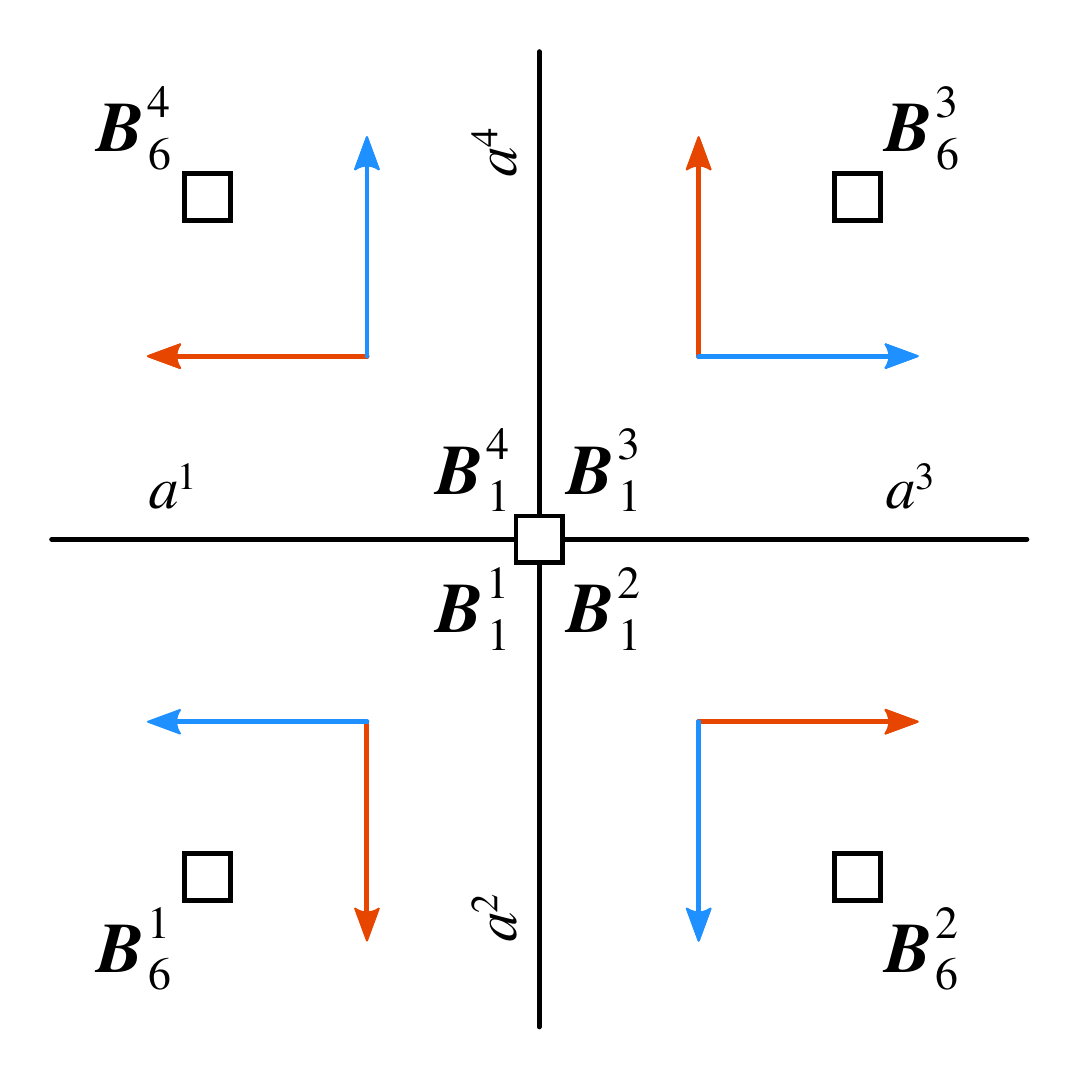}}\\
\caption{(a)-(b) Vertex B\'ezier control points that are not located at the boundary of the elemental T-mesh are defined in terms of adjacent face B\'ezier control points using Eqs. \eqref{7theq}-\eqref{lasteq}.}
\label{cpunstructured2}
\end{figure}

\subsubsection{Irregular and transition faces in design}

\begin{figure} [t!] 
 \centering
 \subfigure[]{\includegraphics[scale=0.53]{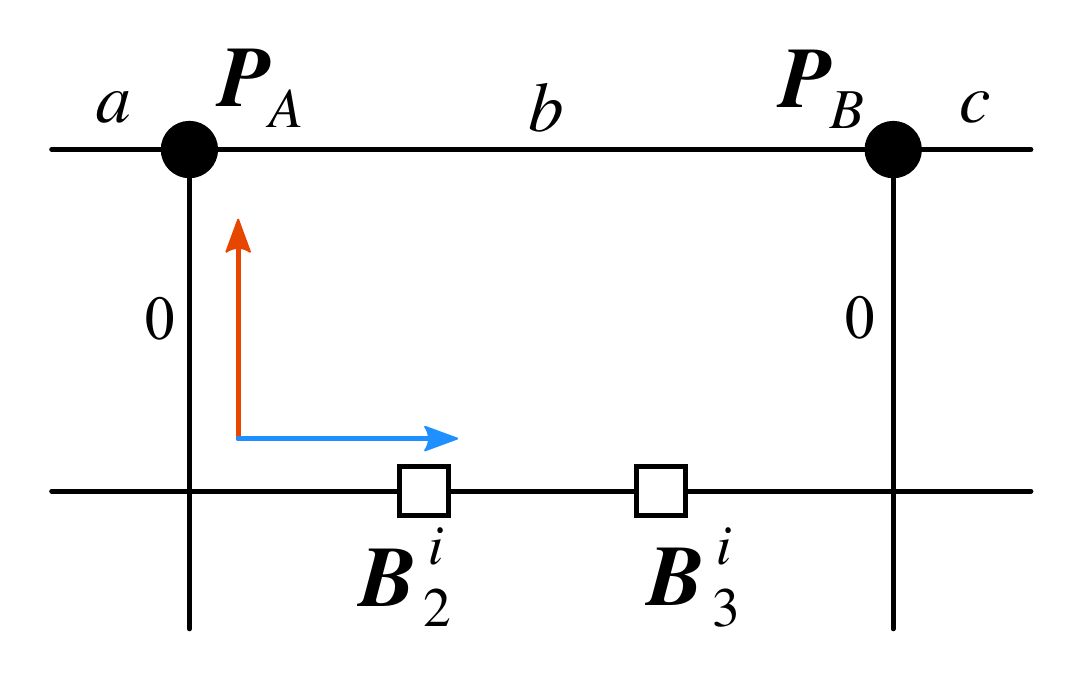}} \hspace*{+2.0mm}
 \subfigure[]{\includegraphics[scale=0.53]{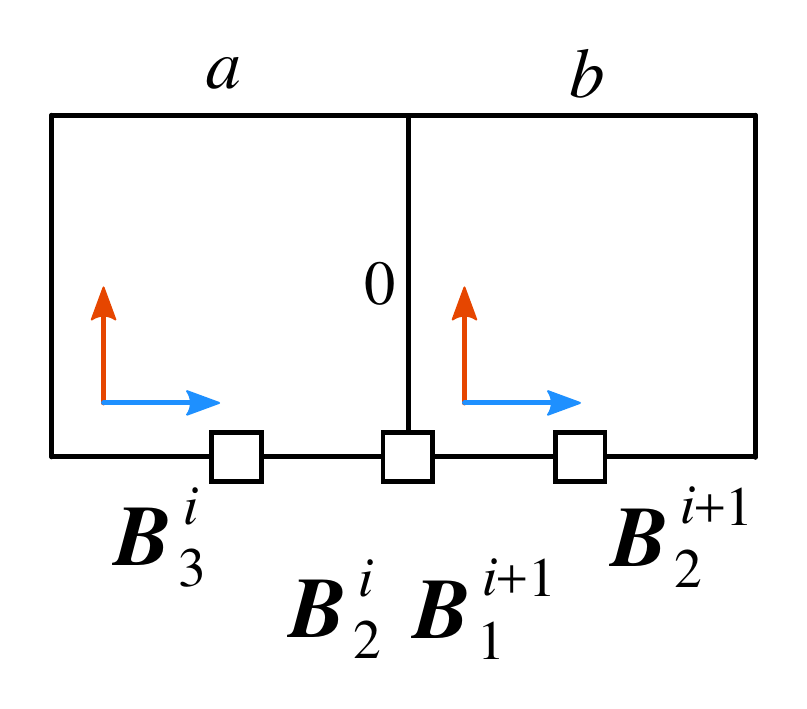}}  \hspace*{+2.0mm}
  \subfigure[]{\includegraphics[scale=0.53]{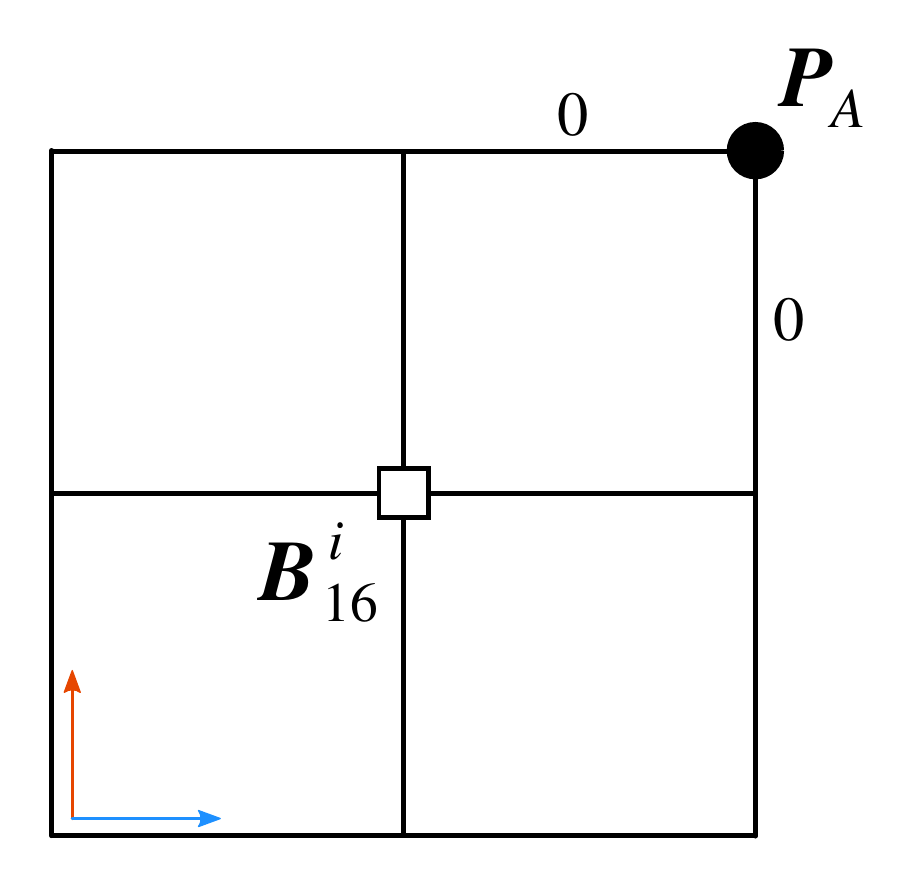}} \\
\caption{(a) Edge B\'ezier control points that are located at the boundary of the elemental T-mesh are defined in terms of adjacent spline control points using Eqs. \eqref{9theq}-\eqref{10theq}. (b) Vertex B\'ezier control points that are located at the boundary of the elemental T-mesh and not placed at a corner of the elemental T-mesh are defined in terms of adjacent edge B\'ezier control points using Eq. \eqref{11theq}. (c) Vertex B\'ezier control points that are located at the boundary of the elemental T-mesh and placed at the corners of the elemental T-mesh are equal to the spline control points placed at the corners of the control net.}
\label{cpunstructured3}
\end{figure}

We begin associating only one element to each face and classify B\'ezier control points in face, edge, and vertex B\'ezier control points. Following \cite{Scott2013, toshniwal2017smooth}, face B\'ezier control points are initially defined in terms of spline control points as

\begin{equation}   \label{firsteq} 
\mathbf{B}_6^i = \frac{ (b+c)(e+f) \mathbf{P}_A + (b+c) d \mathbf{P}_B + a (e+f) \mathbf{P}_C + a d \mathbf{P}_D }{(a+b+c)(d+e+f)} \text{,} 
\end{equation}

\begin{equation}    
\mathbf{B}_7^i = \frac{ (b+c)f \mathbf{P}_A + (b+c) (d+e) \mathbf{P}_B + a f \mathbf{P}_C + a (d+e) \mathbf{P}_D }{(a+b+c)(d+e+f)} \text{,} 
\end{equation}

\begin{equation}    
\mathbf{B}_{10}^i = \frac{  c(e+f) \mathbf{P}_A + c d \mathbf{P}_B + (a +b) (e+f) \mathbf{P}_C + (a +b) d \mathbf{P}_D }{(a+b+c)(d+e+f)} \text{,} 
\end{equation}

\begin{equation}     \label{4theq} 
\mathbf{B}_{11}^i = \frac{ c f \mathbf{P}_A + c ( d + e) \mathbf{P}_B + (a + b) f \mathbf{P}_C + ( a + b) ( d +e ) \mathbf{P}_D }{(a+b+c)(d+e+f)} \text{.} 
\end{equation}

\noindent Edge B\'ezier control points that are not located at the boundary of the elemental T-mesh are defined in terms of adjacent face B\'ezier control points as
\begin{figure} [t!] 
 \centering
 \subfigure[Before $2 \times 2$ split]{\includegraphics[scale=0.58]{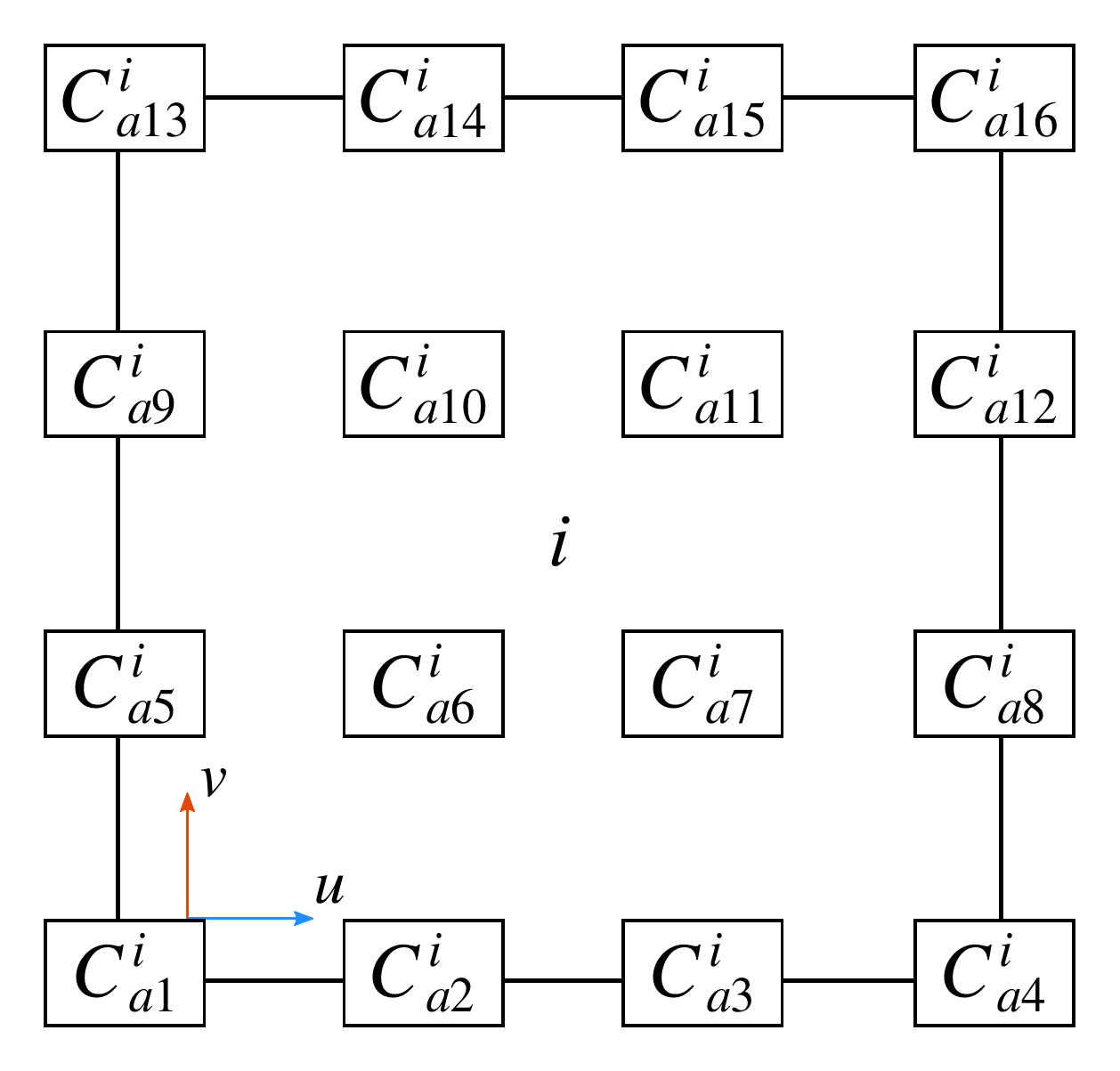}} \hspace*{+6mm}
 \subfigure[After $2 \times 2$ split]{\includegraphics[scale=0.60]{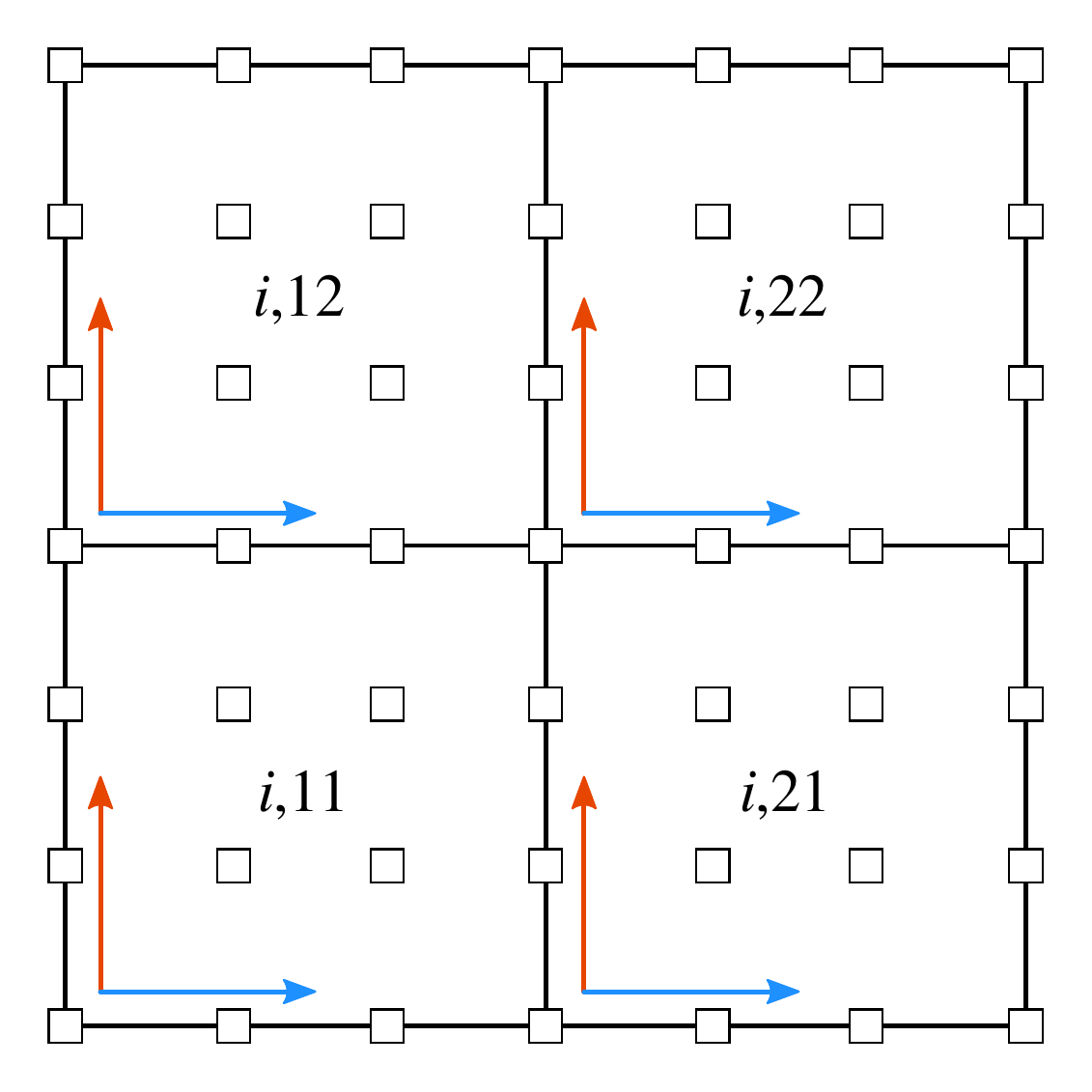}} \\
\caption{(a) Extraction coefficients and local coordinate system before the $2 \times 2$ split. (b) Extraction coefficients and local coordinate system after the $2 \times 2$ split.}
\label{bezier}
\end{figure}

\begin{equation}    \label{5theq}
\mathbf{B}_{8}^i = \mathbf{B}_{5}^{i+1} = \frac{b}{a+b} \mathbf{B}_{7}^i + \frac{a}{a+b} \mathbf{B}_{6}^{i+1} \text{,} 
\end{equation}

\begin{equation}    \label{6theq}
\mathbf{B}_{12}^i = \mathbf{B}_{9}^{i+1} = \frac{b}{a+b} \mathbf{B}_{11}^i + \frac{a}{a+b} \mathbf{B}_{10}^{i+1} \text{.} 
\end{equation}

\noindent Vertex B\'ezier control points that are not located at the boundary of the elemental T-mesh are defined in terms of adjacent face B\'ezier control points as

\begin{equation}   \label{7theq} 
\mathbf{B}_{1}^1 = \mathbf{B}_{1}^2 = ... = \mathbf{B}_{1}^{\mu} = \frac{1}{\mu} \sum_{j=1}^{\mu} \mathbf{B}_{6}^{j} \quad \text{if} \quad \mu \neq 4 \text{,} 
\end{equation}

\begin{equation}   \label{lasteq} 
\mathbf{B}_{1}^1 = \mathbf{B}_{1}^2 = ... = \mathbf{B}_{1}^{\mu} = \sum_{j=1}^{\mu} \frac{a^{j+2} a^{j-1}}{(a^{j+2} + a^{j}) (a^{j-1} + a^{j+1}) }\mathbf{B}_{6}^{j} \quad \text{if} \quad \mu = 4 \text{.} 
\end{equation}

\noindent The labels used in Eqs. \eqref{firsteq}-\eqref{lasteq} correspond to Figs. \ref{cpunstructured} and \ref{cpunstructured2}. Following \cite{Wei2018}, edge B\'ezier control points that are located at the boundary of the elemental T-mesh are defined in terms of adjacent spline control points as

\begin{equation}    \label{9theq} 
\mathbf{B}_{2}^i = \frac{b+c}{a+b + c} \mathbf{P}_A + \frac{a}{a+b + c} \mathbf{P}_B \text{,} 
\end{equation}

\begin{equation}   \label{10theq} 
\mathbf{B}_{3}^i =  \frac{c}{a+b+c} \mathbf{P}_A + \frac{a+b}{a+b+c} \mathbf{P}_B \text{.} 
\end{equation}

\noindent Vertex B\'ezier control points that are located at the boundary of the elemental T-mesh and not placed at a corner of the elemental T-mesh are defined in terms of adjacent edge B\'ezier control points as

\begin{equation}   \label{11theq} 
\mathbf{B}_{4}^i = \mathbf{B}_{1}^{i+1} = \frac{b}{a+b} \mathbf{B}_{3}^i + \frac{a}{a+b} \mathbf{B}_{2}^{i+1}  \text{.} 
\end{equation}

\noindent Vertex B\'ezier control points that are located at the boundary of the elemental T-mesh and placed at the corners of the elemental T-mesh are equal to the spline control points placed at the corners of the control net

\begin{equation}   \label{lasteq2} 
\mathbf{B}_{16}^i = \mathbf{P}_A \text{.} 
\end{equation}

\noindent The labels used in Eqs. \eqref{9theq}-\eqref{lasteq2} correspond to Fig. \ref{cpunstructured3}. Eqs. \eqref{firsteq}-\eqref{lasteq2} define initial expressions for the extraction operators of the elements within irregular and transition faces. 

Eqs. \eqref{firsteq}-\eqref{lasteq2} result in edges that are $C^2$-continuous with the exception of spoke edges, which are only $C^0$-continuous. To reach $C^1$ continuity across the spoke edges, we modify the extraction operators of the elements within irregular faces using the split-then-smoothen approach \cite{toshniwal2017smooth, nguyen2016refinable, reif1997refineable}. This approach is applied to each basis function that has support on the 1-ring faces of at least one EP as follows:

\begin{figure} [t!] 
 \centering
 \subfigure[]{\includegraphics[scale=0.62]{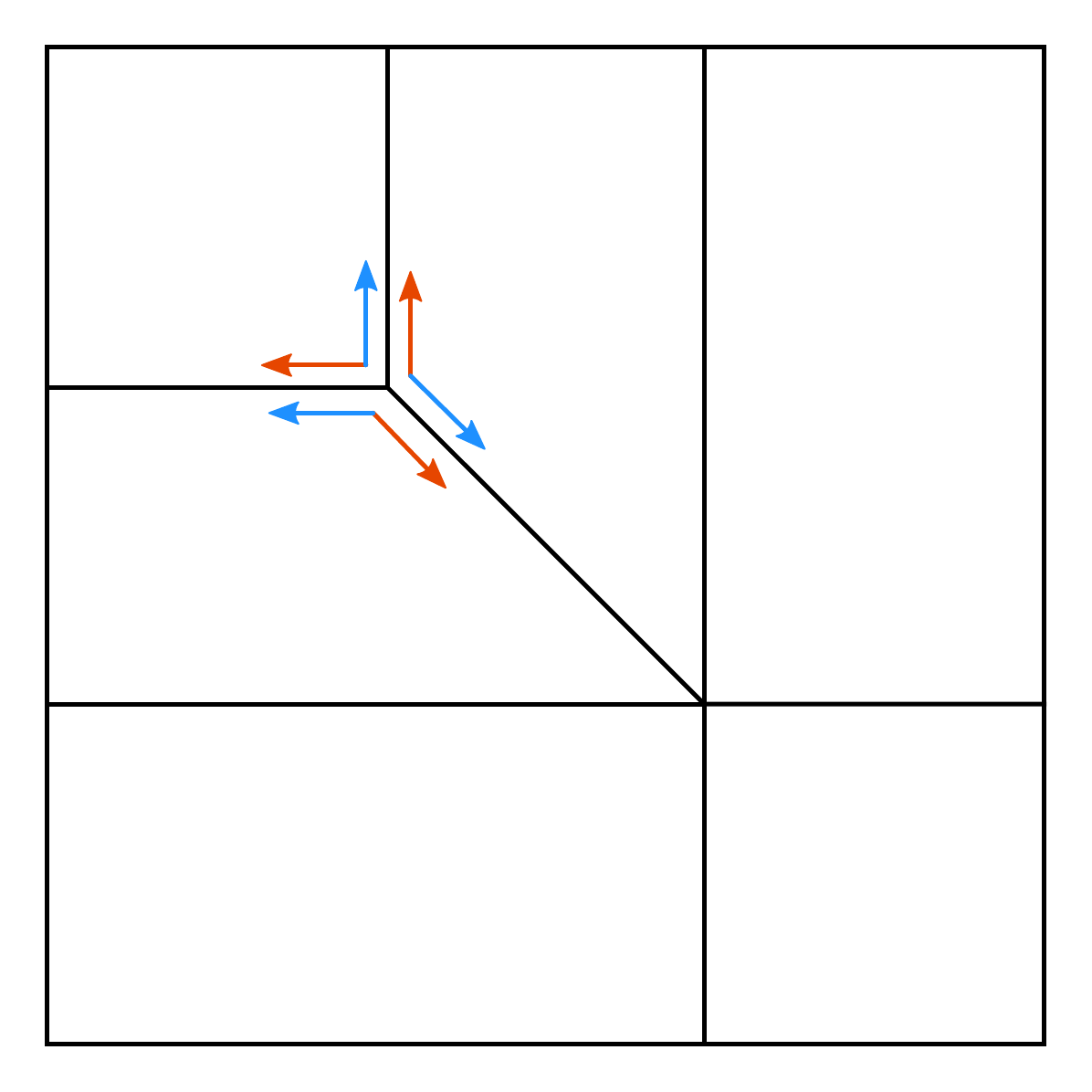}} \hspace*{+3mm}
 \subfigure[]{\includegraphics[scale=0.62]{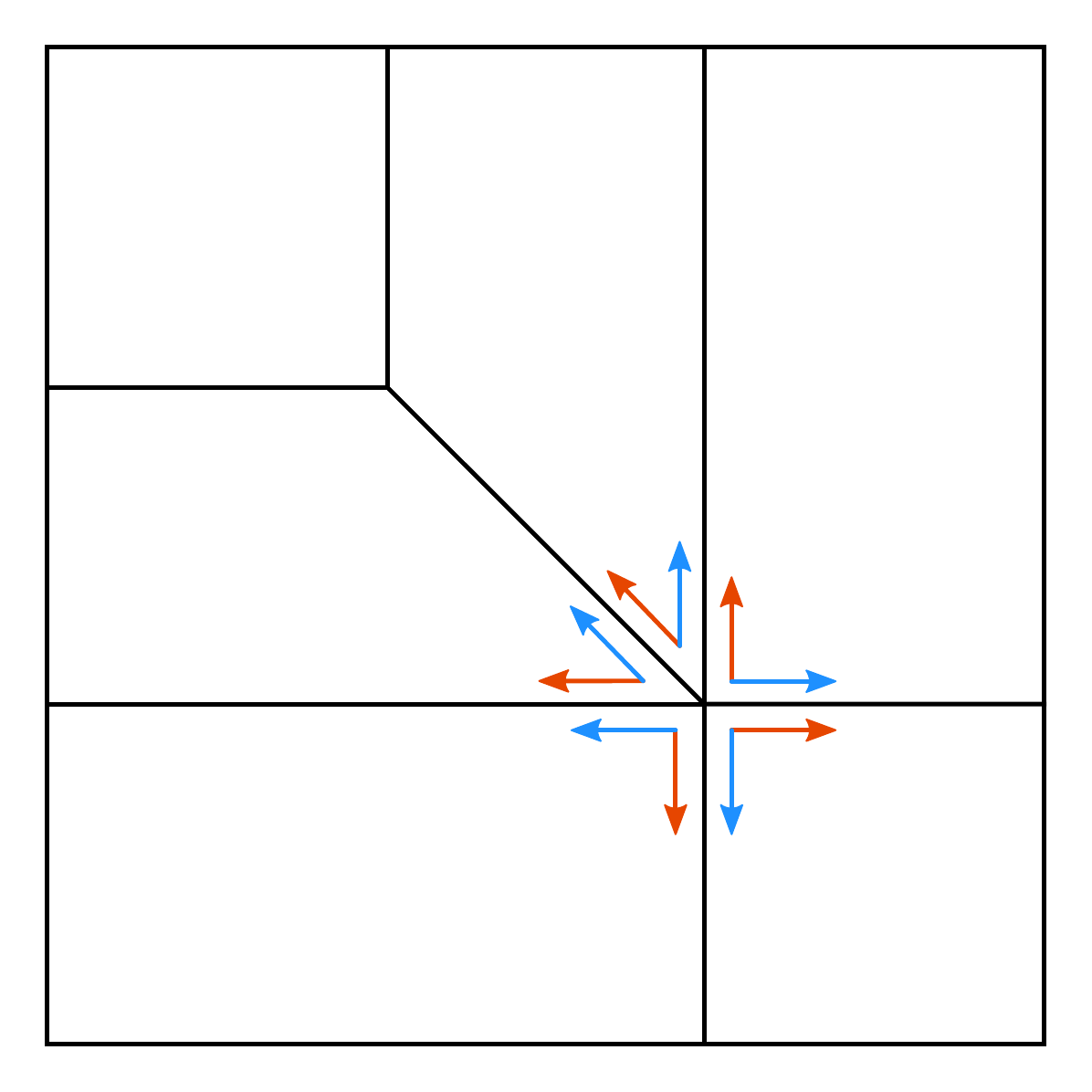}} \\
\caption{ 1-ring faces of two EPs with valences 3 and 5. (a) Local axes for the EP with valence 3. (b) Local axes for the EP with valence 5.}
\label{3ep5ep}
\end{figure}

\begin{figure} [t!] 
 \centering
 \subfigure[]{\includegraphics[scale=0.63]{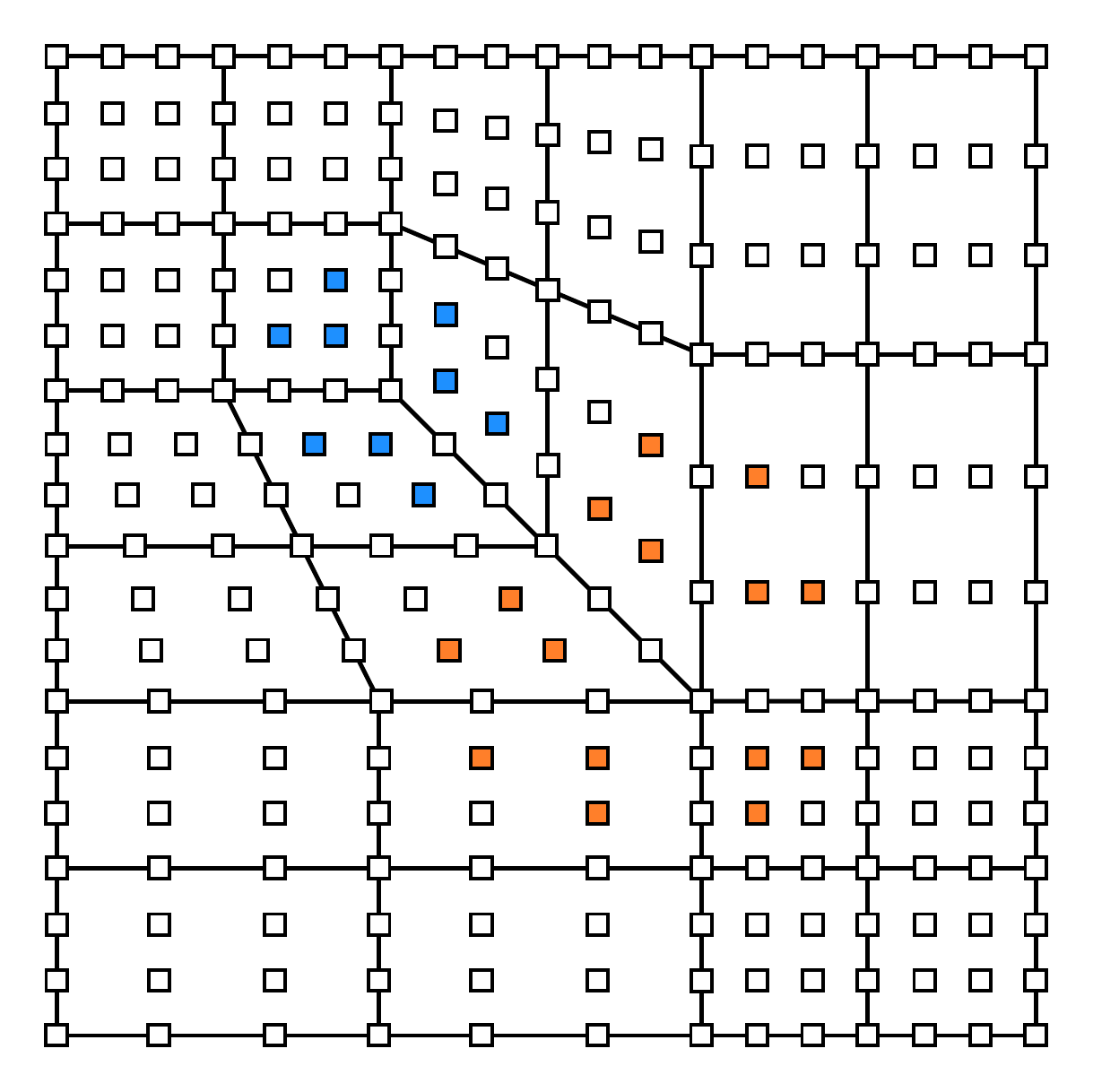}} \hspace*{+3mm}
 \subfigure[]{\includegraphics[scale=0.63]{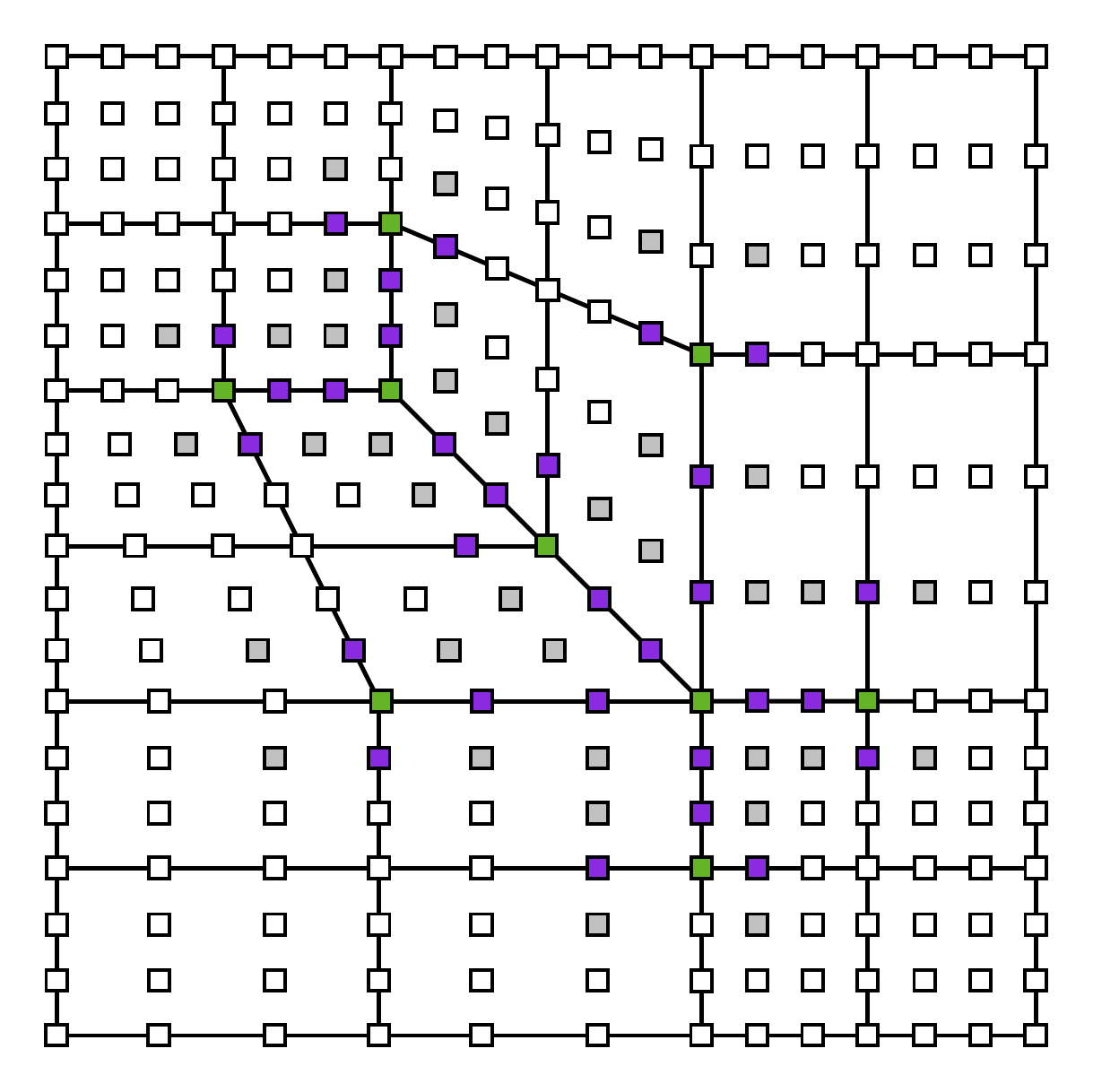}} \\
\caption{(Color online) 1-ring faces of two EPs with valences 3 and 5 after the $2 \times 2$ split. (a) The face extraction coefficients that are changed by the smoothing matrix of the EP with valence 3 and valence 5 are plotted in blue and orange, respectively. (b) The edge and vertex extraction coefficients that need to be recomputed after application of the smoothing matrix are plotted in violet and green, respectively. The face extraction coefficients that intervene in the aforementioned recomputations are plotted in gray.}
\label{sthens}
\end{figure}

\begin{enumerate}[label=(\alph*)]

\item For each irregular face in which the basis function has support, the extraction coefficients that define the basis function on that irregular face are collected in a row vector $\mathbf{C}_a^{i} = (C_{a1}^i, C_{a2}^i, ..., C_{a16}^i)$ following the numbering given by local axes as shown in Fig. \ref{bezier} (a).

\item For each irregular face in which the basis function has support, the basis function is refined at parametric lines $u = a / 2$ and $v = a / 2$ using the de Casteljau algorithm (see Appendix A). Therefore, each irregular face has four elements now and the basis function is defined in each element by 16 extraction coefficients as shown in Fig. \ref{bezier} (b). The extraction coefficients that define the basis function on these four elements, denoted by $\mathbf{C}_{a}^{i,pq}$ with $p,q \in \{1,2\}$, are obtained as follows

\begin{equation}
	\mathbf{C}^{i,pq}_a = \mathbf{C}_a^{i} \left( \mathbf{S}_q \otimes \mathbf{S}_p \right) \text{,}
\end{equation}

\noindent with

\begin{equation}
	\mathbf{S}_1 = 
	\begin{pmatrix}
	1 & \frac{1}{2} & \frac{1}{4} & \frac{1}{8}\\
	0 & \frac{1}{2} & \frac{1}{2} & \frac{3}{8}\\
	0 & 0 & \frac{1}{4} & \frac{3}{8}\\
	0 & 0 & 0 & \frac{1}{8}
	\end{pmatrix} ,\; \;
		\mathbf{S}_2 = 
		\begin{pmatrix}
	\frac{1}{8} & 0 & 0 & 0 \\
	\frac{3}{8} & \frac{1}{4} & 0 & 0 \\
	\frac{3}{8} & \frac{1}{2} & \frac{1}{2} & 0\\
	\frac{1}{8} & \frac{1}{4} & \frac{1}{2} & 1
	\end{pmatrix}
	\text{.}
\end{equation}

\noindent where $\otimes$ represents the Kronecker product of two matrices.

\item For each EP within the support of the basis function, the D-patch framework \cite{reif1997refineable} is used. Assuming we place local axes at each EP (an example with two EPs is shown in Fig. \ref{3ep5ep}), the extraction coefficients $C_{a6}^{i,11}$, $C_{a7}^{i,11}$, and $C_{a10}^{i,11}$ with $i$ being cyclic in $\{1,2,...,\mu_j \}$ are modified using a smoothing matrix $\mathbf{\Pi}^+$ with dimension $3\mu_j \times 3\mu_j$ as follows
\begin{equation}
		\begin{pmatrix}
			\mathbf{a}_{6}\\
			\mathbf{a}_{7}\\
			\mathbf{a}_{10}
		\end{pmatrix} = 
		\mathbf{\Pi}^+ \begin{pmatrix}
		\mathbf{A}_{6}\\
		\mathbf{A}_{7}\\
		\mathbf{A}_{10}
		\end{pmatrix}
			\text{,}
\end{equation}

\noindent with

	\begin{equation}
		\mathbf{A}_{6} = 
		\begin{pmatrix}
			C^{1,11}_{a6}\\
			C^{2,11}_{a6}\\
			\vdots\\
			C^{\mu_j,11}_{a6}
		\end{pmatrix},\; \;
		\mathbf{A}_{7} = 
		\begin{pmatrix}
		C^{1,11}_{a7}\\
		C^{2,11}_{a7}\\
		\vdots\\
		C^{\mu_j,11}_{a7}
		\end{pmatrix},\; \;
		\mathbf{A}_{10} = 
		\begin{pmatrix}
		C^{1,11}_{a10}\\
		C^{2,11}_{a10}\\
		\vdots\\
		C^{\mu_j,11}_{a10}
		\end{pmatrix}\text{.}
	\end{equation}
	
	\begin{equation}
\mathbf{\Pi}^+ = 
\begin{pmatrix}
\mathbf{\Pi}^+_1 & \mathbf{\Pi}^+_2 & \mathbf{\Pi}^+_3 \\
\mathbf{\Pi}^+_4 & \mathbf{\Pi}^+_5 & \mathbf{\Pi}^+_6 \\
\mathbf{\Pi}^+_7 & \mathbf{\Pi}^+_8 & \mathbf{\Pi}^+_9
\end{pmatrix},
\end{equation}
\begin{equation}
(\mathbf{\Pi}_I^+)_{JK} = (\mathbf{p}_{I})_{ \text{mod}(J-K,\mu_j)}  \text{,}
\end{equation}
	\begin{equation}
	\begin{aligned}
	&(\mathbf{p}_1)_J = (\mathbf{p}_4)_J = (\mathbf{p}_7)_J = 0 \text{,}\\
	&(\mathbf{p}_2)_J = (\mathbf{p}_3)_J = \frac{1}{2\mu_j} \text{,}\\
	&(\mathbf{p}_5)_J = (\mathbf{p}_9)_J = \frac{1}{2\mu_j}\left(1 + \cos(J\phi_\mu)\right) \text{,}\\
	&(\mathbf{p}_6)_J = \frac{1}{2\mu_j}\left(1 + \cos(2\psi + J\phi_\mu)\right) \text{,}\\
	&(\mathbf{p}_8)_J = \frac{1}{2\mu_j}\left(1 + \cos(2\psi - J\phi_\mu)\right) \text{,}
	\end{aligned}
	\end{equation}
	where $I \in \{1,2,...,9\}$, $J,K \in \{0,1,...,\mu_j - 1\}$, $\mathbf{\Pi}^+_I$ is a circulant matrix with dimension $\mu_j \times \mu_j$, $\mathbf{p}_{I}$ is the vector of length $\mu_j$ that defines the circulant matrix $\mathbf{\Pi}^+_I$, $\text{mod}(a,b)$ returns the remainder after division of $a$ by $b$, $\phi_\mu = 2\pi/\mu_j$, $\psi = \arg\left((1+\iota \beta \sin(\phi_\mu))e^{-\iota\phi_\mu/2}\right)$, $\iota = \sqrt{-1}$, and we choose $\beta = 0.4$. The extraction coefficients $C_{a6}^{i,11}$, $C_{a7}^{i,11}$, and $C_{a10}^{i,11}$ with $ i \in \{1,2,...,\mu_j \}$ are replaced by their modified values obtained from the column vectors $\mathbf{a}_6$, $\mathbf{a}_7$, and $\mathbf{a}_{10}$, respectively. For the example shown in Fig. \ref{3ep5ep}, Fig. \ref{sthens} (a) colors the face extraction coefficients that are changed by the smoothing matrix of the EP with valence 3 and valence 5 in blue and orange, respectively. Note that the face extraction coefficients being modified by the presence of an EP will not be affected by any other EP.

\begin{figure} [t!] 
 \centering
 \subfigure[$\mathbb{S}_D^1$]{\includegraphics[scale=1.6]{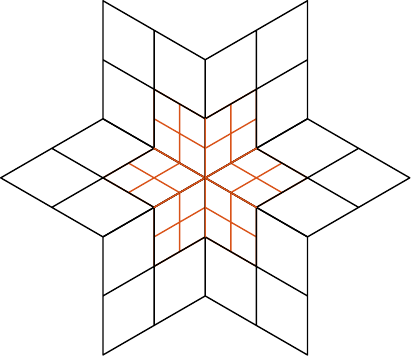}} \hspace*{+6mm}
 \subfigure[$\mathbb{S}_A^1$]{\includegraphics[scale=1.6]{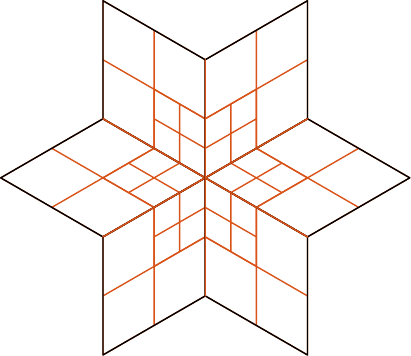}} \\
\caption{(Color online) (a)-(b) Continuity of the basis at element boundaries in the 2-disk faces of an EP with valence 6 for $\mathbb{S}_D^1$ and $\mathbb{S}_A^1$, respectively. Black and orange lines represent $C^2$ and $C^1$ continuity lines, respectively.}
\label{continuity}
\end{figure}

\item As a consequence of the face extraction coefficients that have been changed by smoothing matrices, Eqs. \eqref{5theq}-\eqref{lasteq} need to be imposed again. In other words, certain edge and vertex extraction coefficients need to be recomputed. For the example shown in Fig. \ref{3ep5ep}, Fig. \ref{sthens} (b) colors the edge extraction coefficients that need to be recomputed, the vertex extraction coefficients that need to be recomputed, and the face extraction coefficients that intervene in the recomputation in violet, green, and gray, respectively.
\end{enumerate}

For an EP with valence 6, the final continuity across each element boundary in $\mathbb{S}^1_D$ is specified in Fig. \ref{continuity} (a).

\subsubsection{Irregular and transition faces in analysis}

\begin{figure} [t!] 
 \centering
 \includegraphics[scale=1.1]{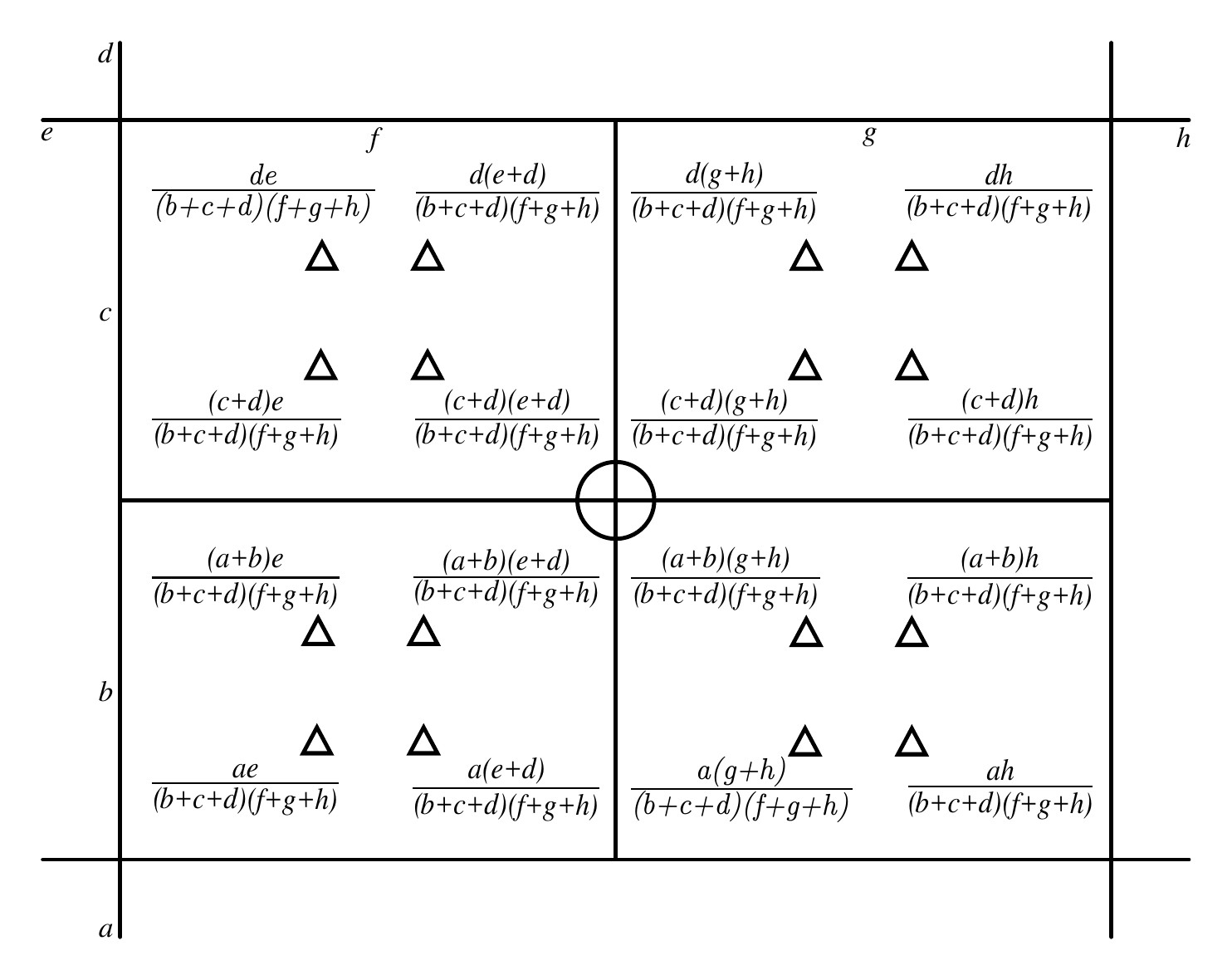} \\
\caption{The children $m_i$ of a transition function $\hat{M}$ are indicated with triangles along with the coefficients $c_i$ needed to obtain the transition function $\hat{M}$ as a linear combination of $m_i$.}
\label{analysisspacemod}
\end{figure}

\begin{figure} [t!] 
 \centering
 \subfigure[]{\includegraphics[scale=0.6]{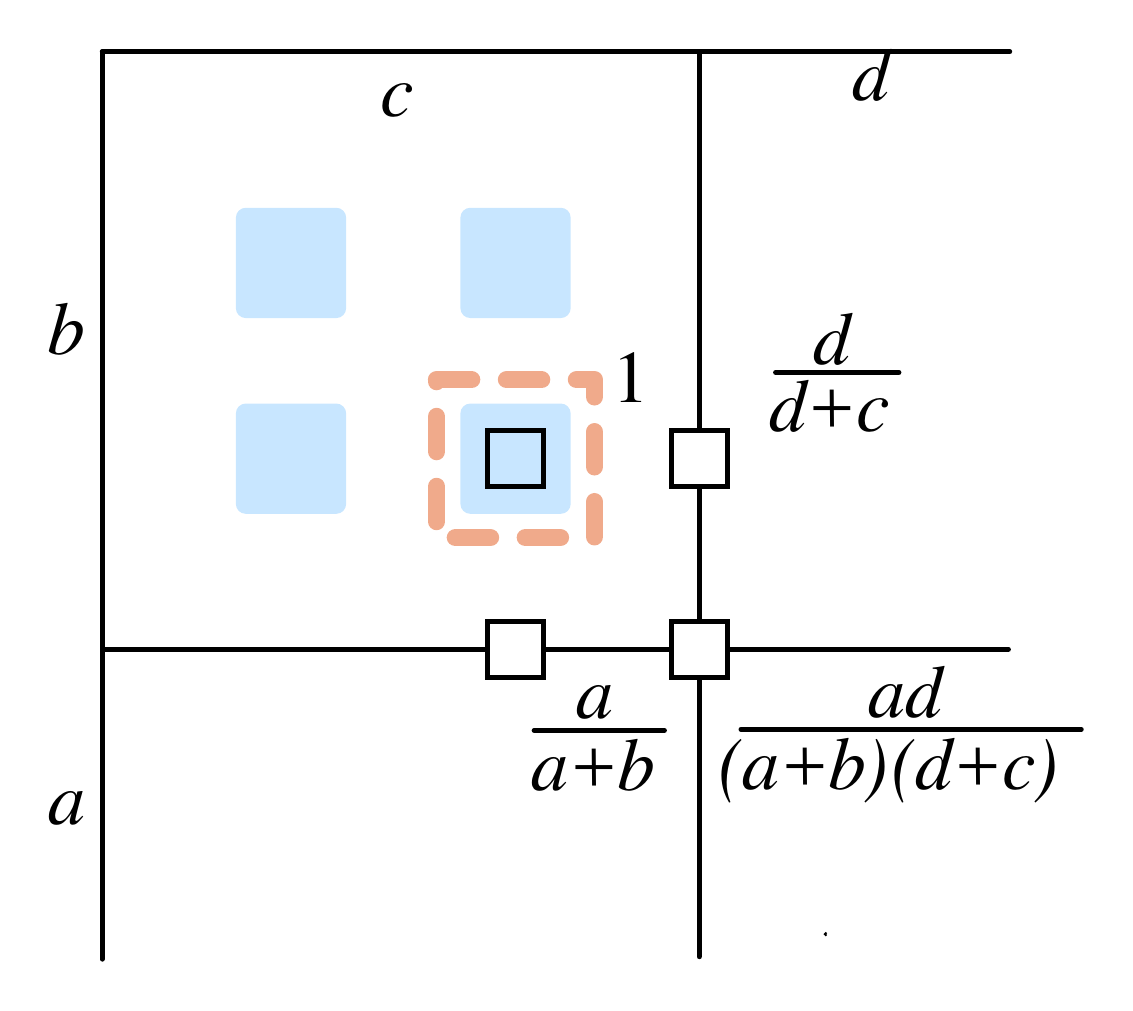}}
 \subfigure[]{\includegraphics[scale=0.6]{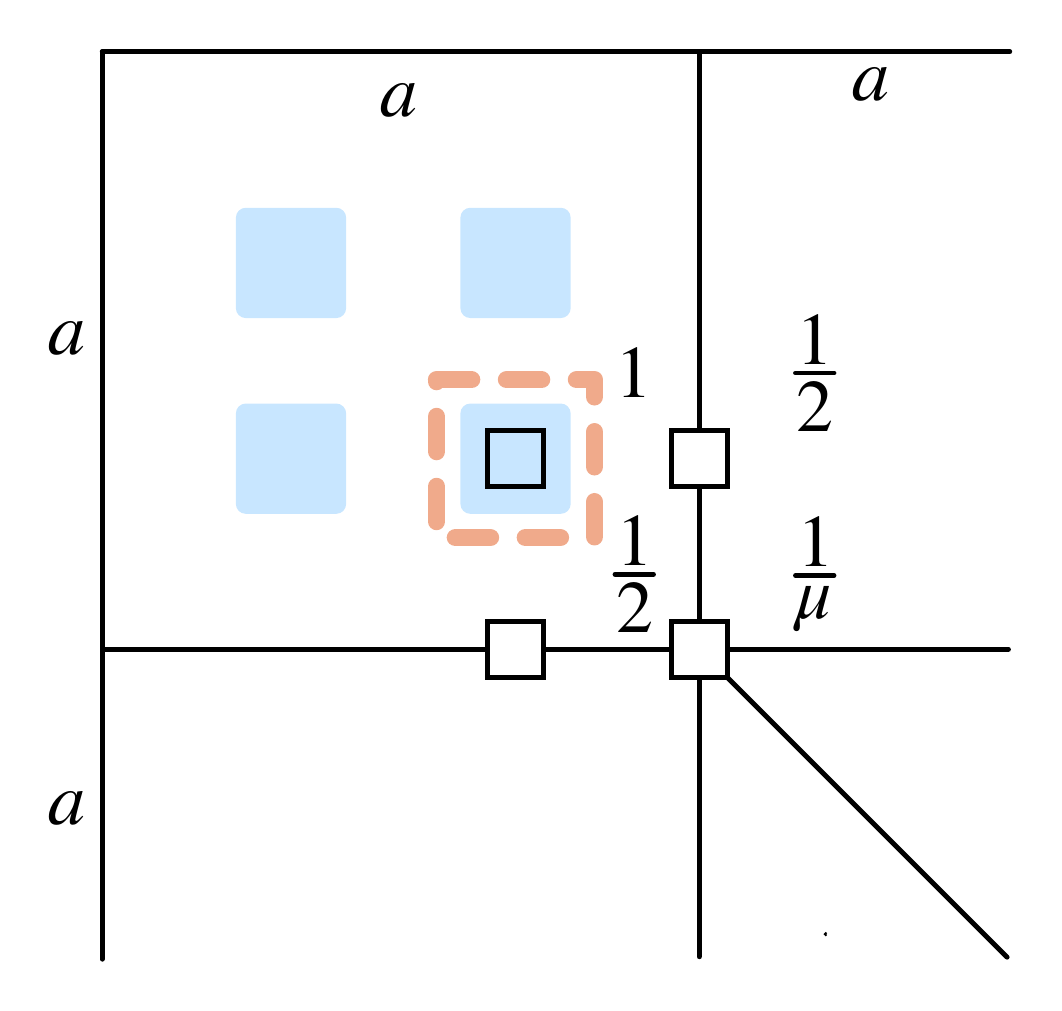}} \\
\caption{The four nonzero extraction coefficients associated with each face-based spline basis function are indicated. (a) Coefficients when the vertex is not an EP. (b) Coefficients when the vertex is an EP.}
\label{analysisspacemod2}
\end{figure}

Due to the fact that both vertex-based and face-based basis functions are present in analysis space, \textit{truncation} \cite{Giannelli2012,Wei2015, Wei2016,Wei2017b,Wei2018} is used to recover partition of unity. The basis functions that need truncation are the transition basis functions. A transition basis function $\hat{M}$ can be obtained as a linear combination of 16 $C^1$-continuous splines $m_i$ as follows
\begin{equation} \label{trunaa}
\hat{M}= \sum_{i = 1}^{16} c_{i} m_i \text{.}
\end{equation}
The functions $m_i$ are called the children of $\hat{M}$ and the values of the coefficients $c_i$ are indicated in Fig. \ref{analysisspacemod}. The children $m_i$ associated with irregular faces are equal to the face-based basis functions of $\mathbb{S}^1_A$. These children are called  \textit{active} children and the others are called \textit{passive} children. The truncated basis function $\hat{M}^t$ is obtained by discarding its active children
\begin{equation} \label{trunbb}
\hat{M}^t= \sum_{ i \in\mathcal{F}_p} c_{i} m_i \text{,}
\end{equation}
where $\mathcal{F}_p$ represents the index set of passive children. Note that depending on the distribution of EPs in the 1-ring vertices of the transition basis function, one, two, three or even all four faces in Fig. \ref{analysisspacemod} may be irregular faces. When all four faces in Fig. \ref{analysisspacemod} are irregular, all the children of the transition basis function are active and therefore the transition basis function is not included in the basis of $\mathbb{S}^1_A$.

The extraction operators are computed following the next steps:

\begin{itemize}
\item Initial extraction operators for vertex-based spline functions are obtained through Eqs. \eqref{firsteq}-\eqref{lasteq2}, but discarding the contributions from irregular and transition control points.
\item Face-based spline basis functions are $C^1$-continuous and have only four nonzero spline extraction coefficients whose values are specified in Fig. \ref{analysisspacemod2}. When the vertex extraction coefficient of one of the face-based spline basis functions is not an EP, that function is just a $C^1$-continuous bi-cubic B-spline.
\item Transition basis functions are truncated, that is, their active children, which are the face-based basis functions, are discarded. The extraction coefficients of truncated transition functions are obtained multiplying the coefficients $c_{i}$ with $i \in\mathcal{F}_p$ (see Eq. \eqref{trunbb} and Fig. \ref{analysisspacemod}) by the coefficients in Fig. \ref{analysisspacemod2}. 
\item The split-then-smoothen approach explained in steps (a)-(d) is applied to each vertex-based and face-based spline basis function with support on the 1-ring faces of at least one EP.
\end{itemize}

For an EP with valence 6, the final continuity across each element boundary in $\mathbb{S}^1_A$ is specified in Fig. \ref{continuity} (b).

\subsection{Control points}

\textit{Control points} determine the geometry of a T-spline surface. Control points have the same role as nodes in standard finite elements, but are not interpolatory. A control point is associated with each basis function in both $\mathbb{S}^1_D$ and $\mathbb{S}^1_A$. From now on, the control points of $\mathbb{S}^1_D$, the vertex-based control points of $\mathbb{S}^1_A$, the face-based control points of $\mathbb{S}^1_A$, and all the control points of $\mathbb{S}^1_A$ will be denoted by $\mathbf{P}_L$, $\mathbf{\hat{Q}}_V$, $\mathbf{\tilde{Q}}_F$, and $\mathbf{Q}_B$, respectively.

The control points of $\mathbb{S}^1_D$ form a control net, which has the same connectivity as the T-mesh, and moving these control points modifies the T-spline surface intuitively. Fig. \ref{controlpoints} (a) plots a possible control net for the T-mesh shown in Fig. \ref{tmesh} (a). Once a satisfactory geometry has been reached by moving the control points of $\mathbb{S}^1_D$, a set of control points for the basis functions of $\mathbb{S}^1_A$ that preserves the geometry can be obtained since $\mathbb{S}^1_D \subseteq \mathbb{S}^1_A$. This set of control points is obtained by:

\begin{figure} [t!] 
 \centering
 \subfigure[CAD]{\includegraphics[scale=.2]{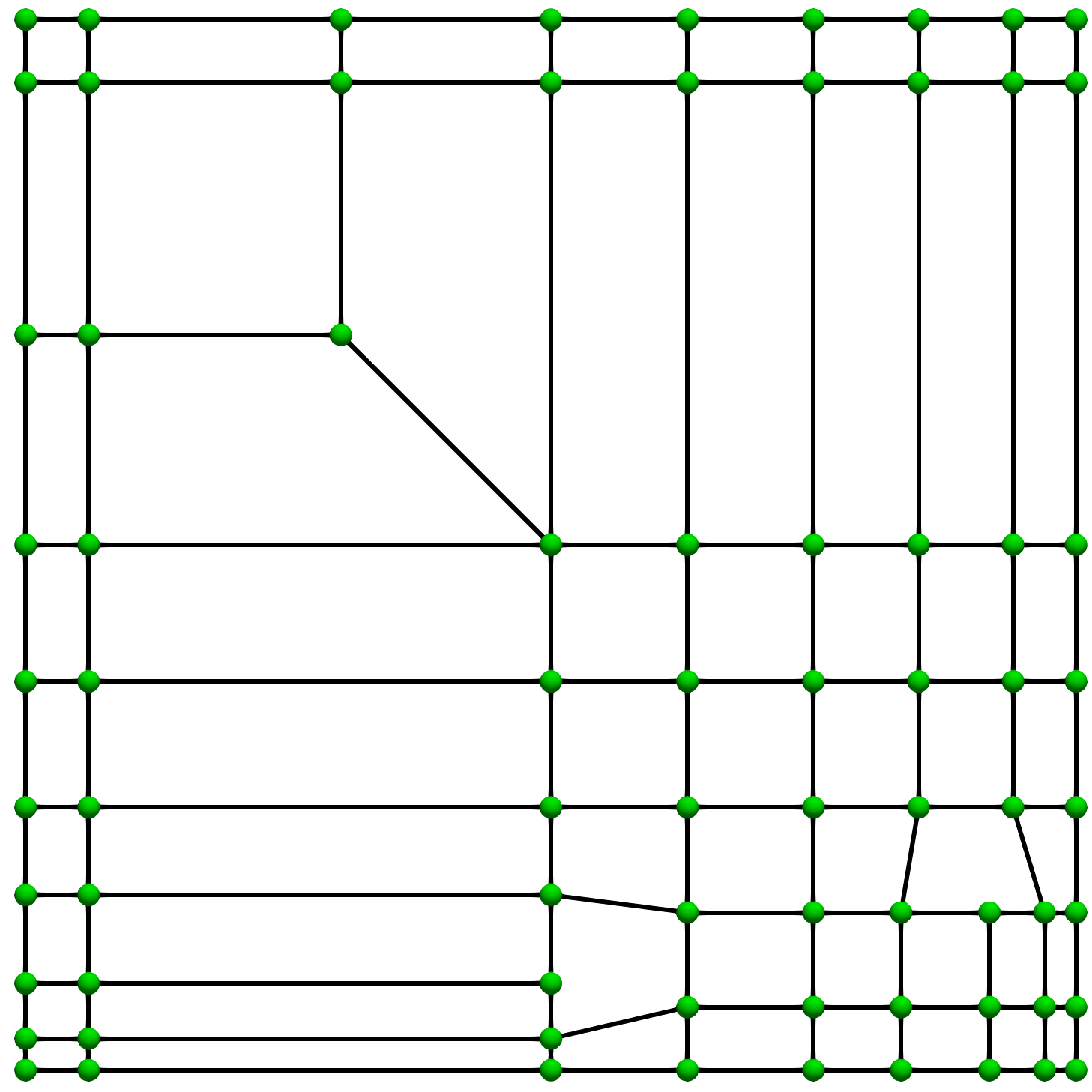}} \hspace*{+6mm}
 \subfigure[CAE]{\includegraphics[scale=.2]{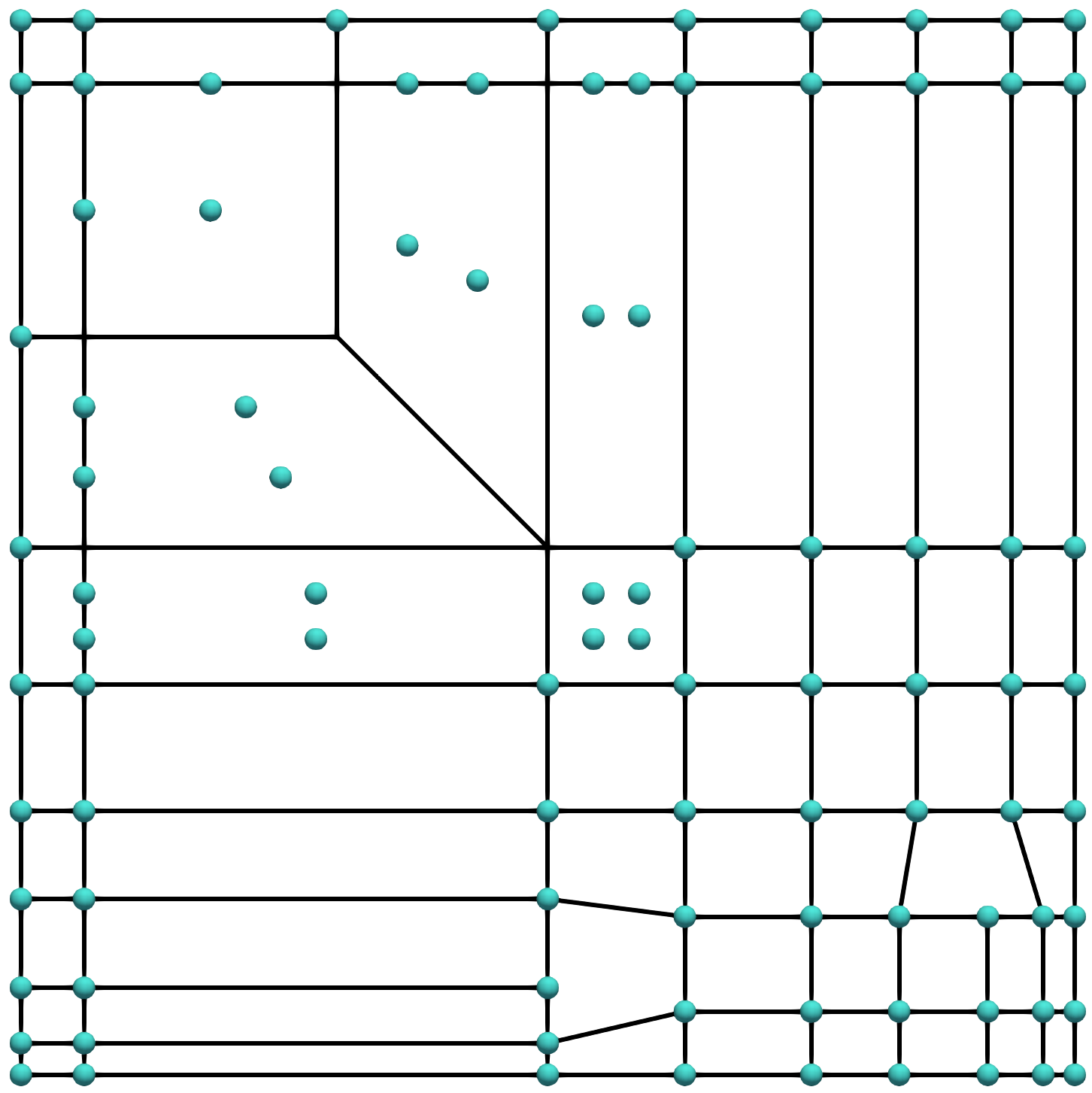}} \\
\caption{(a) The control points of $\mathbb{S}_D^1$ are represented with green circles and the control net of $\mathbb{S}_D^1$ is represented by solid black lines. (b) The control points of $\mathbb{S}_A^1$ are represented with blue circles. The control net of $\mathbb{S}_D^1$ is also plotted.}
\label{controlpoints}
\end{figure}

\begin{itemize}
\item Making the vertex-based control points of $\mathbb{S}^1_A$ ($\mathbf{\hat{Q}}_V$) equal to the equivalent control points of $\mathbb{S}^1_D$.
\item For each irregular face, making the face-based control points of $\mathbb{S}^1_A$ ($\mathbf{\tilde{Q}}_F$) equal to the face B\'ezier control points of the geometry in that face.
\end{itemize}

\noindent Fig. \ref{controlpoints} (b) plots the control points of $\mathbb{S}_A^1$ associated with the control points of $\mathbb{S}_D^1$ shown in Fig. \ref{controlpoints} (a).

The AST-spline surface is obtained by mapping each element of the elemental T-mesh into the Eucledian space as follows

\begin{equation}
\vec x^e \left( \vec {\xi} \right) = \sum_{a=1}^{n^e_D} \vec{P}^e_a N^e_a \left( \vec {\xi} \right)  = \sum_{a=1}^{n^e_A} \vec{Q}^e_a M^e_a \left( \vec {\xi} \right)   \quad \forall e\in\{1,2,...,n_{el}\}, \quad\vec\xi\in\square \text{,}
\end{equation}

\noindent where $n^e_D$ and $n^e_A$ are the number of basis functions with support on element $e$ in $\mathbb{S}^1_D$ and $\mathbb{S}^1_A$, respectively. The B\'ezier mesh is obtained by plotting the element boundaries over the T-spline surface. Fig. \ref{beziermesh} plots the B\'ezier mesh associated with the T-mesh, the knot span configuration, and the control points shown in Fig. \ref{tmesh} (a), Fig. \ref{knotspans} (a), and Fig. \ref{controlpoints}, respectively.

\begin{figure} [t!] 
\centering
\includegraphics[width=7cm]{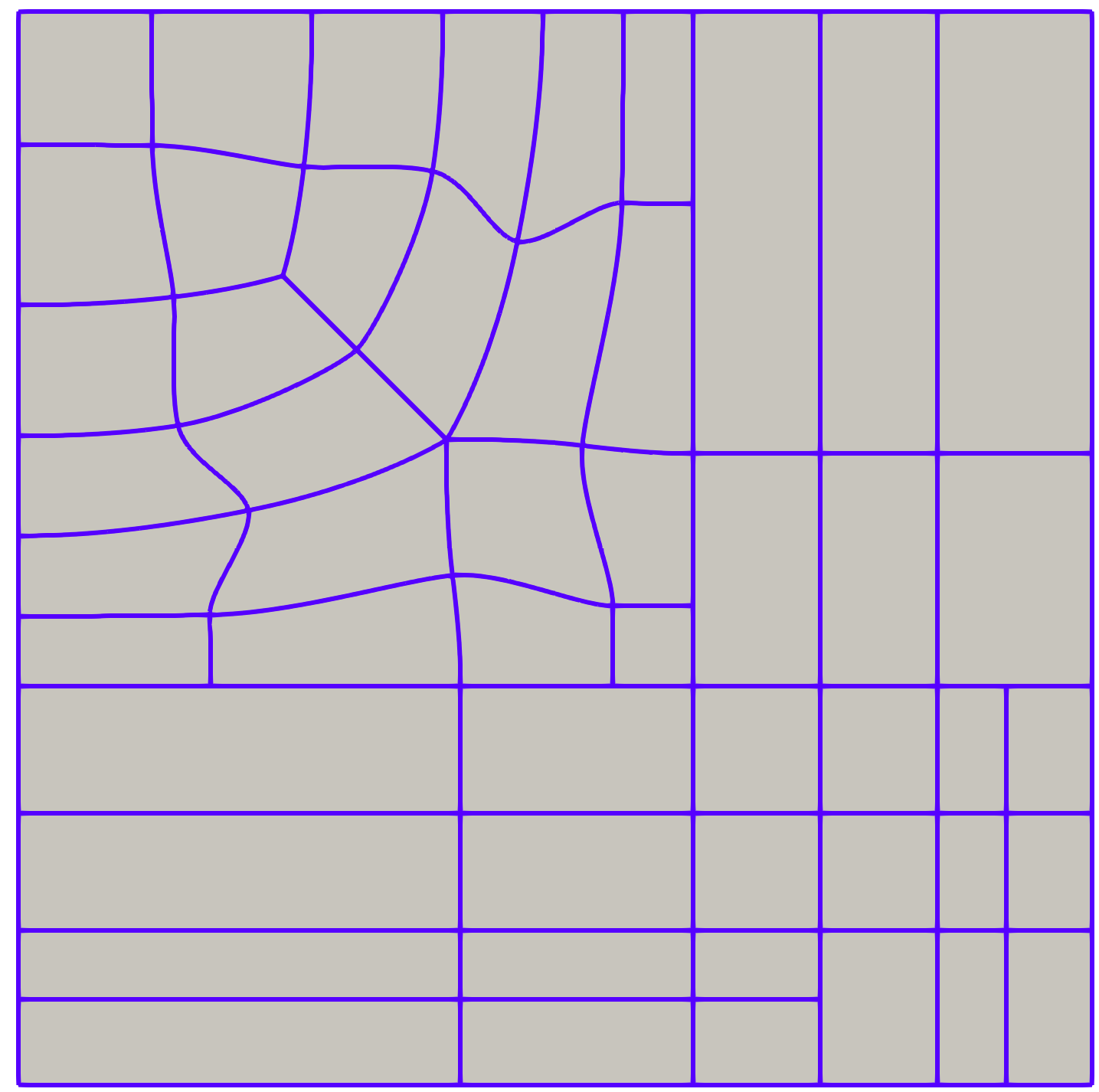}
\caption{(Color online) B\'ezier mesh associated with the T-mesh, the knot span configuration, and the control points shown in Fig. \ref{tmesh} (a), Fig. \ref{knotspans} (a), and Fig. \ref{controlpoints}, respectively.} 
\label{beziermesh}
\end{figure}

\subsection{Refinement}

Allowing multiple EPs per face does not change the way in which refinement is performed. Algorithms to perform local refinement in regular faces are explained in \cite{Scott2012}. How to refine irregular and transition faces in $\mathbb{S}^1_D$ is explained in \cite{toshniwal2017smooth} while how to refine irregular and transition faces in $\mathbb{S}^1_A$ is explained in \cite{casquero2020seamless}.


\section{Proofs}

In this section, we prove the properties of linear independence and non-negative partition of unity of the spaces $\mathbb{S}^1_D$ and $\mathbb{S}^1_A$. The linear independence and non-negative partition of unity in regular faces influenced by T-junctions are proven in \cite{Li2012, zhang2015linear, DaVeiga2013}. Therefore, in this section, we focus on irregular and transition faces influenced by EPs. To carry out the proofs, we denote as $\mathbb{S}^0_D$ and $\mathbb{S}^0_A$ the design and analysis space before the split-then-smoothen approach is applied, respectively. In the remaining of this section, the blending functions of $\mathbb{S}^0_D$, $\mathbb{S}^1_D$, $\mathbb{S}^0_A$, and $\mathbb{S}^1_A$ are denoted by $\{ N^0_L \}^{n}_{L=1}$, $\{ N^1_L \}^{n}_{L=1}$, $\{ M^0_B \}^{n_b}_{B=1}$, and $\{ M^1_B \}^{n_b}_{B=1}$, respectively.

\begin{lemma}\label{lemma:1}If the blending functions of $\mathbb{S}^0_D$ and $\mathbb{S}^0_A$ are linearly independent, then the blending functions of $\mathbb{S}^1_D$ and $\mathbb{S}^1_A$ are linearly independent.
\end{lemma}

\begin{proof}

$\{ N^1_L \}^n_{L=1}$ and $\{ M^1_B \}^{n_b}_{B=1}$ are obtained by applying the split-then-smoothen approach to $\{ N^0_L \}^n_{L=1}$ and $\{ M^0_B \}^{n_b}_{B=1}$, respectively. Therefore, this proof boils down to show that the split-then-smoothen approach preserves linear independence. In the remaining of this proof, we will work with $\{ M^1_B \}^{n_b}_{B=1}$ and $\{ M^0_B \}^{n_b}_{B=1}$, but the same reasoning can be applied to $\{ N^1_L \}^n_{L=1}$ and $\{ N^0_L \}^{n}_{L=1}$.

We wish to prove that if $\sum c_B M^1_B = 0$, then $\sum c_B M^0_B = 0$, which would imply that $ c_B = 0 \; \forall B \in \{1,2,...,n_b\}$ since $\{ M^0_B \}^{n_b}_{B=1}$ are assumed to be linearly independent. We observe that

\begin{itemize}
\item even though the spline coefficients $c_B$ are real numbers instead of vectors with three coordinates (as in the case of control points), B\'ezier extraction can be applied to the spline coefficients $c_B$ to obtain B\'ezier coefficients as in Eq. \eqref{bextractioncoef},
\item the zero function has zero B\'ezier coefficients since Bernstain polynomials form a basis, and
\item the function $\sum c_B M^0_B$ having zero face B\'ezier coefficients and zero B\'ezier coefficients placed at the boundary of the elemental T-mesh is equivalent to $\sum c_B M^0_B$ being the zero function due to Eqs. \eqref{5theq}-\eqref{lasteq2}\footnote{Eqs. \eqref{5theq}-\eqref{lasteq2} are true for any function of the spaces $\mathbb{S}^0_D$, $\mathbb{S}^1_D$, $\mathbb{S}^0_A$, and $\mathbb{S}^1_A$.}.

\end{itemize}

\begin{figure} [htbp]
 \centering
 \subfigure[Before split]{\includegraphics[scale=0.41]{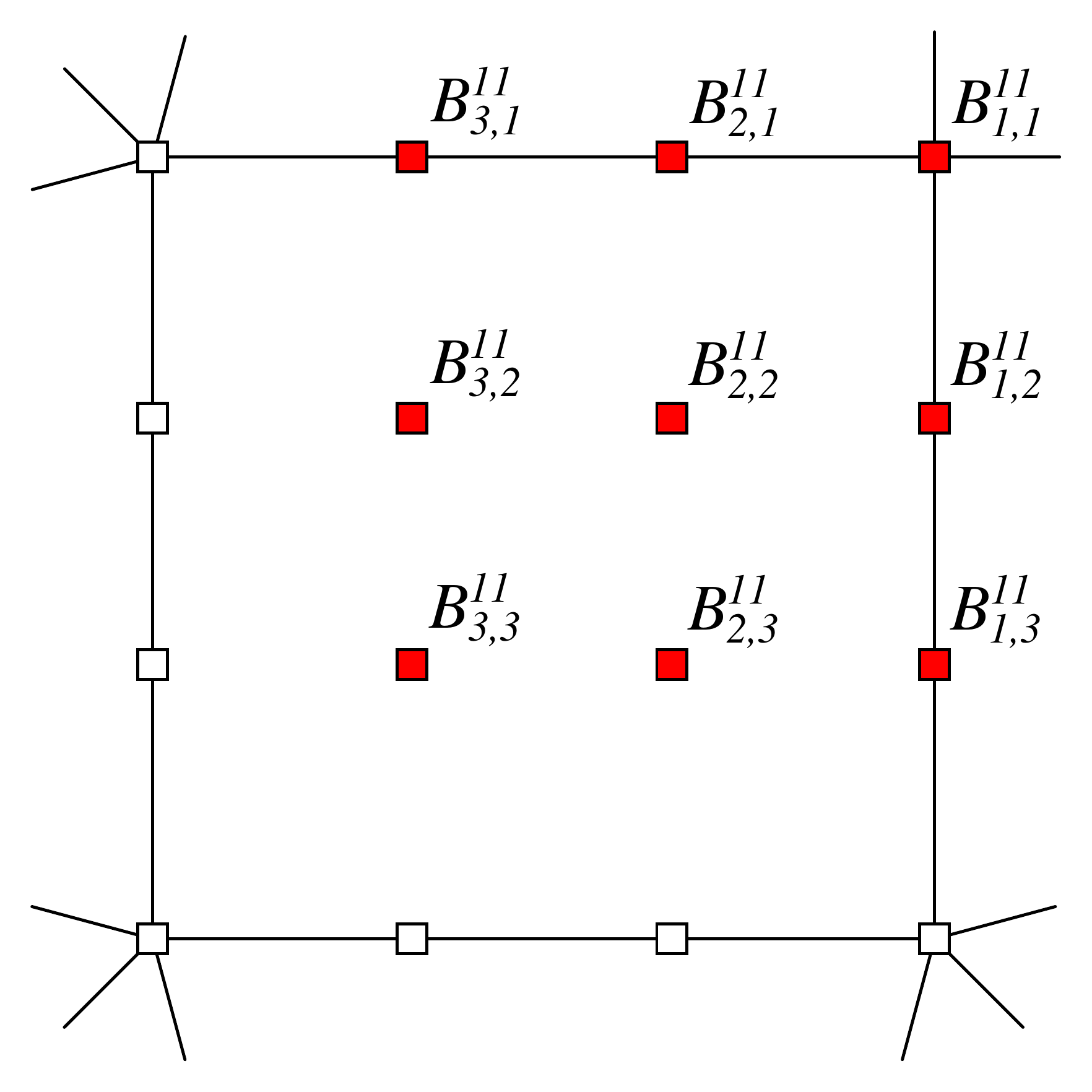}}
 \subfigure[After $2 \times 2$ split]{\includegraphics[scale=0.41]{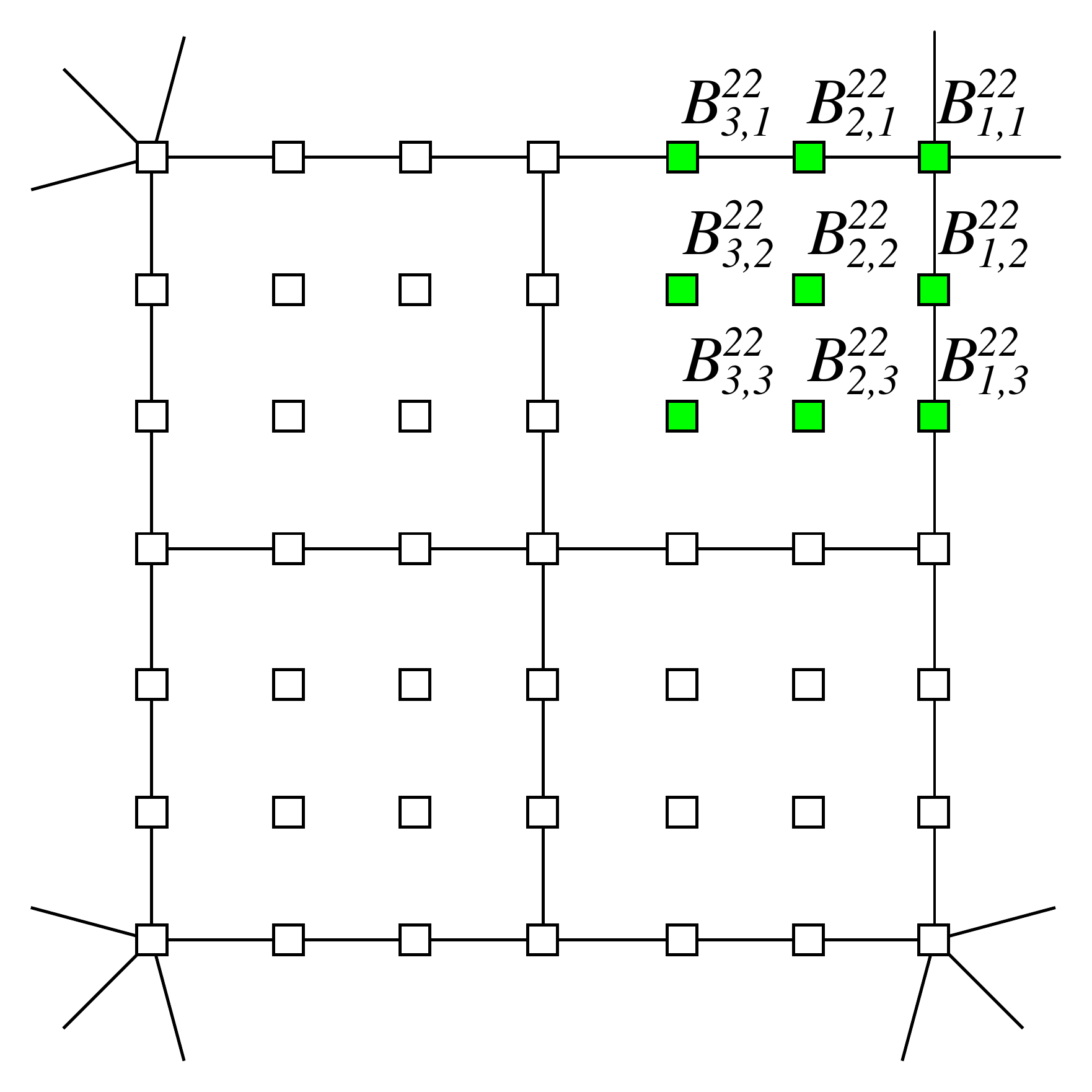}}  \\
\caption{(Color online) An irregular face with at least one vertex that is not an EP. The B\'ezier coefficients colored in green are computed from the B\'ezier coefficients colored in red using Eqs. \eqref{3epa}-\eqref{3epb}.}
\label{bezier3}
\end{figure}

Taking into account the above three items, we will prove a stronger statement: \textit{If the B\'ezier coefficients of the function $\sum c_B M^1_B$ that are not influenced by the D-patch smoothing process are zero, then the function $\sum c_B M^0_B$ has zero face B\'ezier coefficients and zero B\'ezier coefficients placed at the boundary of the elemental T-mesh.}

In transition faces, the B\'ezier coefficients of the functions $\{ M^1_B \}^{n_b}_{B=1}$ and $\{ M^0_B \}^{n_b}_{B=1}$ are the same. Thus, the B\'ezier coefficients of the function $\sum c_B M^0_B$ are zero in transition faces.

\begin{figure} [t!] 
 \centering
 \subfigure[]{\includegraphics[scale=0.41]{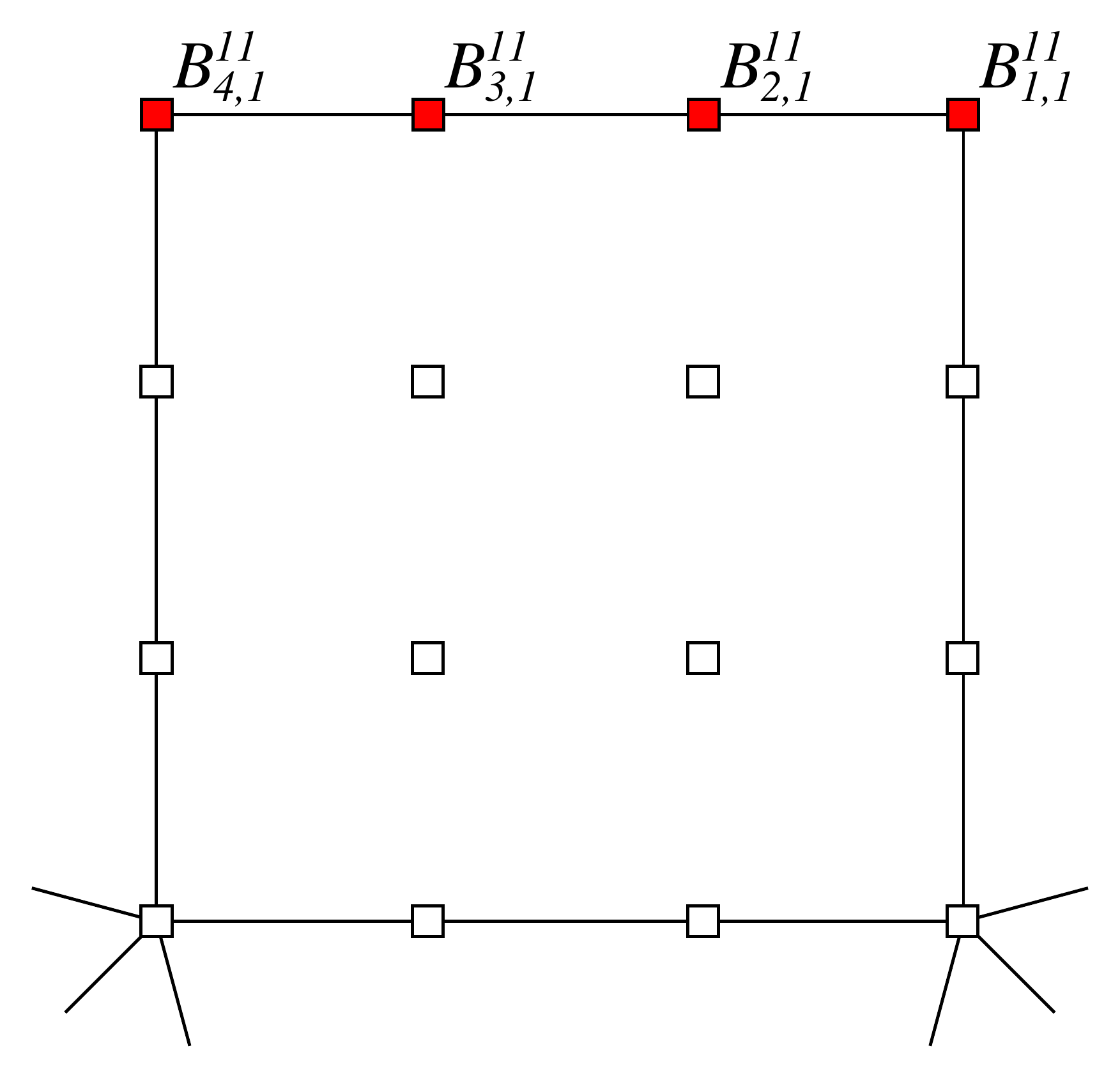}} \hspace*{+3mm}
 \subfigure[]{\includegraphics[scale=0.41]{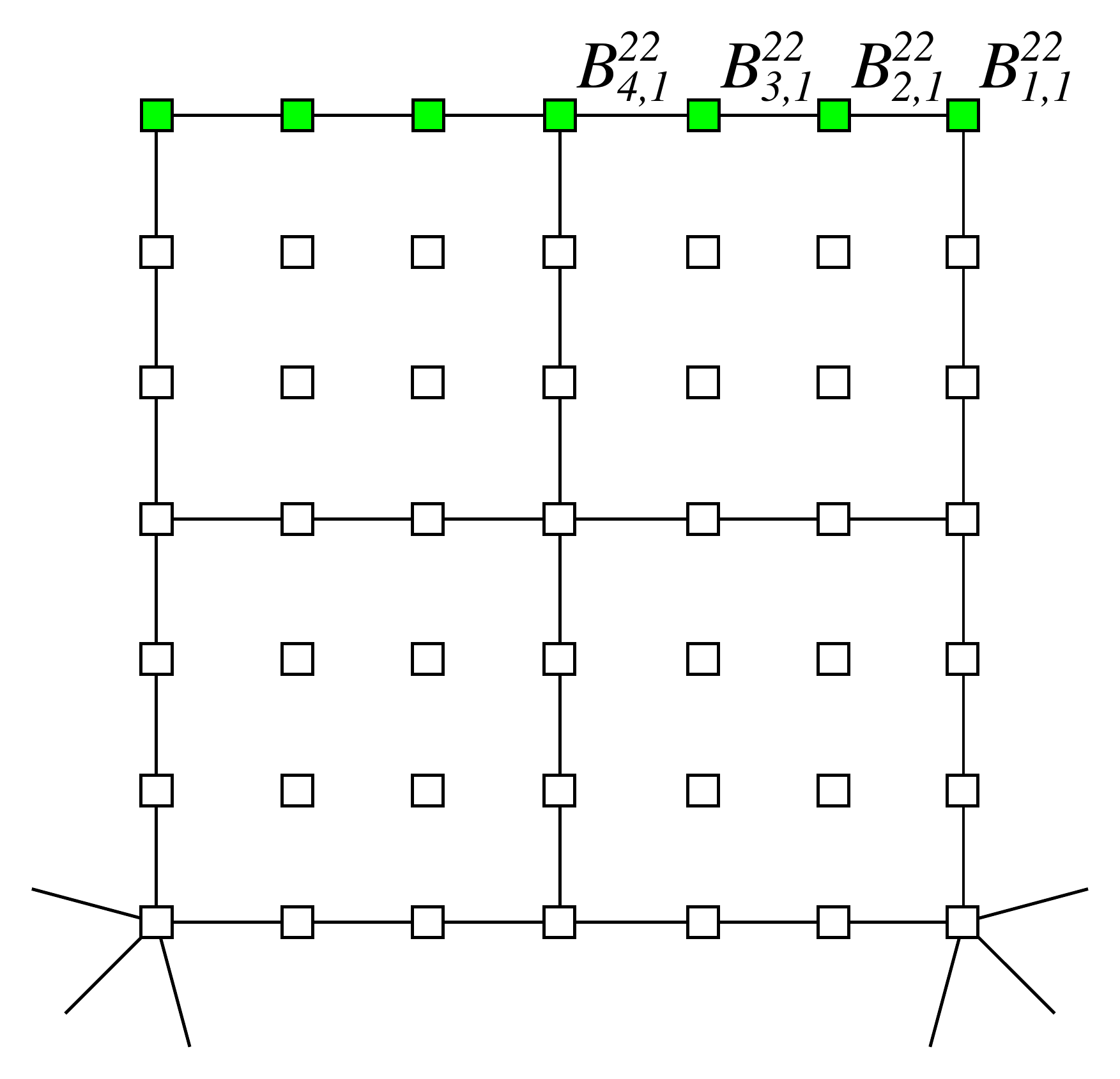}} \\
\caption{(Color online) An irregular face at the boundary of the elemental T-mesh. The B\'ezier coefficients colored in green are computed from the B\'ezier coefficients colored in red using Eq. \eqref{3epbboundary}.}
\label{bezier333}
\end{figure}

Let us now consider an irregular face with at least one vertex that is not an EP; see Fig. \ref{bezier3}. In Fig. \ref{bezier3} b), the B\'ezier coefficients colored in green are not influenced by the D-patch smoothing process. As a consequence, these coefficients are computed from the B\'ezier coefficients colored in red in Fig. \ref{bezier3} a) using the de Casteljau algorithm (see Appendix A), namely,
\begin{equation} \label{3epa}
\left(\begin{array}{ccc}
  B^{21}_{1, j} & B^{21}_{2, j} & B^{21}_{3, j}
\end{array}\right) = \left(\begin{array}{ccc}
  B^{11}_{1, j} & B^{11}_{2, j} & B^{11}_{3, j}
\end{array}\right) \left(
              \begin{array}{ccc}
                1 & 0 & 0 \\
                \frac{1}{2} & \frac{1}{2} & 0 \\
                \frac{1}{4} & \frac{1}{2} & \frac{1}{4} \\
              \end{array}
            \right), 1 \leq j \leq 3,
\end{equation}
where
\begin{equation} \label{3epb}
\left(\begin{array}{ccc}
  B^{22}_{i, 1} & B^{22}_{i, 2} & B^{22}_{i, 3}
\end{array}\right) = \left(\begin{array}{ccc}
  B^{21}_{i, 1} & B^{21}_{i, 2} & B^{21}_{i, 3}
\end{array}\right) \left(
              \begin{array}{ccc}
                1 & 0 & 0 \\
                \frac{1}{2} & \frac{1}{2} & 0 \\
                \frac{1}{4} & \frac{1}{2} & \frac{1}{4} \\
              \end{array}
            \right), 1 \leq i \leq 3.
\end{equation}
Since $B^{22}_{i, j} = 0 \; \forall i,j \in \{1,2,3\}$ and the matrix in Eqs. \eqref{3epa}-\eqref{3epb} is full rank, we have $B^{11}_{i, j} = 0 \; \forall i,j \in \{1,2,3\}$. Thus, the face coefficients of the function $\sum c_B M^0_B$ are zero in irregular faces with at least one vertex that is not an EP.

\begin{figure} [htbp]
 \centering
 \subfigure[Before split]{\includegraphics[scale=0.35]{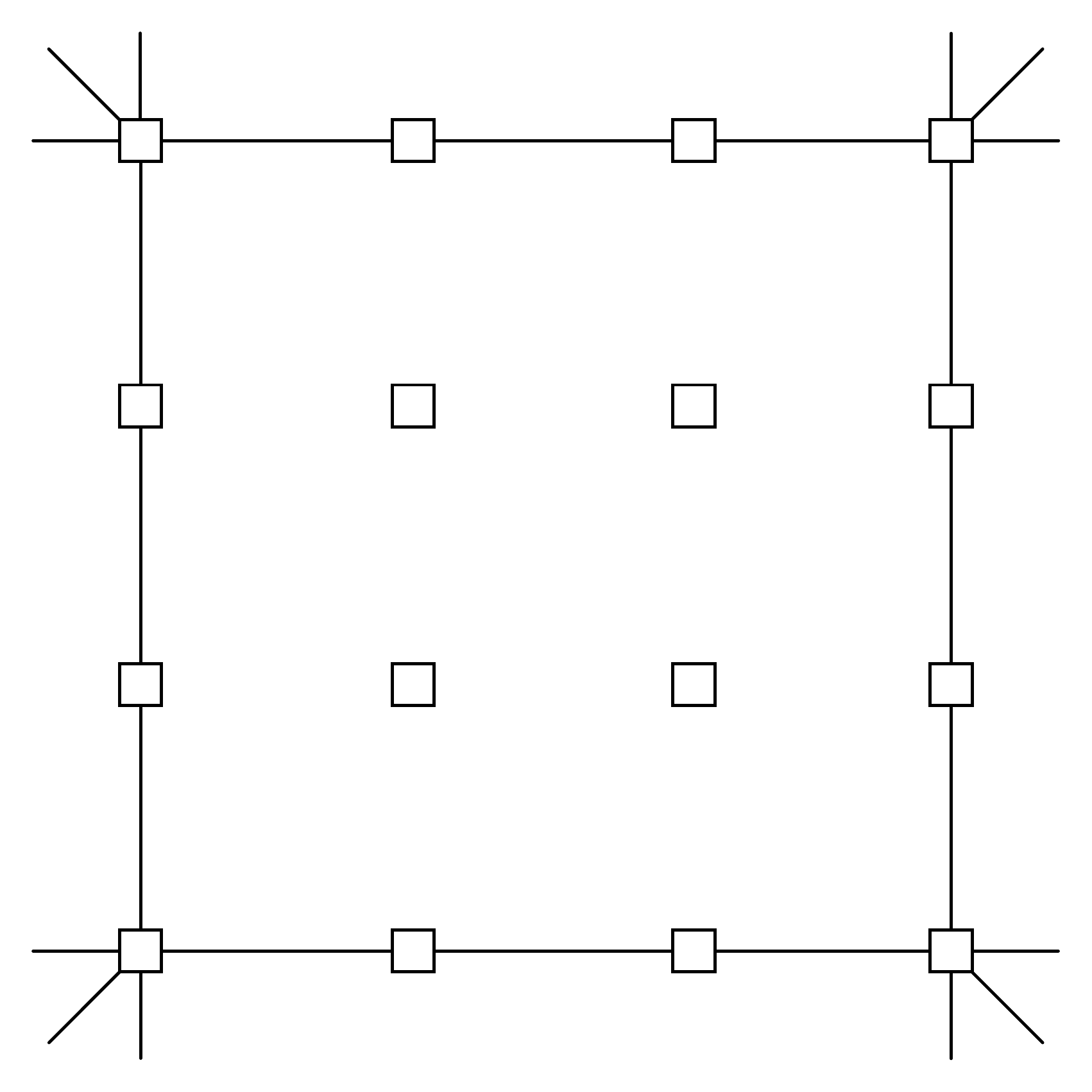}} \hspace*{+6mm}
 \subfigure[After $2 \times 2$ split]{\includegraphics[scale=0.35]{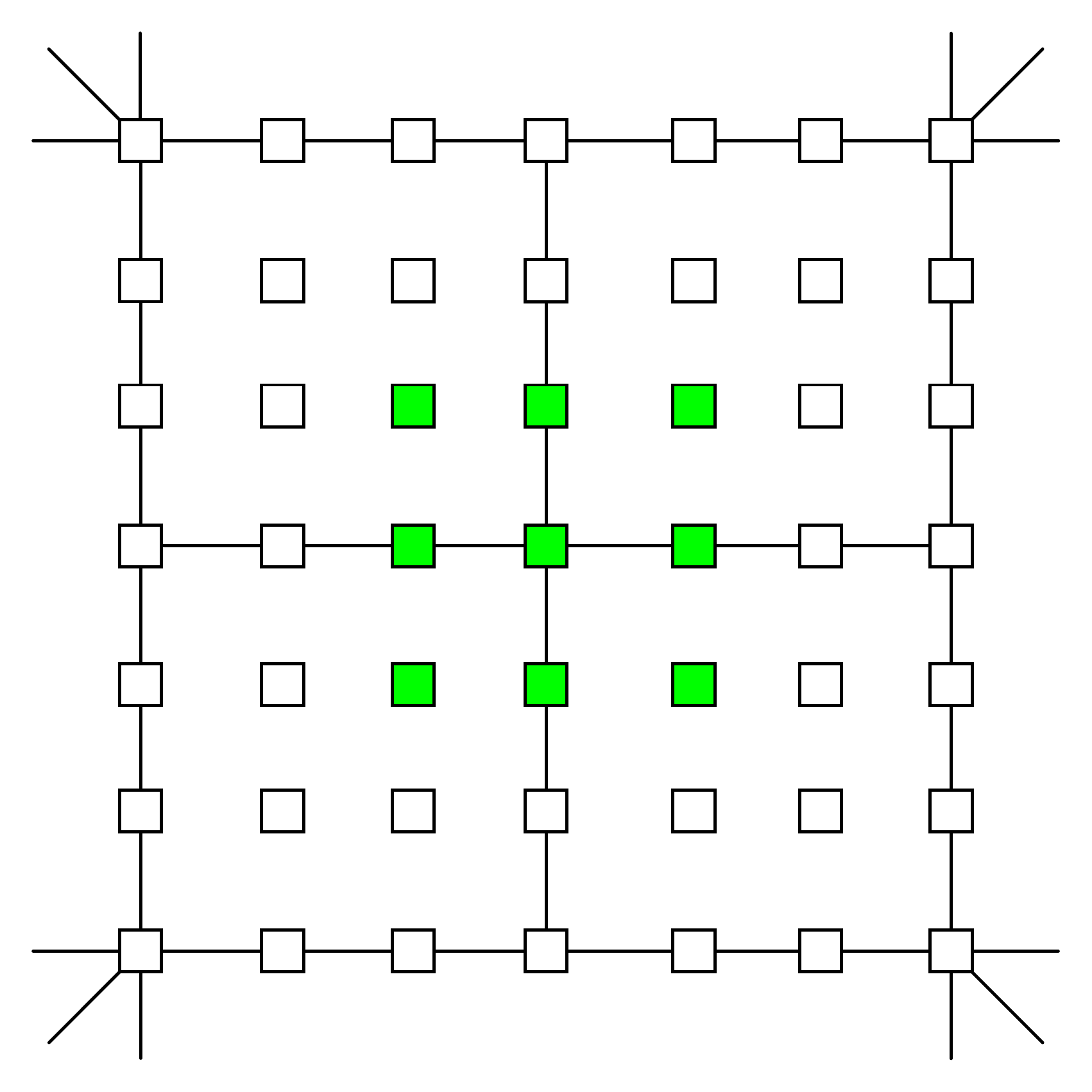}}  \\
\caption{(Color online) An irregular face with four EPs. The B\'ezier coefficients colored in green are not influenced by the D-patch smoothing process.}
\label{bezier40}
\end{figure}

When an irregular face is located at the boundary of the elemental T-mesh, we need to show that the function $\sum c_B M^0_B$ has zero B\'ezier coefficients at the boundary of the elemental T-mesh; see Fig. \ref{bezier333}. Taking into account that no EP is placed at the boundary of the elemental T-mesh since no EP belongs to the 0- and 1-layer vertices around the T-mesh boundary, the B\'ezier coefficients colored in green in Fig. \ref{bezier333} b) are not influenced by the D-patch smoothing process. Therefore, these coefficients are computed from the B\'ezier coefficients colored in red in Fig. \ref{bezier333} a) using the de Casteljau algorithm (see Appendix A), namely,
\begin{equation} \label{3epbboundary}
\left(\begin{array}{cccc}
  B^{22}_{1, 0} & B^{22}_{2, 0} & B^{22}_{3, 0} & B^{22}_{4, 0}
\end{array}\right) = \left(\begin{array}{cccc}
  B^{11}_{1, 0} & B^{11}_{2, 0} & B^{11}_{3, 0} & B^{11}_{4, 0}
\end{array}\right) \left(
              \begin{array}{cccc}
                1 & 0 & 0 & 0\\
                \frac{1}{2} & \frac{1}{2} & 0 & 0\\
                \frac{1}{4} & \frac{1}{2} & \frac{1}{4} & 0\\
                \frac{1}{8} & \frac{3}{8} & \frac{3}{8} & \frac{1}{8}\\
              \end{array}
            \right).
\end{equation}
Since $B^{22}_{i, 0} = 0 \; \forall i \in \{1,2,3,4\}$ and the matrix in Eq. \eqref{3epbboundary} is full rank, we have $B^{11}_{i, 0} = 0 \; \forall i \in \{1,2,3,4\}$. Thus, the function $\sum c_B M^0_B$ has zero B\'ezier coefficients placed at the boundary of the elemental T-mesh in irregular faces.

Let us now consider an irregular face with four EPs, henceforth known as totally irregular face (see Fig. \ref{bezier40}). Firstly, note that a totally irregular face cannot be located at the boundary of the elemental T-mesh since no EP belongs to the 0- and 1-layer vertices around the T-mesh boundary. The key observation for the remaining of this proof is that the B\'ezier coefficients plotted in green in Fig. \ref{bezier40} b) are not influenced by the D-patch smoothing process; i.e., these coefficients are computed from the B\'ezier coefficients in Fig. \ref{bezier40} a) using the de Casteljau algorithm. We will first prove that if the B\'ezier coefficients plotted in brown in Fig. \ref{bezier2} c) are zero, then the B\'ezier coefficients plotted in brown in Fig. \ref{bezier2} b) are zero. After that, we will prove that if the B\'ezier coefficients plotted in brown in Fig. \ref{bezier2} b) are zero, then the B\'ezier coefficients plotted in brown in Fig. \ref{bezier2} a) are zero.

\begin{figure} [htbp]
 \centering
 \subfigure[Before split]{\includegraphics[scale=0.33]{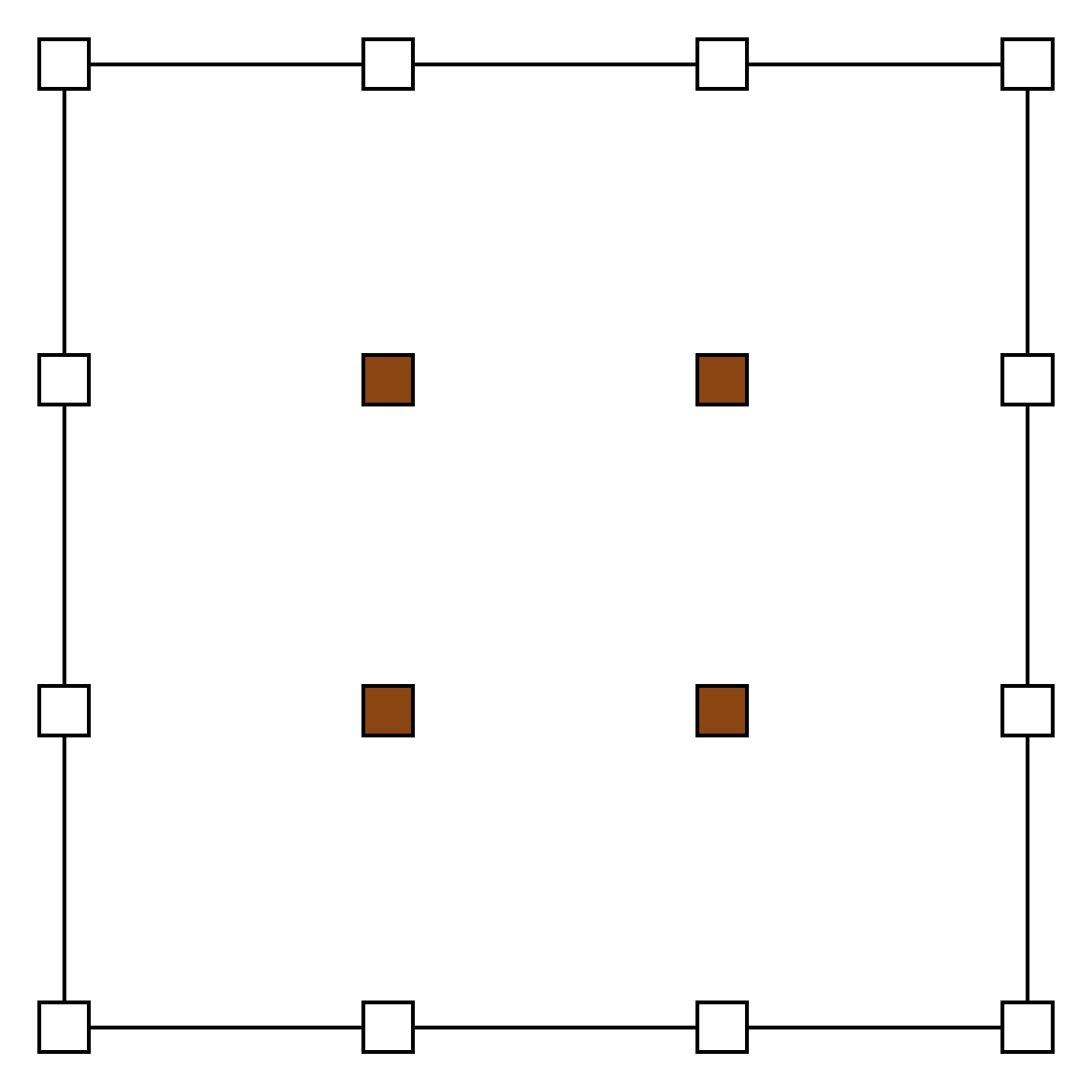}} \hspace*{+6mm}
 \subfigure[After $2 \times 1$ split]{\includegraphics[scale=0.33]{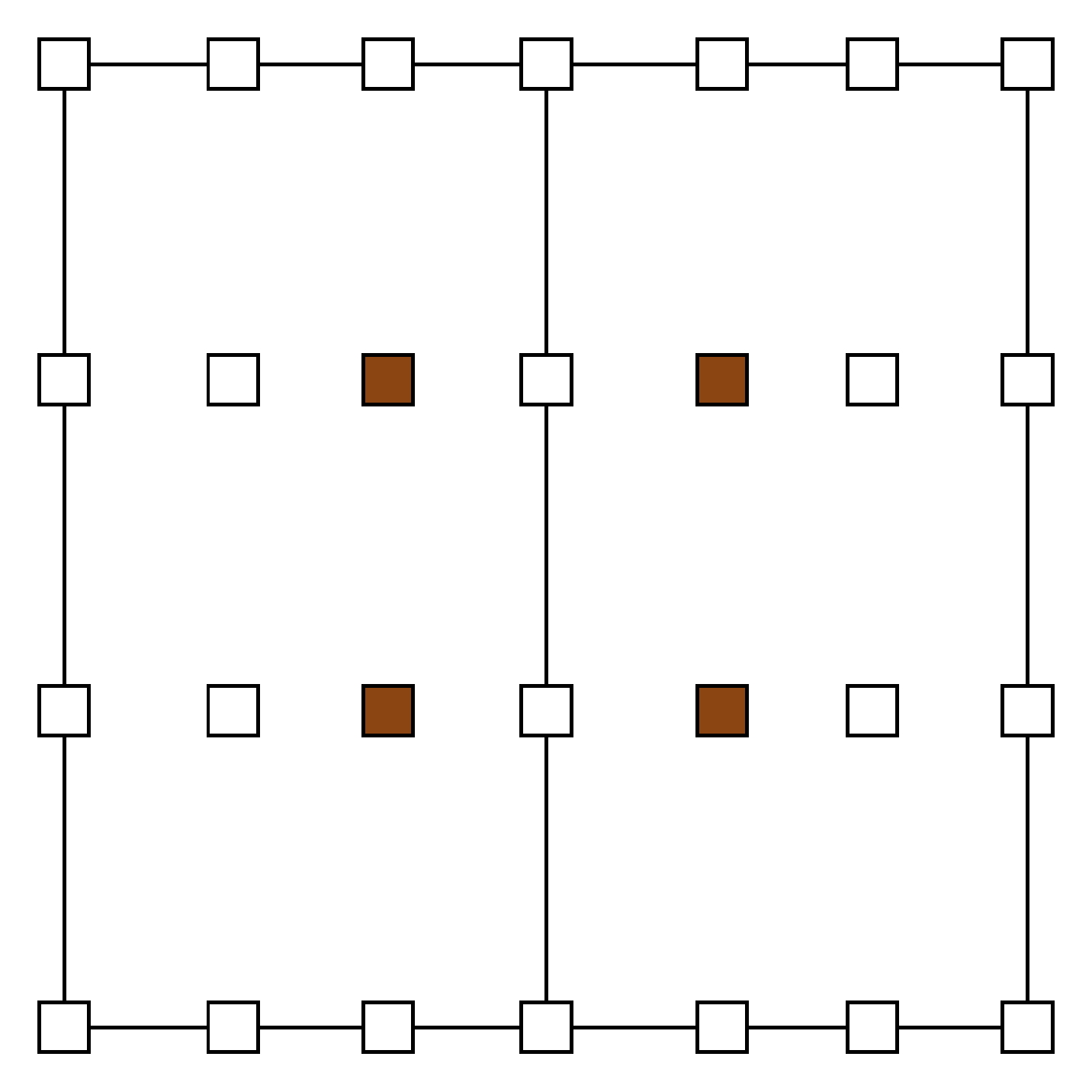}} \hspace*{+6mm}
 \subfigure[After $2 \times 2$ split]{\includegraphics[scale=0.33]{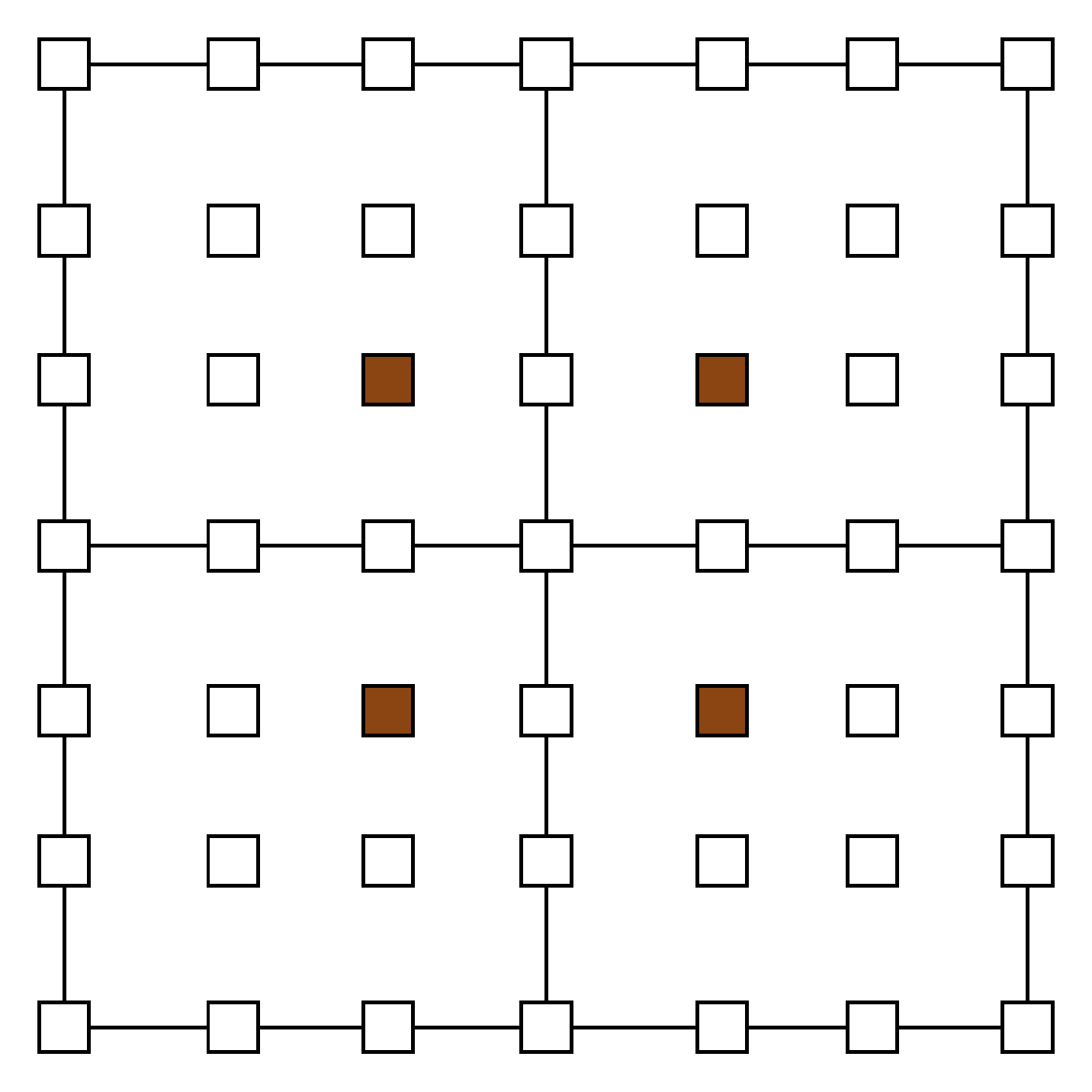}} \\
\caption{(Color online) Linear transformations between the B\'ezier coefficients colored in brown will be established, which requires to take into consideration neighboring faces as well.}
\label{bezier2}
\end{figure}

To establish a relation between the B\'ezier coefficients plotted in brown in Fig. \ref{bezier2} c) and b), two cases need to be distinguished as shown in Figs. \ref{bezier4} and \ref{bezier5}. Fig. \ref{bezier4} represents a strip of totally irregular faces connected with two faces that are not totally irregular faces while Fig. \ref{bezier5} represents a closed strip of totally irregular faces. Since we do not allow EPs in the 0- and 1-layer around the T-mesh boundary, a closed  strip of totally irregular faces can only exist in closed surfaces, e.g., a cube meshed with one T-mesh face for each cube face (all the vertices are EPs with valence 3).

\begin{figure} [htbp]
 \centering
 \subfigure[After $2 \times 1$ split]{\includegraphics[scale=0.35]{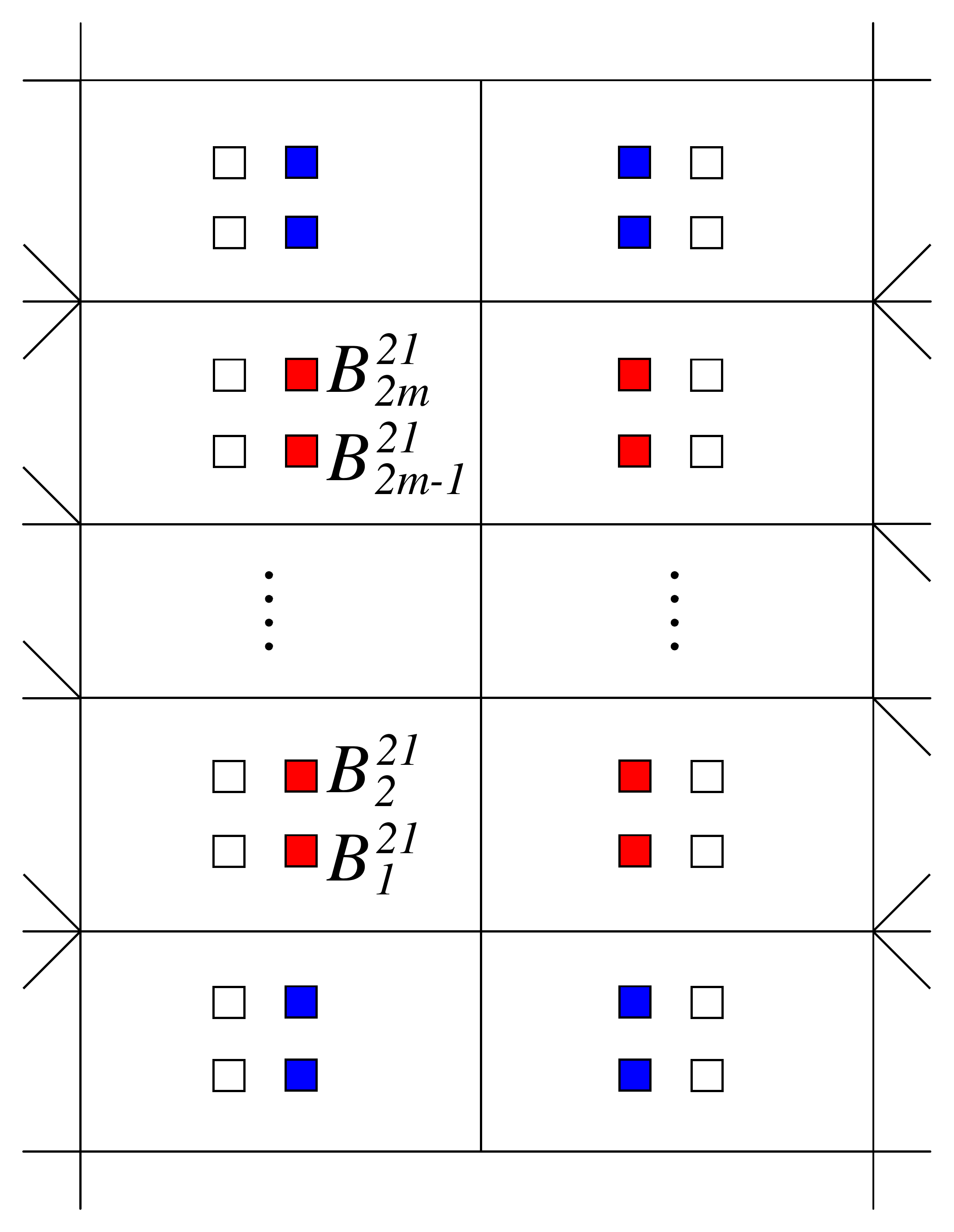}}
 \subfigure[After $2 \times 2$ split]{\includegraphics[scale=0.35]{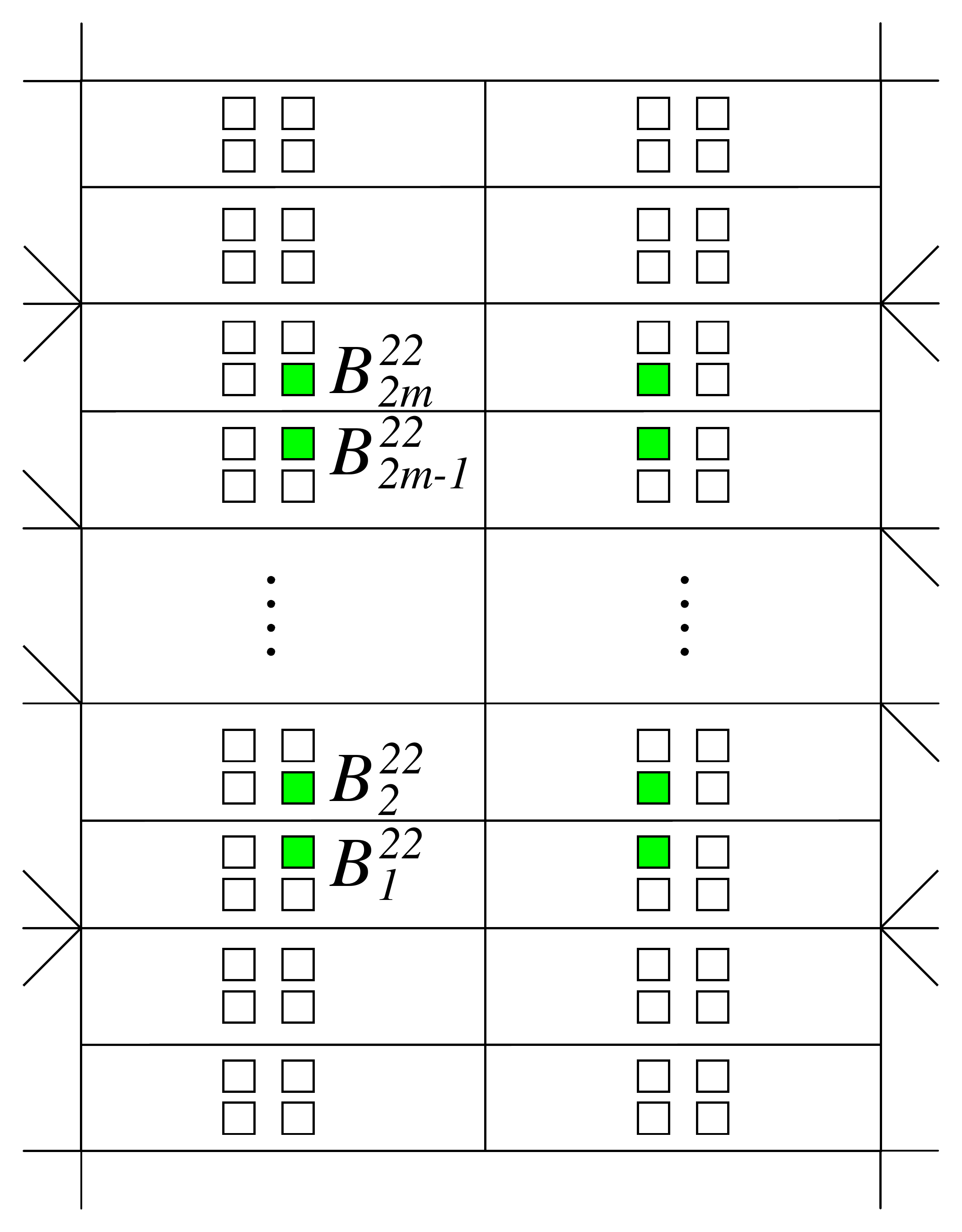}}  \\
\caption{(Color online) A strip of $m$ totally irregular faces connected with two faces that are not totally irregular faces.}
\label{bezier4}
\end{figure}

\begin{figure} [htbp]
 \centering
 \subfigure[After $2 \times 1$ split]{\includegraphics[scale=0.38]{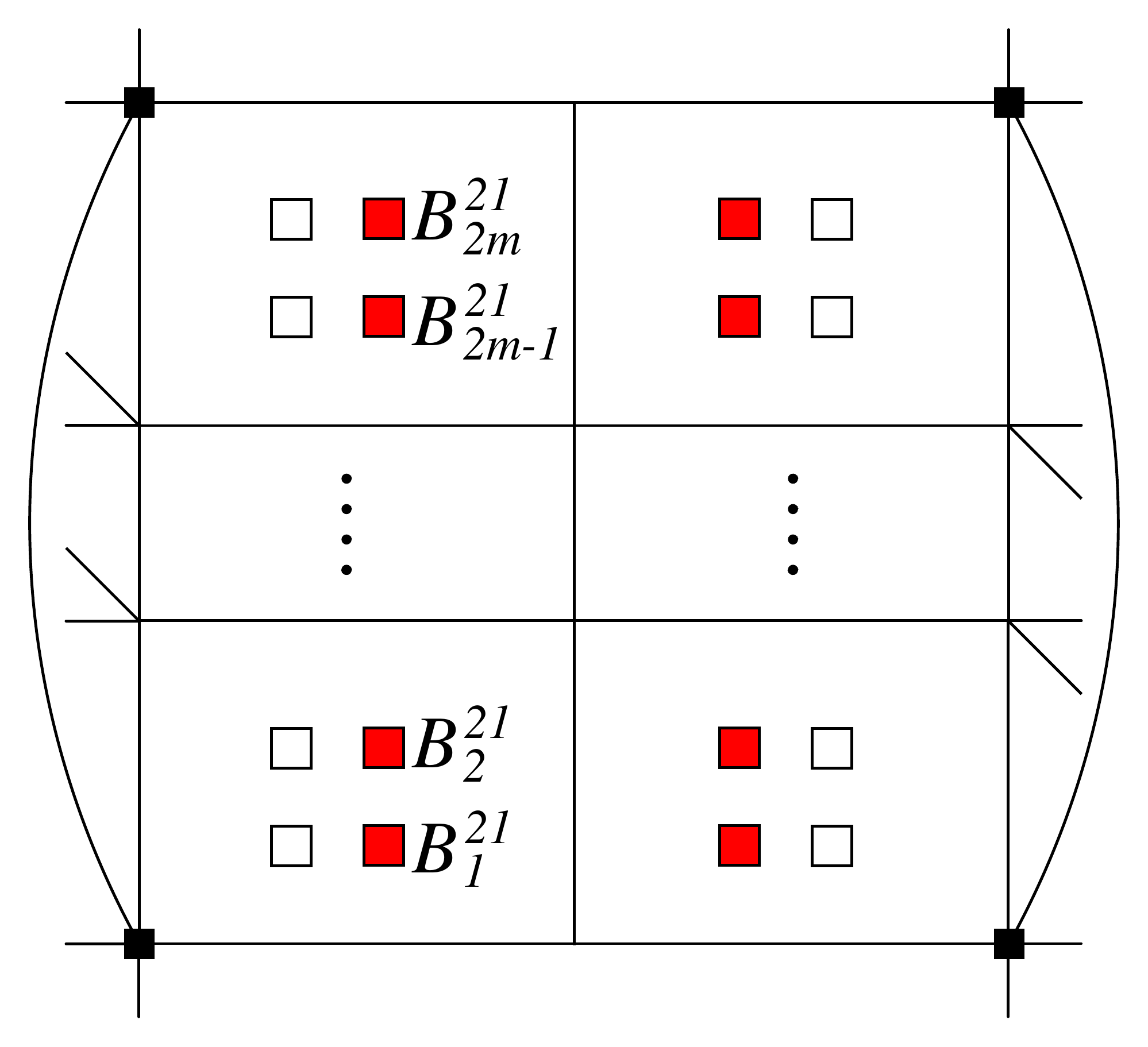}}
 \subfigure[After $2 \times 2$ split]{\includegraphics[scale=0.38]{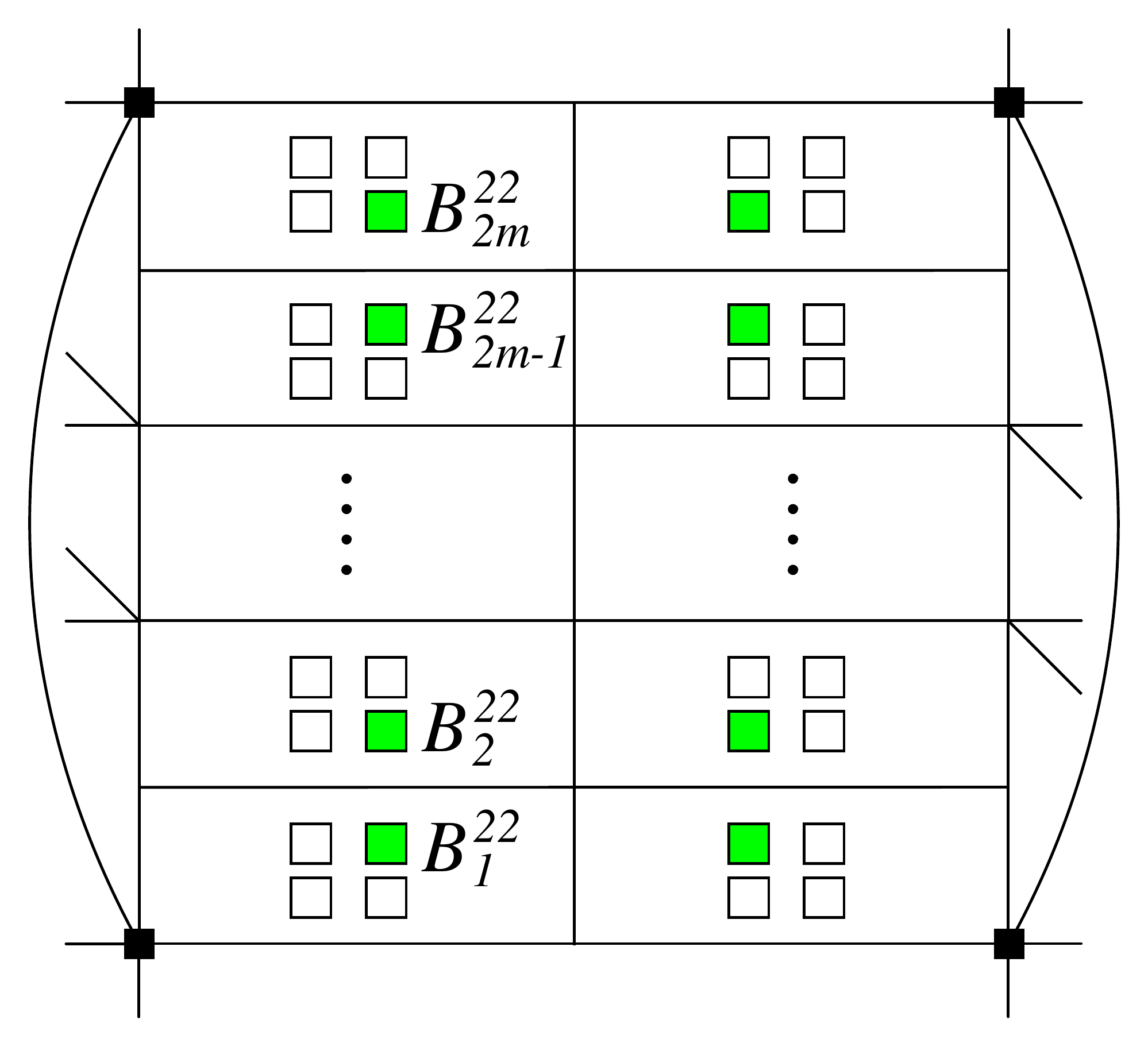}}  \\
\caption{(Color online) A closed strip of $m$ totally irregular faces.}
\label{bezier5}
\end{figure}

Using the labels indicated in  Fig. \ref{bezier4}, we define the vectors $\vec{B}^{22} = (B^{22}_1, B^{22}_2, ..., B^{22}_{2m})^	T$ and $\vec{B}^{21} = (B^{21}_1, B^{21}_2, ..., B^{21}_{2m})^T$ together with the matrix $\vec{M}_1$ such that $\vec{B}^{22} = \vec{M}_1 \vec{B}^{21}$. Using the de Casteljau algorithm (see Appendix A), Eqs. \eqref{5theq} - \eqref{6theq}, and taking into account that the B\'ezier coefficients plotted in blue in Fig. \ref{bezier4} a) are zero due to Eqs. \eqref{3epa}-\eqref{3epb}, the matrix $\vec{M}_1$ turns out to take the form
\begin{equation}
\vec{M}_{1} = \left(
      \begin{array}{cccccc}
        \frac{5}{8} & \frac{1}{4} & 0 & \dots & 0 & 0 \\
        \frac{1}{8} & \frac{5}{8} & \frac{1}{4} & 0 & \dots & 0 \\
        0 & \frac{1}{8} & \frac{5}{8} & \frac{1}{4} & \dots & 0 \\
        \dots & \dots & \dots & \dots & \dots & \dots \\
        0 & 0 & \dots & 0 & \frac{1}{8} & \frac{5}{8} \\
      \end{array}
    \right) \text{.}
\end{equation}
The matrix $\vec{M}_1$ is full rank since it is a diagonally dominant matrix. 

Using the labels indicated in  Fig. \ref{bezier5}, we define the vectors $\vec{B}^{22} = (B^{22}_1, B^{22}_2, ..., B^{22}_{2m})^T$ and $\vec{B}^{21} = (B^{21}_1, B^{21}_2, ..., B^{21}_{2m})^T$ together with the matrix $\vec{M}_2$ such that $\vec{B}^{22} = \vec{M}_1 \vec{B}^{21}$. Using the de Casteljau algorithm (see Appendix A) and Eqs. \eqref{5theq} - \eqref{6theq}, the matrix $\vec{M}_2$ turns out to take the form
\begin{equation}
\vec{M}_{2} = \left(
      \begin{array}{cccccc}
        \frac{5}{8} & \frac{1}{4} & 0 & \dots & 0 & \frac{1}{8} \\
        \frac{1}{8} & \frac{5}{8} & \frac{1}{4} & 0 & \dots & 0 \\
        0 & \frac{1}{8} & \frac{5}{8} & \frac{1}{4} & \dots & 0 \\
        \dots & \dots & \dots & \dots & \dots & \dots \\
        \frac{1}{4} & 0 & \dots & 0 & \frac{1}{8} & \frac{5}{8} \\
      \end{array}
    \right) \text{.}
\end{equation}
The matrix $\vec{M}_2$ is full rank since it is a diagonally dominant matrix. 

Since matrices $\vec{M}_1$ and $\vec{M}_2$ are full rank, the B\'ezier coefficients plotted in brown in Fig. \ref{bezier2} c) being zero imply that the B\'ezier coefficients plotted in brown in Fig. \ref{bezier2} b) are also zero.

\begin{figure} [t!]
 \centering
 \subfigure[Before split]{\includegraphics[scale=0.35]{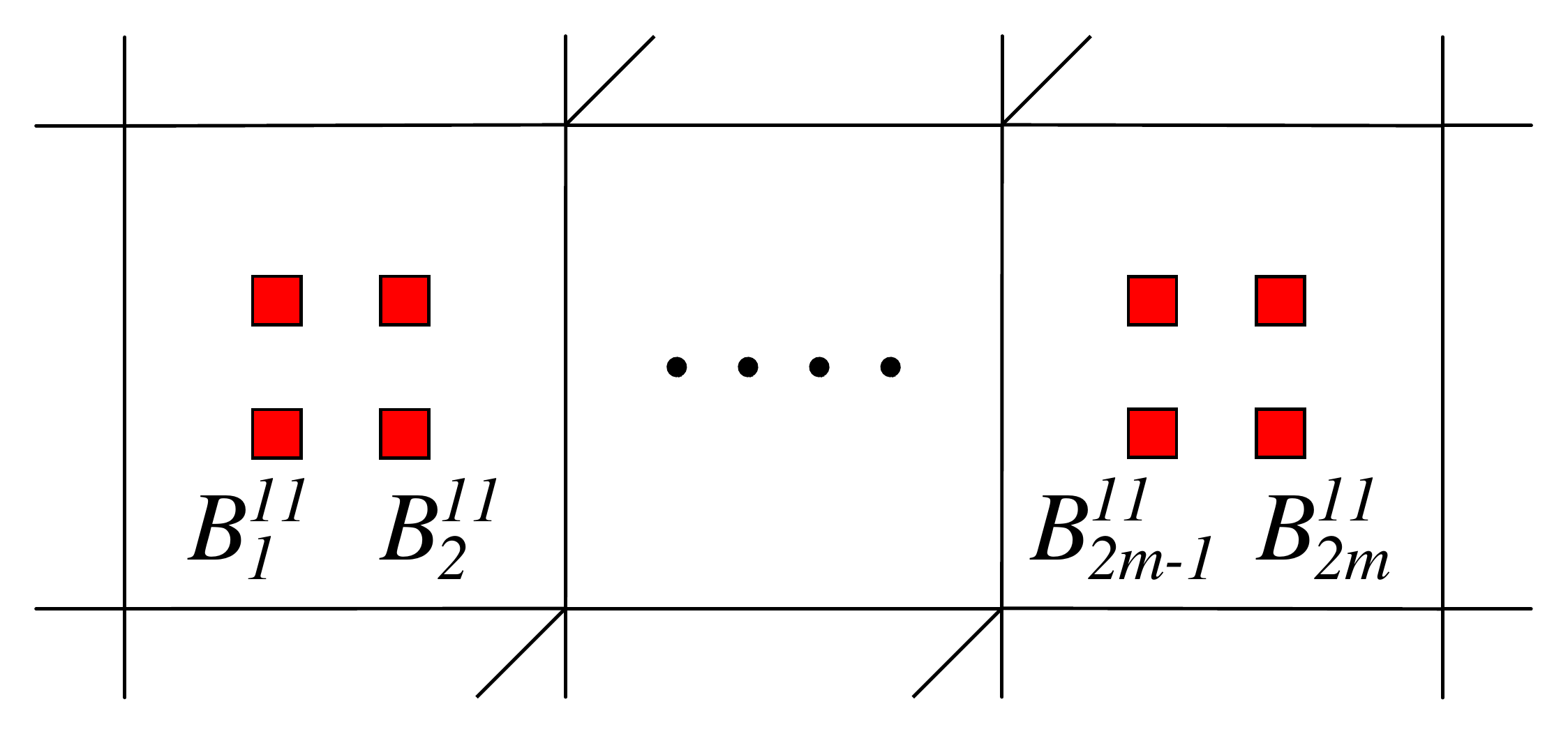}}
 \subfigure[After $2 \times 1$ split]{\includegraphics[scale=0.35]{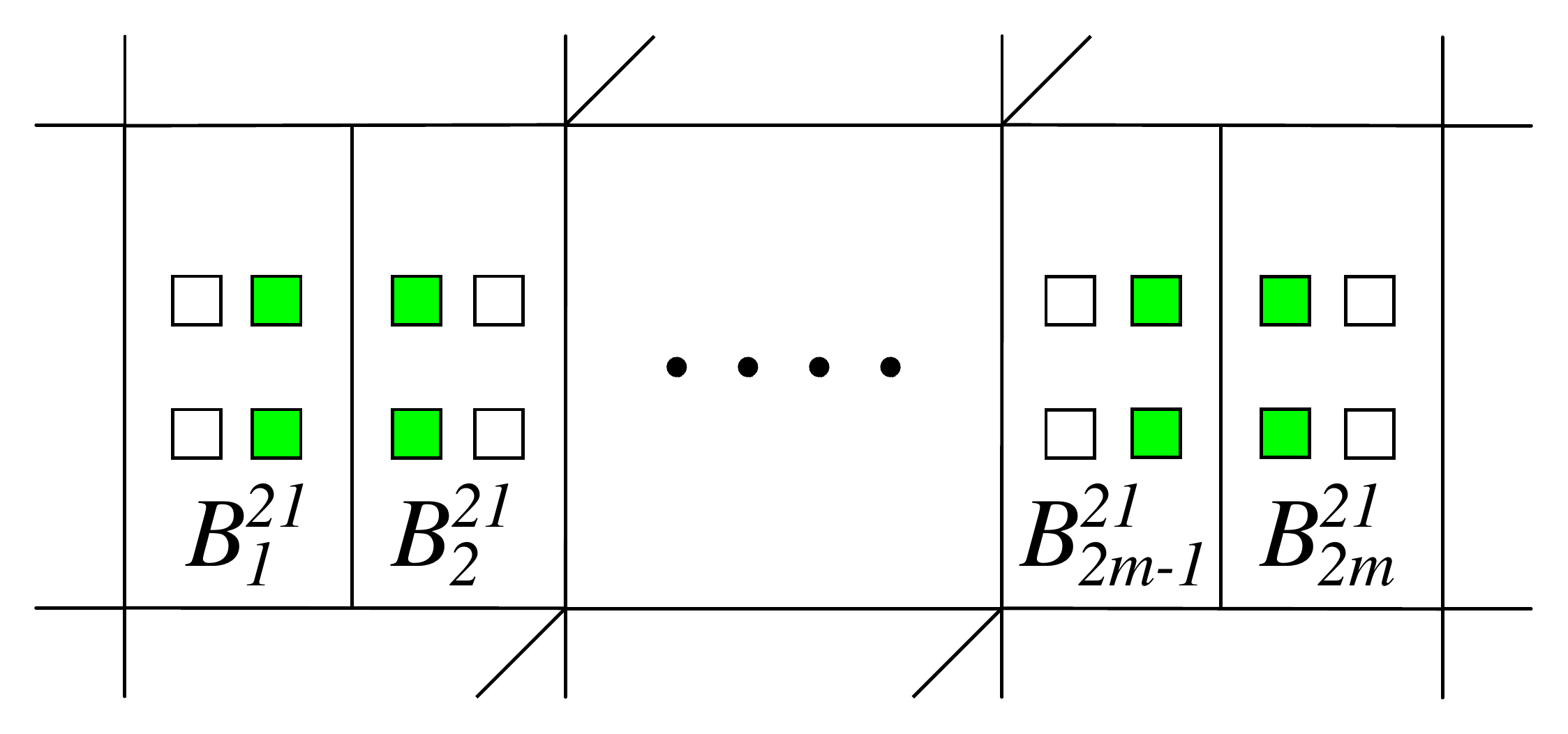}}  \\
\caption{(Color online) A strip of $m$ totally irregular faces connected with two faces that are not totally irregular faces.}
\label{bezier6}
\end{figure}

\begin{figure} [t!]
 \centering
 \subfigure[Before split]{\includegraphics[scale=0.35]{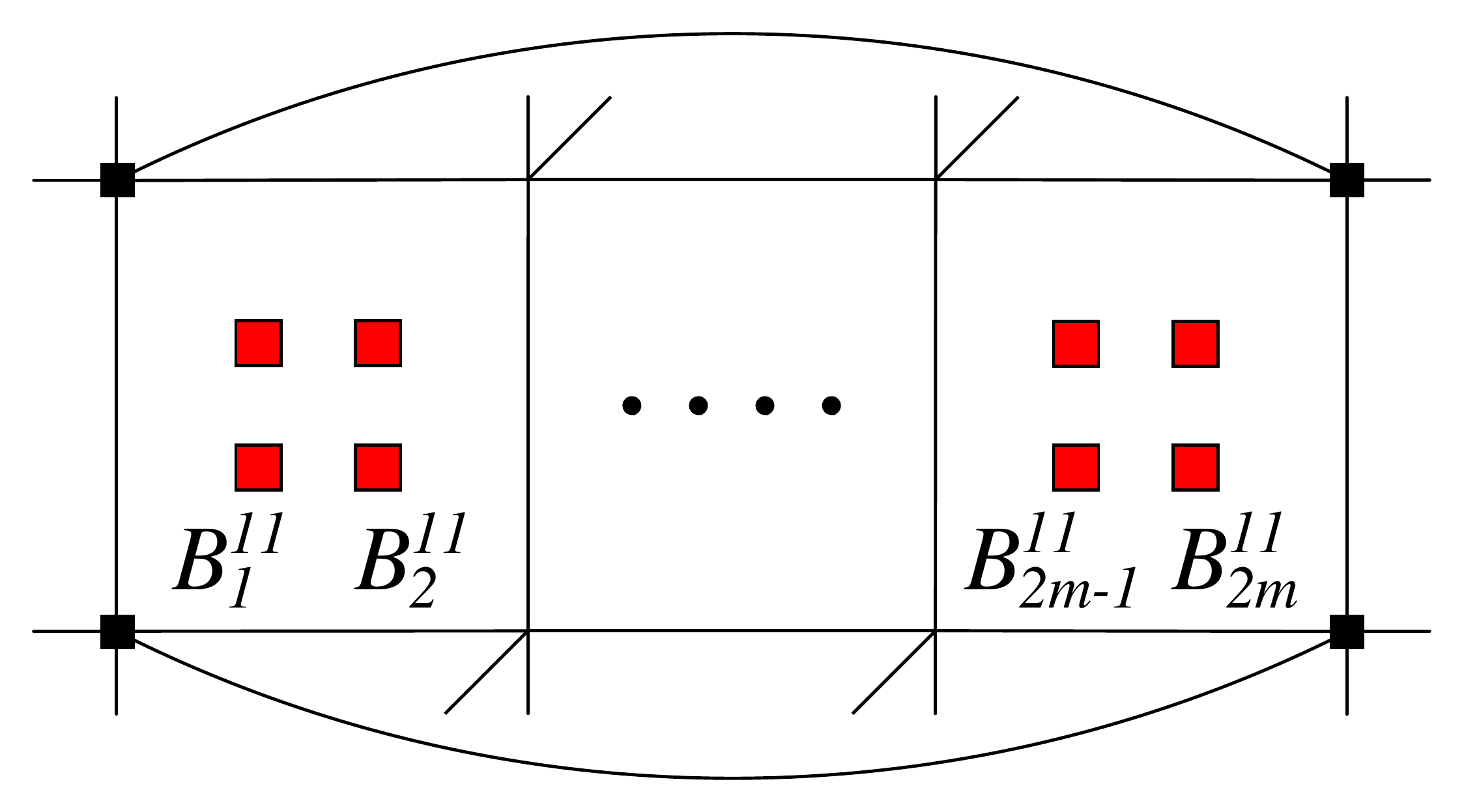}}
 \subfigure[After $2 \times 1$ split]{\includegraphics[scale=0.35]{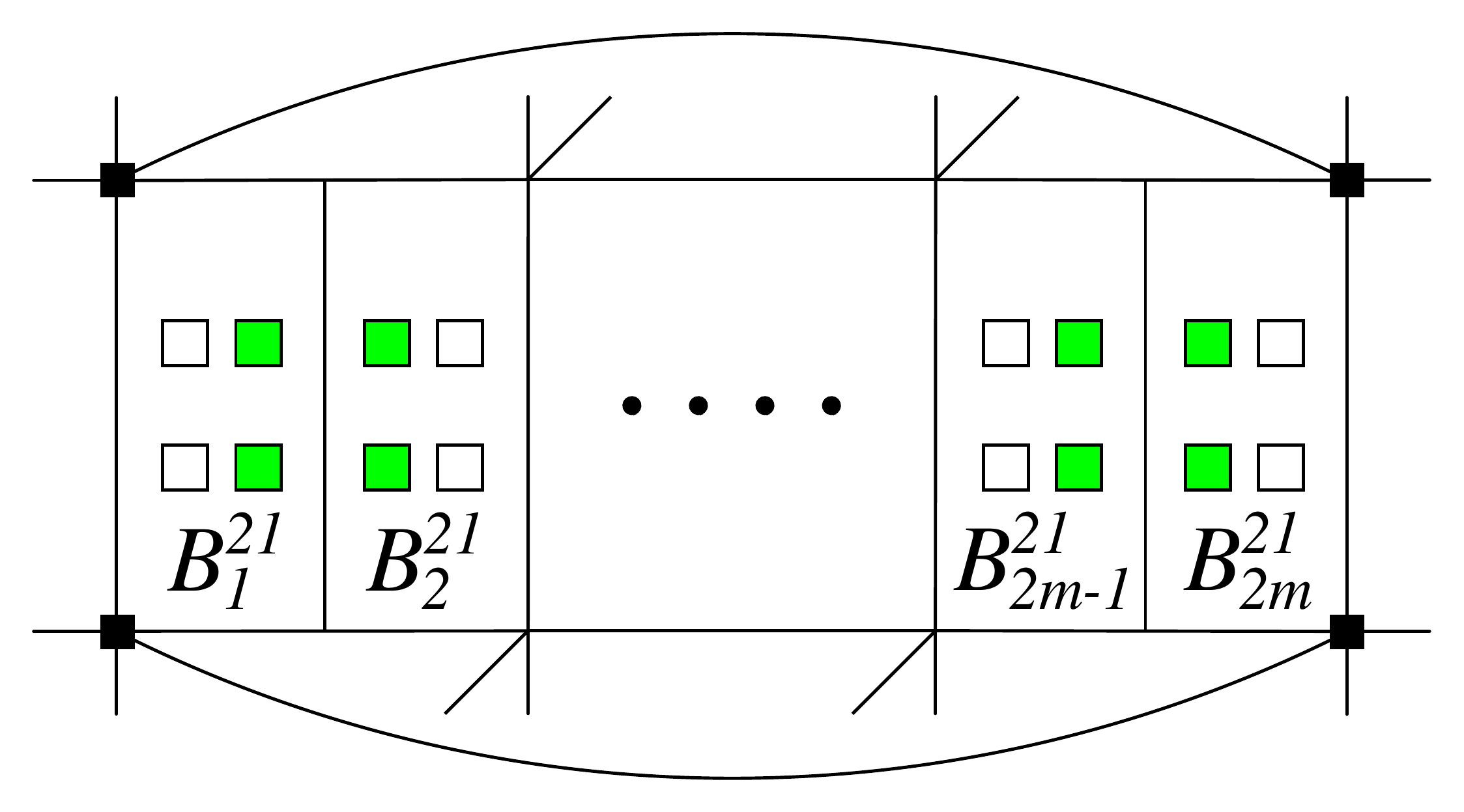}}  \\
\caption{(Color online) A closed strip of $m$ totally irregular faces.}
\label{bezier7}
\end{figure}

To establish a relation between the B\'ezier coefficients plotted in brown in Fig. \ref{bezier2} b) and a), the same two cases need to be considered again as shown in Figs. \ref{bezier6} and \ref{bezier7}. Using the same reasoning as above, it is obtained that the B\'ezier coefficients $\vec{B}^{21} = (B^{21}_1, B^{21}_2, ..., B^{21}_{2m})^T$  and $\vec{B}^{11} = (B^{11}_1, B^{11}_2, ..., B^{11}_{2m})^T$  of Figs. \ref{bezier6} and \ref{bezier7} are once again related through matrices $\vec{M}_1$ and $\vec{M}_2$, respectively. Thus, the B\'ezier coefficients plotted in brown in Fig. \ref{bezier2} b) being zero implies that the B\'ezier coefficients plotted in brown in Fig. \ref{bezier2} a) are also zero. As a result, the face B\'ezier coefficients of the function $\sum c_B M^0_B$ are zero in totally irregular faces.

We have already proven that the function $\sum c_B M^0_B$ has zero face B\'ezier coefficients and zero B\'ezier coefficients placed at the boundary of the elemental T-mesh, which implies $ c_B = 0 \; \forall B \in \{1,2,...,n_b\}$ since $\{ M^0_B \}^{n_b}_{B=1}$ are assumed to be linearly independent. Therefore, $\{ M^1_B \}^{n_b}_{B=1}$ (and $\{ N^1_L \}^{n}_{L=1}$) are linearly independent if $\{ M^0_B \}^{n_b}_{B=1}$ (and $\{ N^0_L \}^{n}_{L=1}$) are linearly independent.


\end{proof}

\begin{lemma}\label{lemma:2} The blending functions of $\mathbb{S}^0_D$ are linearly independent.
\end{lemma}

\begin{proof}
We wish to prove that if $\sum c_L N^0_L = 0$, then $ c_L = 0 \; \forall L \in \{1,2,...,n\}$. In \cite{zhang2015linear}, it was proven that regular faces influenced by T-junctions are locally linearly independent. Therefore, the spline coefficients $c_L$ associated with functions $N_L$ with support on regular faces are zero. In other words, the spline coefficients $c_L$ associated with regular and transition functions are zero.

In an irregular face, at least one of the four vertices is an EP (see Fig. \ref{aaaa}). Since the spoke edges of an EP have the same knot span associated, Eqs. \eqref{firsteq}-\eqref{4theq} lead to the following relations between the face B\'ezier coefficients $B^i_6$, $B^i_7$, $B^i_{10}$, and $B^i_{11}$ and the spline coefficients with indices $j_1$, $j_2$, $j_3$, and $j_4$
\begin{equation} \label{matrix44}
\left(
  \begin{array}{c}
    B^i_6 \\
    B^i_7 \\
    B^i_{10} \\
    B^i_{11} \\
  \end{array}
\right) = \frac{1}{(2a+b)(2a+c)}\left(
            \begin{array}{cccc}
              2a(a+b) & 2a^2 & c(a+b) & ac \\
              2ab & 4a^2 & bc & 2ac \\
              a(a+b) & a^2 & (a+c)(a+b) & a(a+c) \\
              ab & 2a^2 & b(a+c) & 2a(a+c) \\
            \end{array}
          \right)
          \left(
            \begin{array}{c}
              c_{j_1} \\
              c_{j_2} \\
              c_{j_3} \\
              c_{j_4} \\
            \end{array}
          \right).
\end{equation}
The determinant of the matrix in Eq. \ref{matrix44} is $a^{4}(2a+b)(2a+c)$, which is greater than zero since $a > 0$ and $b, c \geq 0$. $\sum c_L N^0_L = 0$ implies that the face B\'ezier coefficients $B^i_6$, $B^i_7$, $B^i_{10}$, and $B^i_{11}$ are zero since the Bernstain polynomials form a basis. Thus, $ c_{j_k}= 0 \; \forall k \in \{1,2,3,4\}$ since the matrix in Eq. \ref{matrix44} is full rank, which means the spline coefficients $c_L$ associated with irregular functions are also zero. Therefore, $\{ N^0_L \}^{n}_{L=1}$ are linearly independent.

\end{proof}

\begin{lemma}\label{lemma:3} The blending functions of $\mathbb{S}^0_A$ are linearly independent.
\end{lemma}

\begin{proof}
We wish to prove that if $\sum c_B M^0_B = 0$, then $ c_B = 0 \; \forall B \in \{1,2,...,n_b\}$. As in Lemma \ref{lemma:2}, the spline coefficients $c_B$ associated with regular and transition functions are zero due to the fact that regular faces influenced by T-junctions are locally linearly independent.

\begin{figure} [t!] 
\centering
\includegraphics[width=7cm]{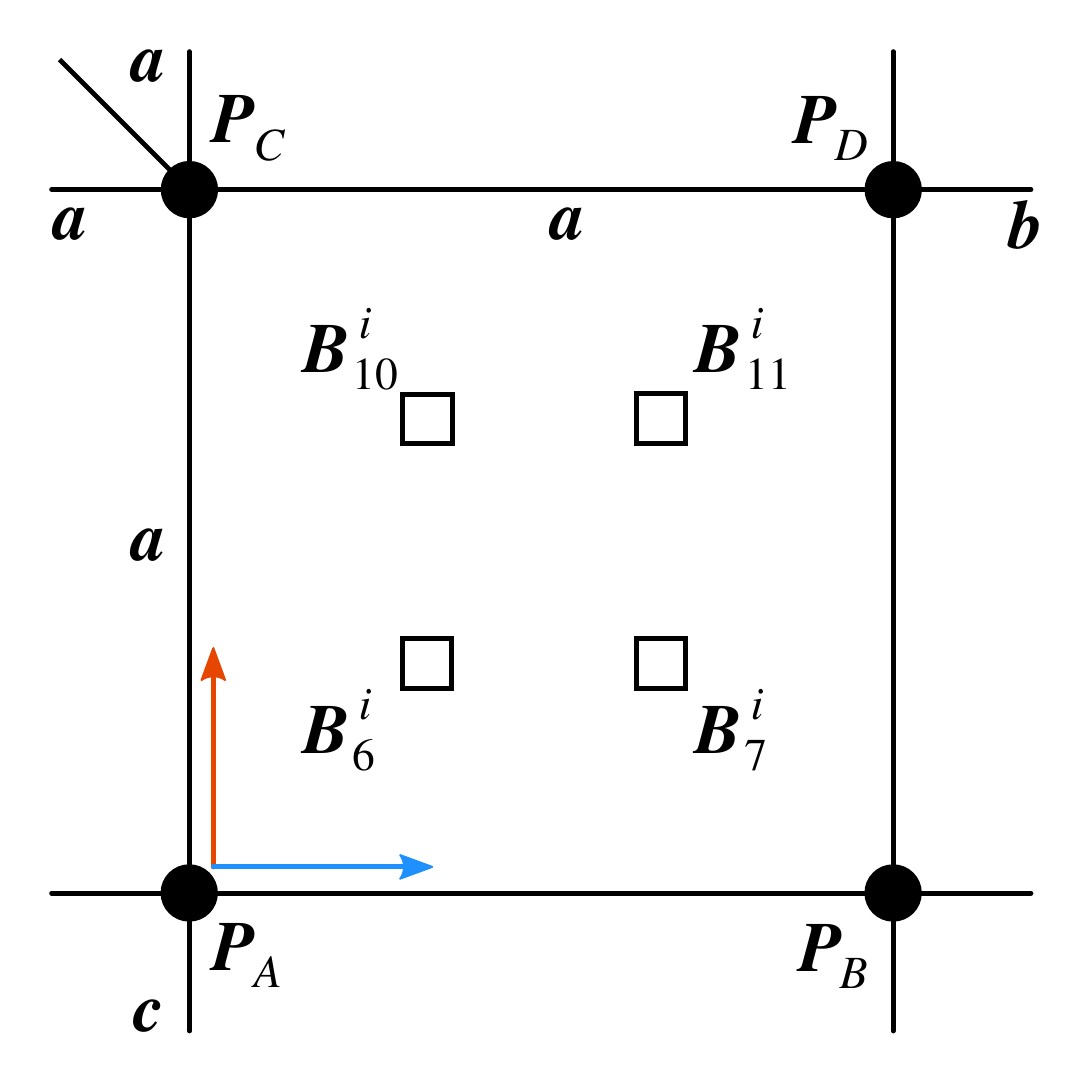}
\caption{Face with at least one EP. The edges emanating from the EP are required to have the same knot span $a$.} 
\label{aaaa}
\end{figure}

According to Fig. \ref{analysisspacemod2}, the relation between the face B\'ezier coefficients $B^i_6$, $B^i_7$, $B^i_{10}$, and $B^i_{11}$ of an irregular face and the face-based spline coefficients with indices $j_1$, $j_2$, $j_3$, and $j_4$ is the following
\begin{equation} \label{matrix55}
\left(
  \begin{array}{c}
    B^i_6 \\
    B^i_7 \\
    B^i_{10} \\
    B^i_{11} \\
  \end{array}
\right) = \left(
            \begin{array}{cccc}
              1 & 0 & 0 & 0 \\
              0 & 1 & 0 & 0 \\
              0 & 0 & 1 & 0 \\
              0 & 0 & 0 & 1 \\
            \end{array}
          \right)
          \left(
            \begin{array}{c}
              c_{j_1} \\
              c_{j_2} \\
              c_{j_3} \\
              c_{j_4} \\
            \end{array}
          \right).
\end{equation}
$\sum c_B M^0_B = 0$ implies that the face B\'ezier coefficients $B^i_6$, $B^i_7$, $B^i_{10}$, and $B^i_{11}$ are zero since the Bernstain polynomials are a basis. Thus, $ c_{j_k}= 0 \; \forall k \in \{1,2,3,4\}$ since the matrix in Eq. \ref{matrix55} is full rank, which means the spline coefficients $c_B$ associated with face-based functions are also zero. Therefore, $\{ M^0_B \}^{n_b}_{B=1}$ are linearly independent.

\end{proof}

\begin{theorem}\label{th:2}The blending functions of $\mathbb{S}^1_{D}$ and $\mathbb{S}^1_{A}$ are linearly independent.
\end{theorem}
\begin{proof}This theorem is a direct consequence of  Lemma \ref{lemma:1}, Lemma \ref{lemma:2}, and Lemma \ref{lemma:3}.
\end{proof}

Eqs \eqref{5theq} - \eqref{lasteq2} lead to basis functions that form a non-negative partition of unity \cite{toshniwal2017smooth, Borden2011}. Therefore, the basis functions of $\mathbb{S}^0_D$ form a non-negative partition of unity. As explained in Section 2.4.2, truncation is used so that the basis functions of $\mathbb{S}^0_A$ form a non-negative partition of unity. Since we are using a smoothing matrix that is affine-invariance and has non-negative entries, the split-then-smoothen approach results in basis functions that maintain a non-negative partion of unity \cite{toshniwal2017smooth}. Therefore, the basis functions of $\mathbb{S}^1_D$ and $\mathbb{S}^1_A$ form a non-negative partition of unity.

\section{Convergence rates}

\begin{figure} [t!] 
 \centering
\subfigure[]{\includegraphics[scale=0.05]{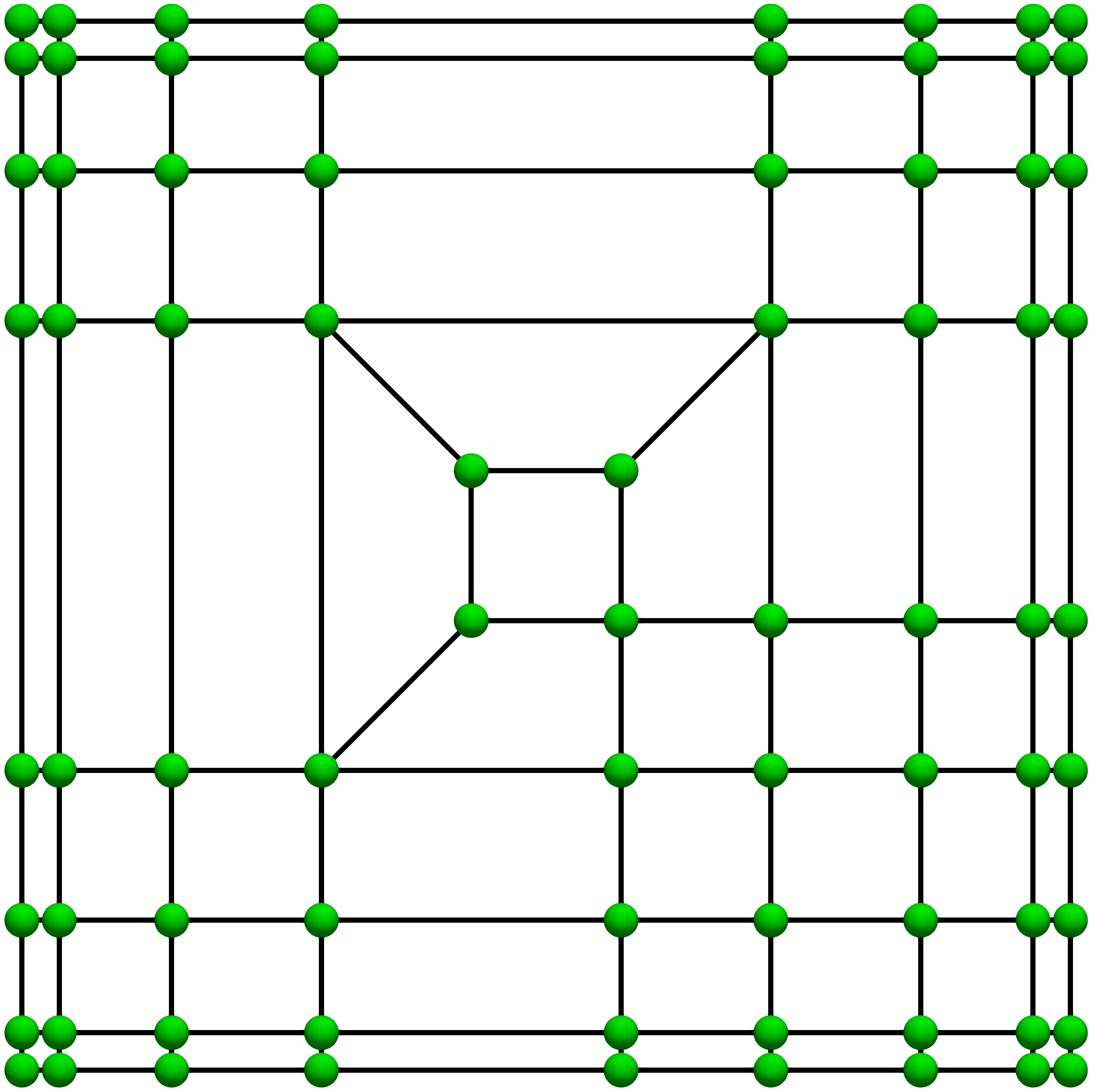}} \hspace*{+8mm}
 \subfigure[]{\includegraphics[scale=0.149]{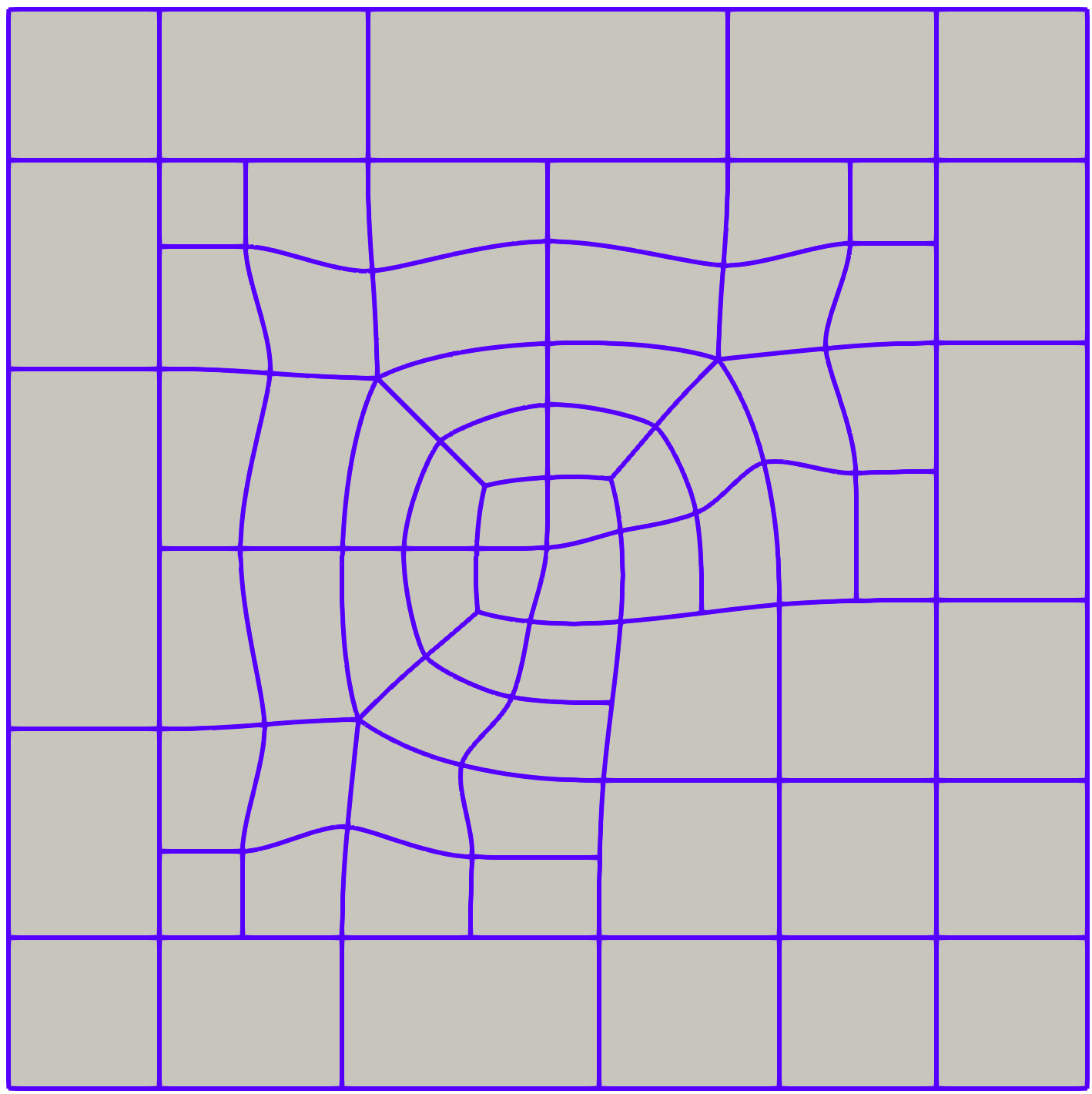}} \\
\caption{Unit square used to compute convergence rates. (a) Control net. (b) B\'ezier mesh. Both totally irregular faces and partially irregular faces are included in this geometry.}
\label{cnbm}
\end{figure}

When defining $C^1$-continuous basis functions around EPs, a difficult challenge is to obtain spaces with optimal approximation properties \cite{collin2016analysis, nguyen2014comparative}. We mesh a unit square ($\Omega = [0,1]^2$) introducing both totally irregular faces and partially irregular faces. The control net and the B\'ezier mesh are shown in Fig. \ref{cnbm} (a) and (b), respectively. After that, five levels of global uniform refinement are performed. We solve both the Poisson equation (second-order linear elliptic problem) and the biharmonic equation (fourth-order linear elliptic problem) using the approach of manufactured solutions.

\begin{figure} [t!] 
 \centering
\subfigure[]{\includegraphics[scale=0.53]{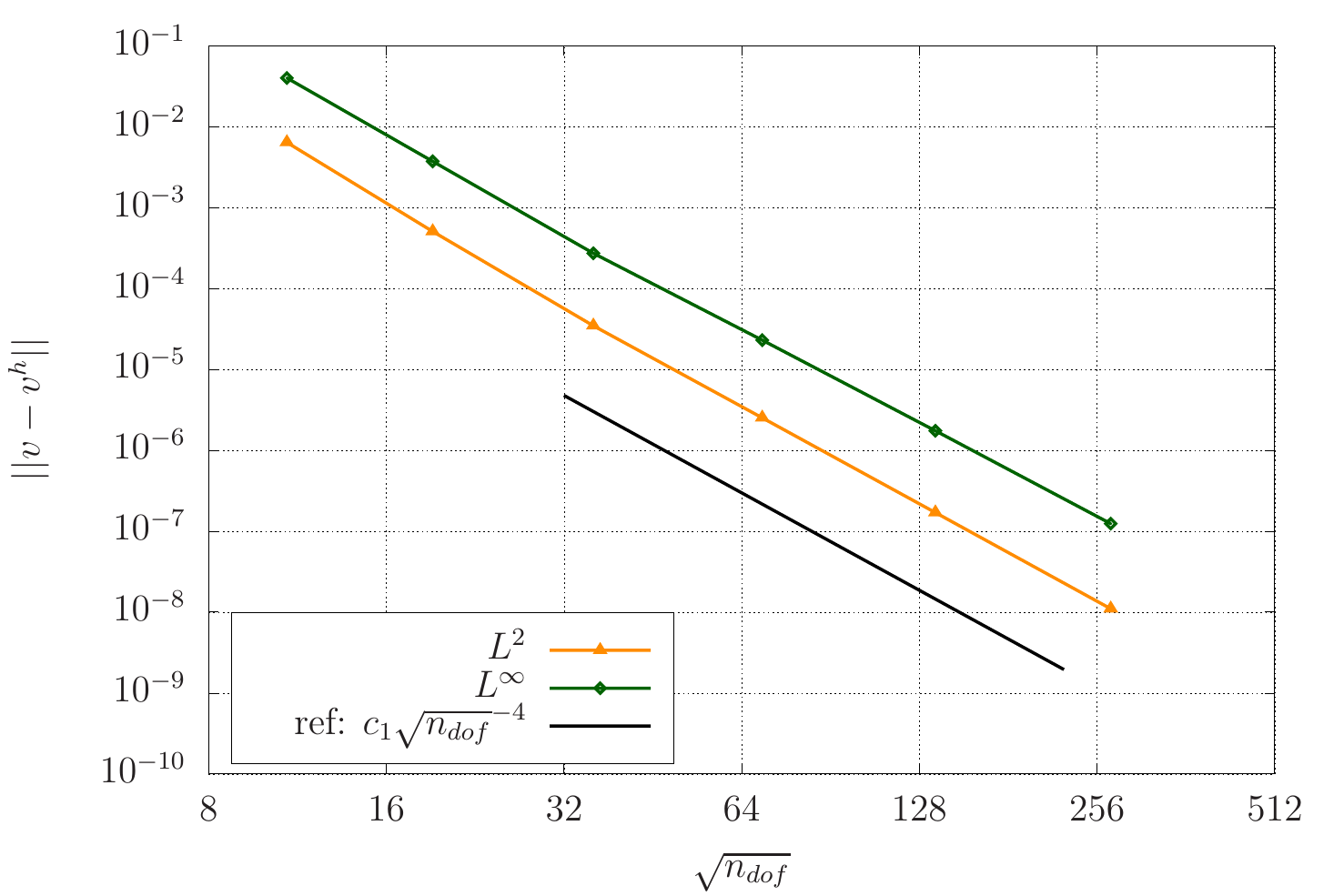}}
 \subfigure[]{\includegraphics[scale=0.53]{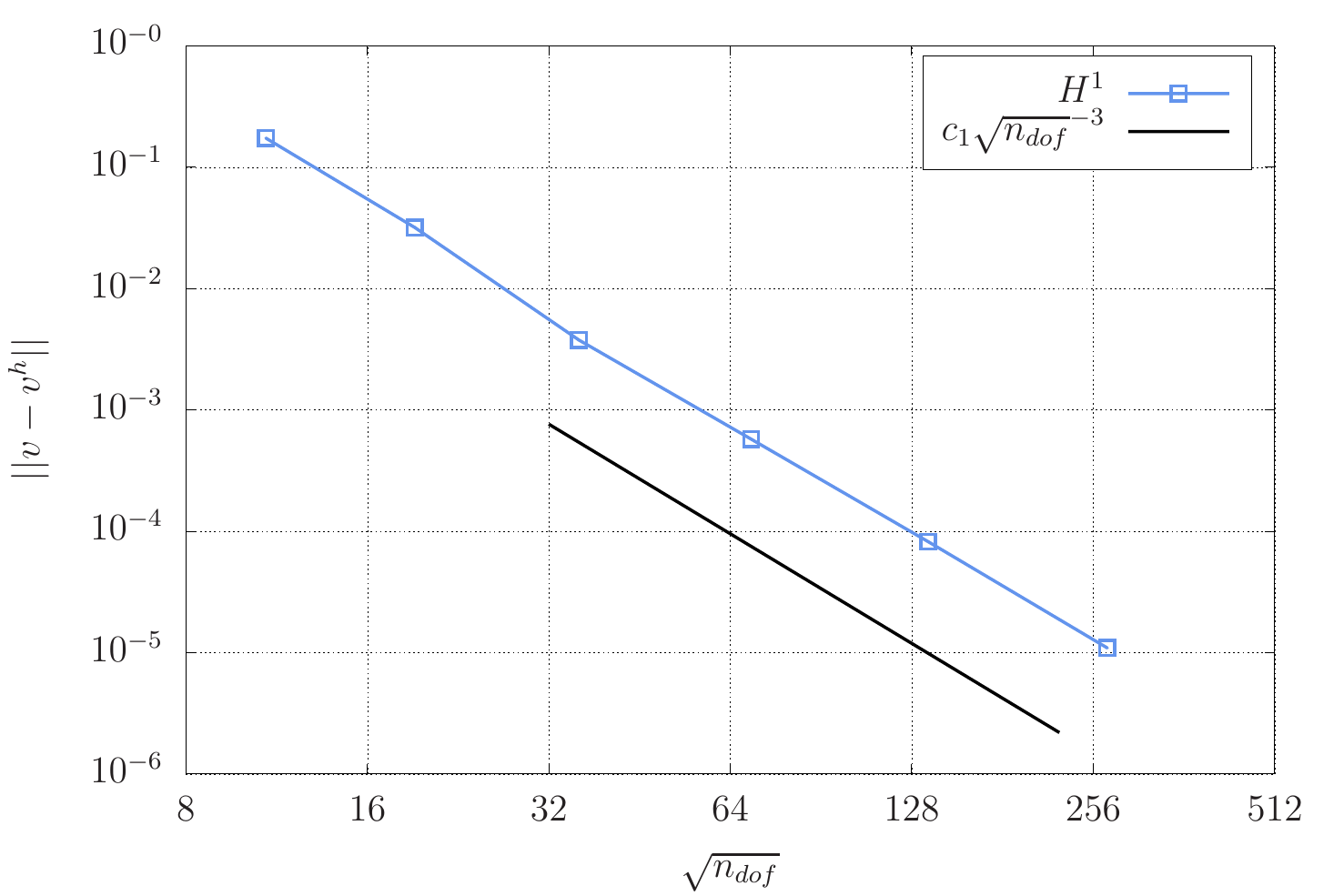}} \\
\caption{Poisson equation. (a) Convergence rates in $L^2$ and $L^{\infty}$ norms. (b) Convergence rate in $H^1$ norm.}
\label{poisson}
\end{figure}

The benchmark problem for the Poisson equation is defined as
\begin{align}
 & \Delta v  = g \quad & \text{in}      &  \quad \Omega  \text{,} \\ 
  & v  = 0 \quad & \text{in}      &  \quad \partial \Omega  \text{,}   \\
 & g  = - 2 \pi^2 \text{sin}(\pi x)\text{sin}(\pi y) \quad & \text{in}      &  \quad \Omega  \text{,} 
\end{align}
for which the exact solution is $v = \text{sin}(\pi x)\text{sin}(\pi y)$. The convergence rates in $L^2$, $L^{\infty}$, and $H^1$ norms for the four cases considered here are plotted in Fig. \ref{poisson} (b), (c), and (d), respectively. Optimal convergence rates are obtained.

\begin{figure} [t!] 
 \centering
\subfigure[]{\includegraphics[scale=0.53]{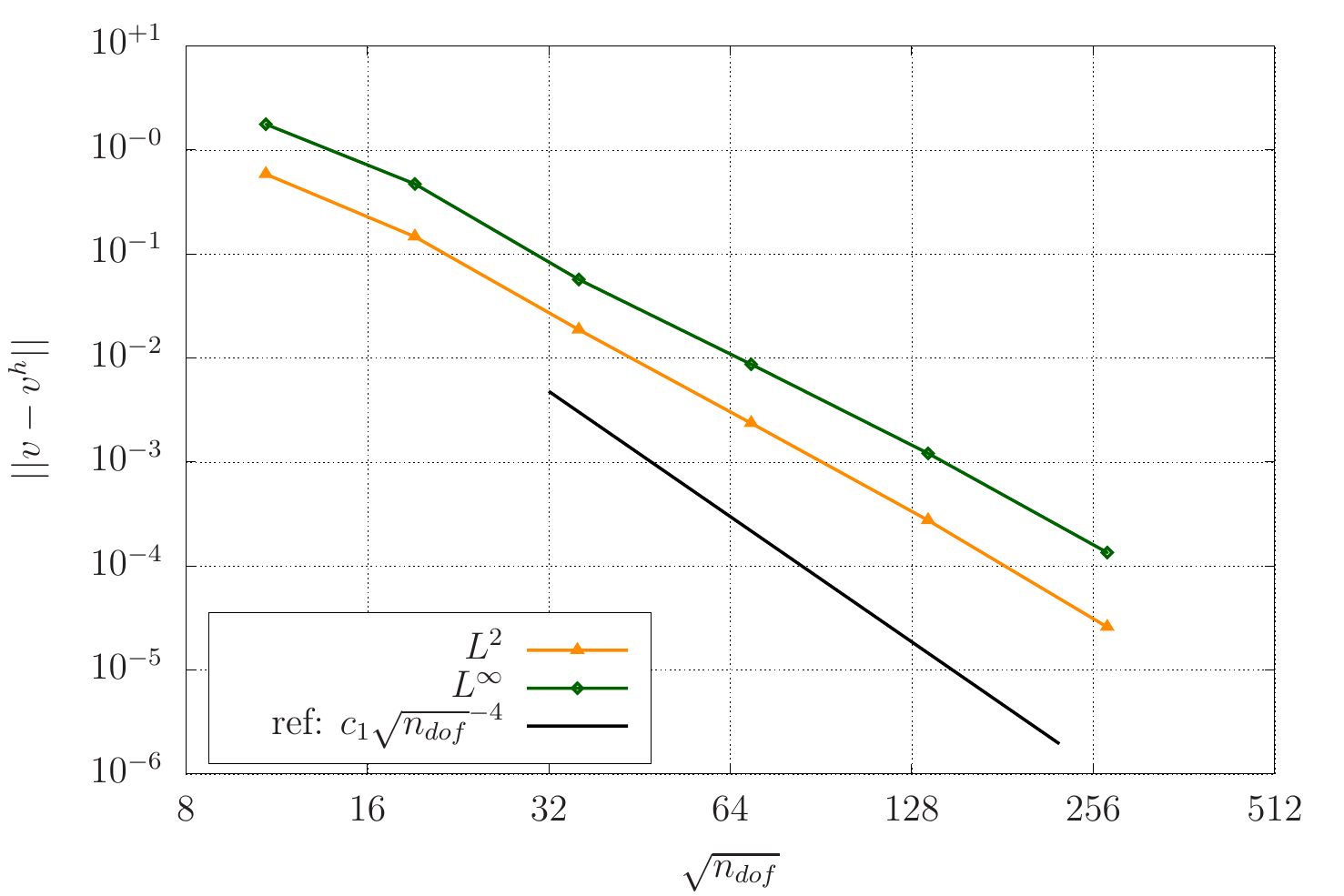}}
 \subfigure[]{\includegraphics[scale=0.53]{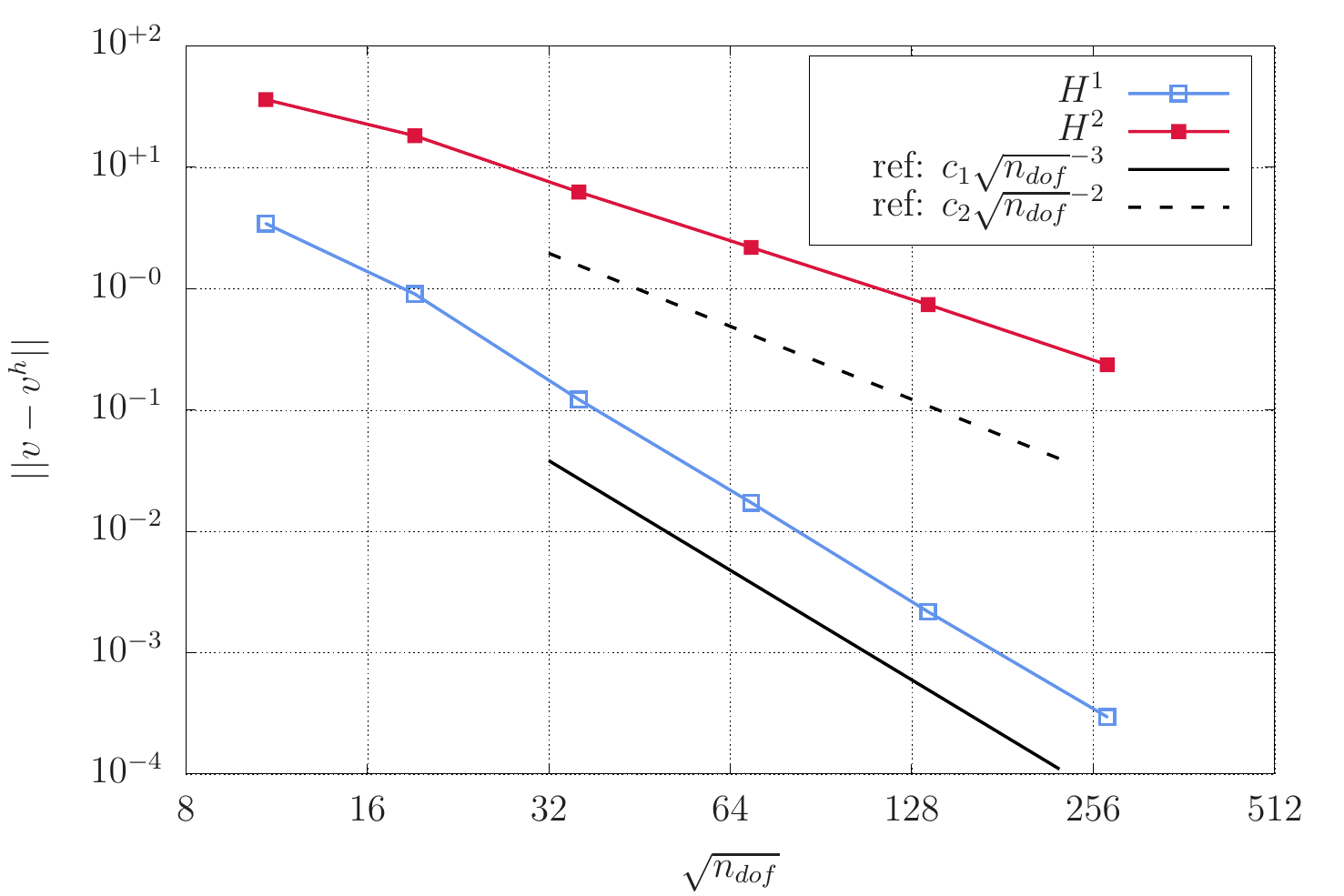}} \\
\caption{Biharmonic equation. (a) Convergence rates in $L^2$ and $L^{\infty}$ norms. (b) Convergence rates in $H^1$ and $H^2$ norms.}
\label{biharmonic}
\end{figure}

The benchmark problem for the biharmonic equation is defined as
\begin{align}
 & \Delta^2 v  = g \quad & \text{in}      &  \quad \Omega  \text{,} \\ 
  & v  = 0 \quad & \text{in}      &  \quad \partial \Omega  \text{,}   \\
 & \nabla v \cdot \mathbf{n} = 0 \quad & \text{in}      &  \quad \partial \Omega  \text{,}  \\
 & g  = - 16 \pi^4 (\text{cos}(2\pi x) - 4\text{cos}(2\pi x)\text{cos}(2\pi y) + \text{cos}(2\pi y)) \quad & \text{in}      &  \quad \Omega  \text{,} 
\end{align}
for which the exact solution is $v=(1 - \text{cos}(2\pi x))(1 - \text{cos}(2\pi x))$. The convergence rates in $L^2$, $L^{\infty}$, $H^1$, and $H^2$ norms for the four cases considered here are plotted in Fig. \ref{biharmonic} (a), (b),(c), and (d), respectively. Optimal convergence rates are obtained.

\section{Automotive applications}

In this section, we begin designing two structural components of an automobile using T-splines with the commercial software Autodesk Fusion360. Since Fusion360 currently handles EPs in a way that does not lead to analysis-suitable spaces, the blending functions created by Fusion360 should not be used in structural analysis. A workaround is to export the control net from Fusion360 and use those control points as the control points of $\mathbb{S}^1_{D}$, i.e., combine the control points of Fusion360 with the basis functions of AST-splines developed in Section 2.4. We observe that the surface stays essentially unchanged for complex geometries such as the B-pillar and the side outer panel of a car (the area change is lower than $0.1\%$). This workflow is exemplified in Fig. \ref{workflow} using the geometry of the side outer panel. In any case, this is a temporary workaround until there are CAD programs that design surfaces from scratch using AST-splines. 

\begin{figure} [t!] 
 \centering
\subfigure[T-spline surface]{\includegraphics[scale=0.38]{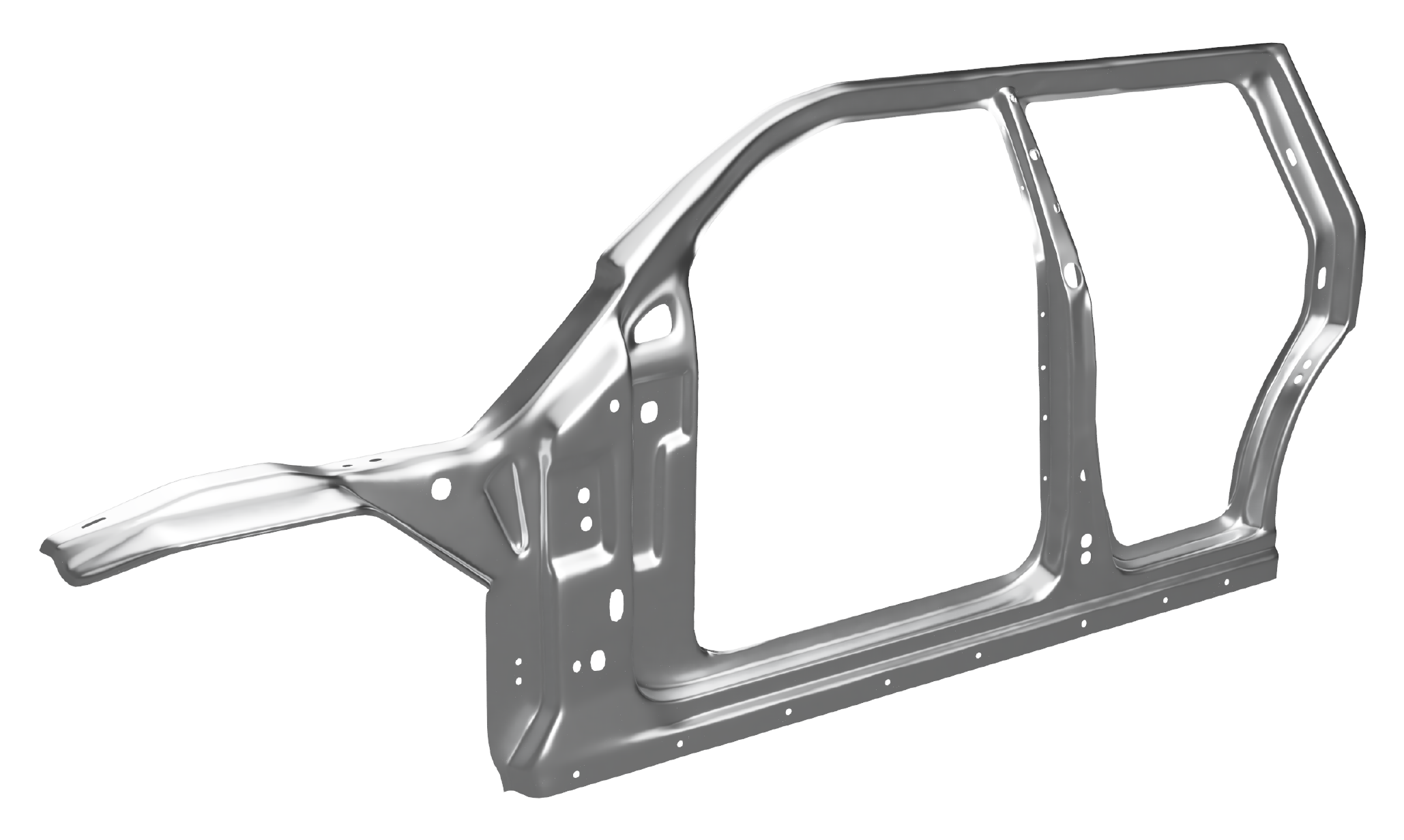}}
 \subfigure[Control net]{\includegraphics[scale=0.38]{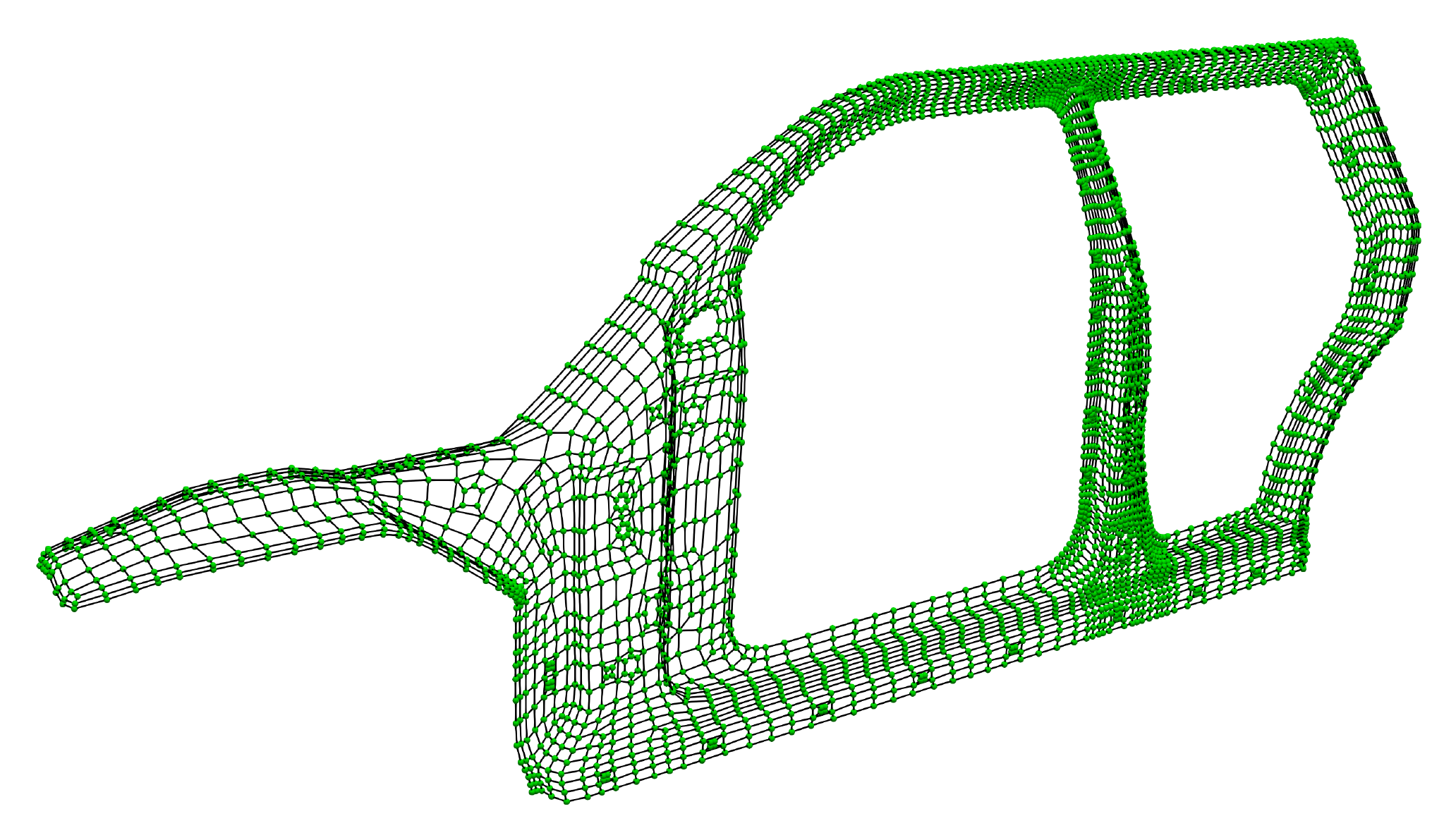}} \\
  \subfigure[AST-spline surface]{\includegraphics[scale=0.38]{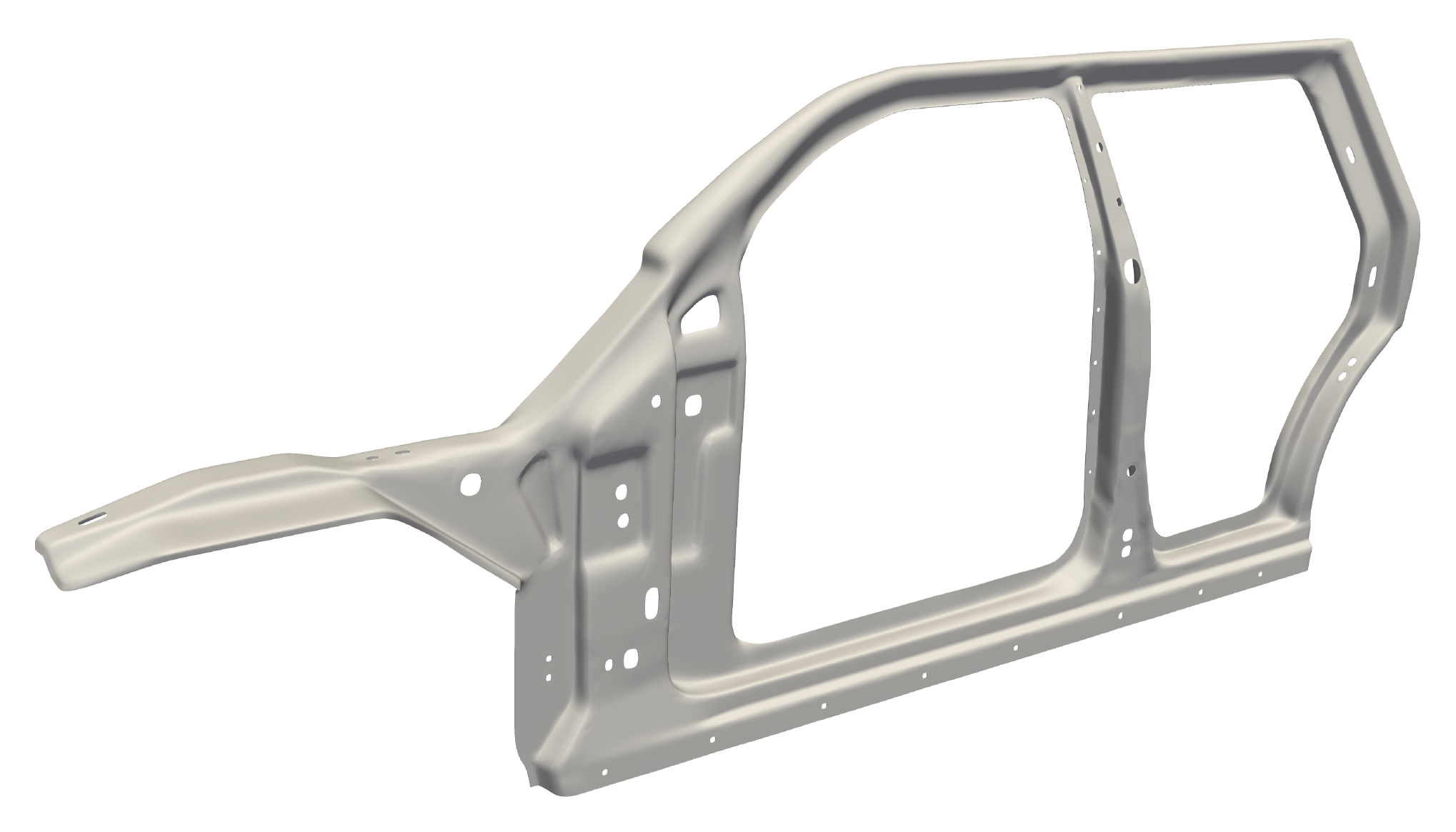}}
 \subfigure[EP layout]{\includegraphics[scale=0.38]{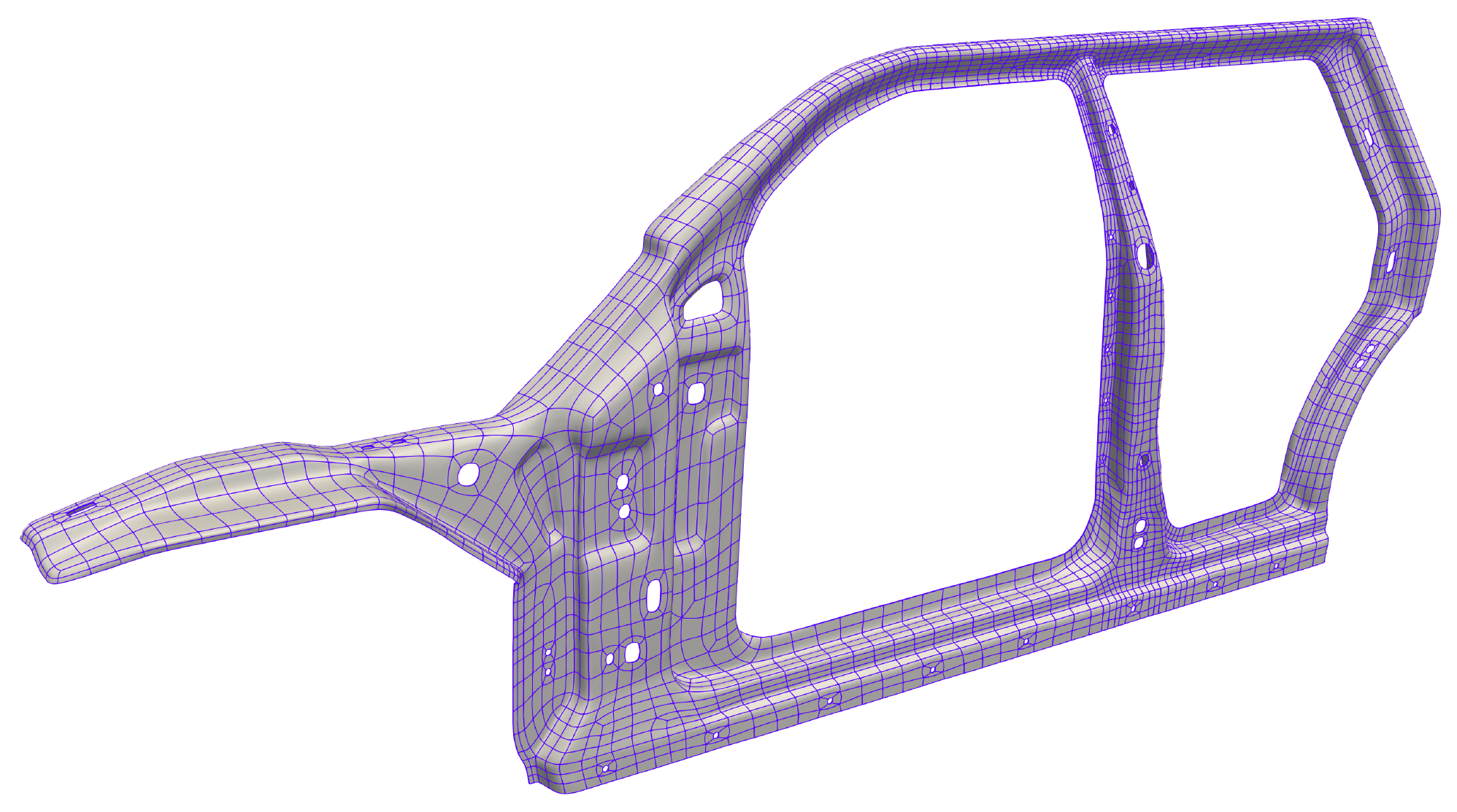}} \\
\caption{Side outer panel. (a) T-spline surface in rendered view designed from scratch in Autodesk Fusion360. (b) Control net exported from Autodesk Fusion360 associated with the T-spline surface shown in (a). (c) AST-spline surface computed from the control net shown in (b). (d) AST-spline surface with the face boundaries plotted on top of it to show the distribution of EPs throughout the geometry.}
\label{workflow}
\end{figure}

In \cite{casquero2020seamless}, we used our AST-spline surfaces to perform geometrically nonlinear Kirchhoff-Love shell simulations. Here, we take a different path. We thicken our AST-spline surfaces using cubic B-splines in the thickness direction. The result is an AST-spline volume with $C^1$ global continuity for a given thickness value. AST-spline volumes, as any other type of splines that admit B\'ezier extraction, can be imported in the commercial software LS-DYNA. In the following, we solve eigenvalue problems with our AST-spline volumes in LS-DYNA and compare the results with various conventional solid discretizations based on trilinear hexahedral meshes available in LS-DYNA.

The use of solid formulation for thin-walled structures as opposed to shell formulations is becoming more common in automotive applications. One of the main reasons is that Kirchhoff-Love shells, Reissner-Mindlin shells, and also other types of shells with more complex formulations in the thickness direction often fail to accurately capture the stress triaxiality of the thin-walled structure. Stress triaxiality plays a key role in predicting ductile fracture. As a result, solid formulations are often found to match experimental data in crash simulations significantly better than shell formulations. This is the motivation behind considering AST-spline volumes in this work.

\subsection{B-pillar}

\begin{figure} [t!] 
 \centering
\includegraphics[scale=0.4]{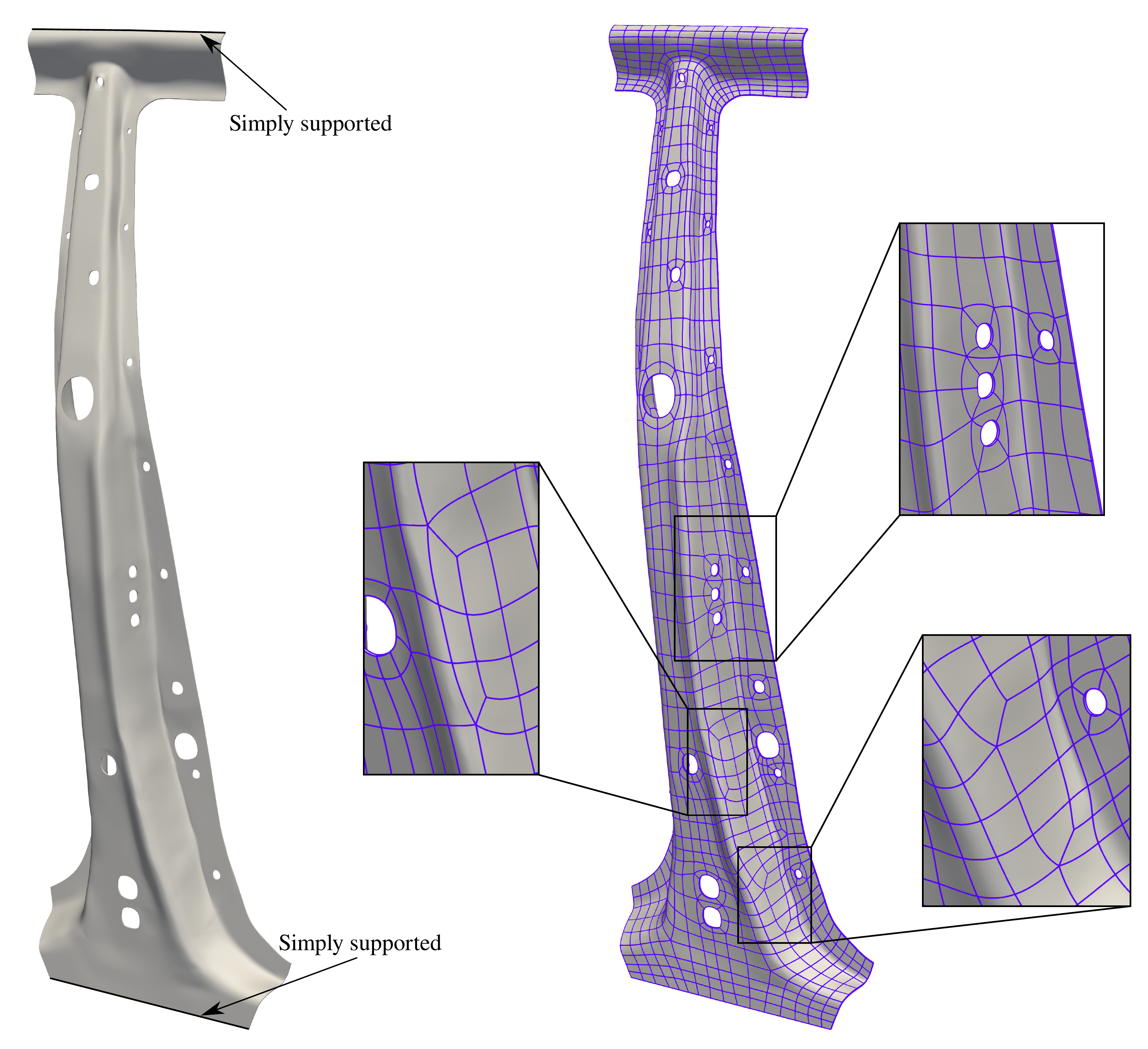}
\caption{Geometry and boundary conditions for the B-pillar. The AST-spline volume with the face boundaries colored in violet shows the positions of the EPs throughout the geometry.}
\label{bpillarbc}
\end{figure}

The geometry considered in this example is the B-pillar of an automobile. The boundary conditions are indicated in Fig. \ref{bpillarbc}. The parameters used in this problem are
\begin{equation}    \label{101}
t = 1.6 \text{,}  \quad  E = 2.1 \times 10^5  \text{,} \quad  \nu = 0.25  \text{,} 
\end{equation}
where $t$ is the thickness, $E$ is the Young modulus, and $\nu$ is the Poisson ratio. The geometry has 26 holes. The AST-spline surface contains 8 EPs with valence 3, 104 EPs with valence 5, and 7 EPs with valence 6. The AST-spline volume with one cubic B-spline element in the thickness direction has 11,980 ($2,995 \cdot 4$) control points. We compute the lowest eigenvalue with AST-splines and perform various comparisons with conventional finite elements below.

   \begin{table}[t!]
   \caption{Lowest eigenvalue for different number of elements in the thickness direction. The AST-spline has one element along the midsurface. The finite-element mesh has five elements along the midsurface.} \label{tablebpillar}
   \bigskip
     \centering
     \begin{tabular}{|c|c|c|c|c|}
\hline
  & AST-splines & ELFORM 2 & ELFORM 1 &  ELFORM -2 \\
\hline
1 element  & 2.6460e+21  & 9.1197e+21 & 8.4638e+20 & 4.1211e+21  \\    
\hline       
2 elements  & 2.6460e+21  & 9.1672e+21 & 2.5043e+21 & 4.5127e+21 \\    
\hline   
3 elements  & 2.6460e+21  & 9.1759e+21 & 2.7425e+21 & 4.6020e+21 \\    
\hline       
4 elements  & 2.6460e+21  & 9.1790e+21 & 2.8229e+21 & 4.6407e+21 \\    
\hline  
5 elements  & 2.6460e+21  & 9.1804e+21 & 2.8607e+21 & 4.6635e+21 \\    
\hline           
     \end{tabular}     
   \end{table}

\begin{figure} [t!] 
\centering
\includegraphics[width=9cm]{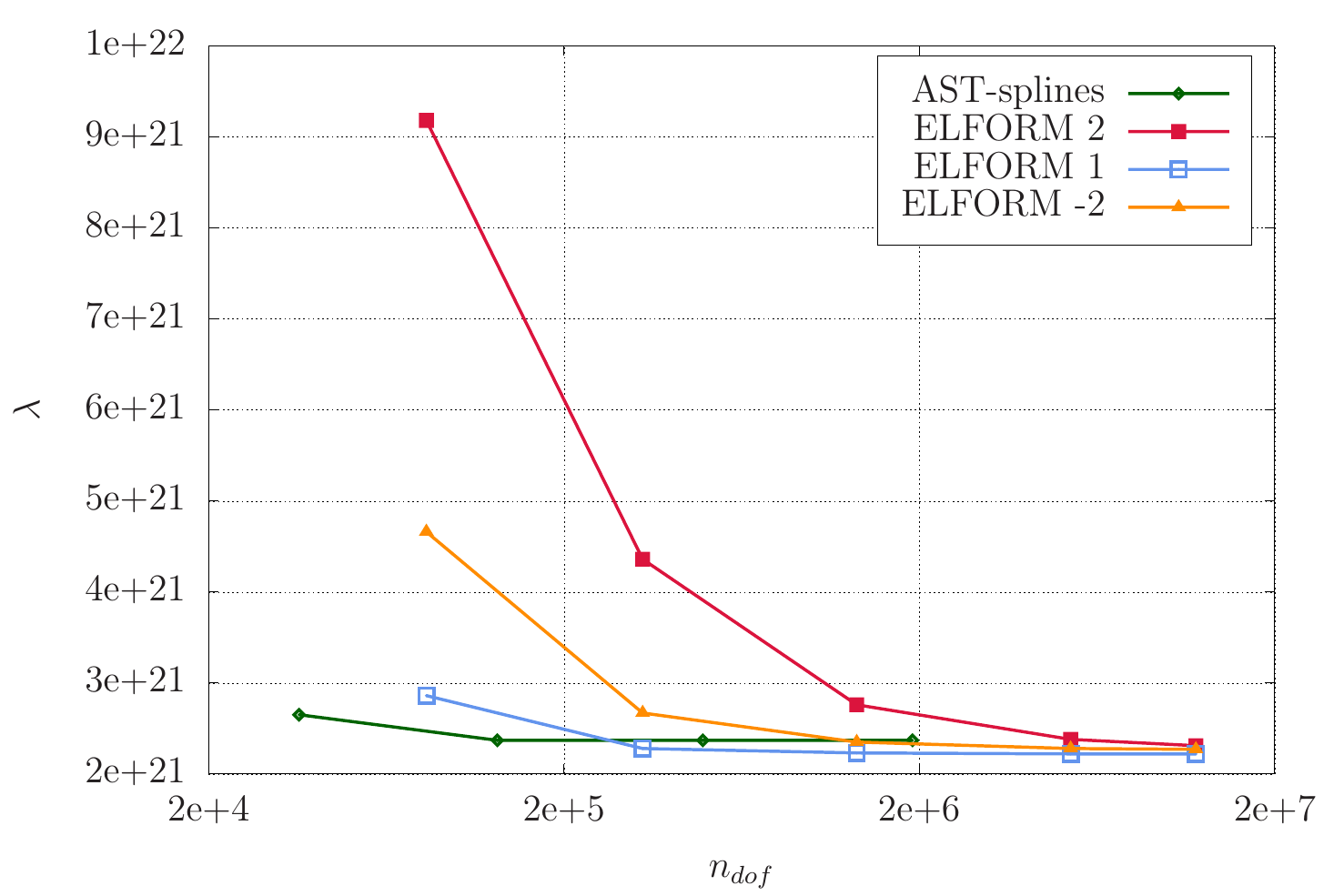}
\caption{(Color online) Lowest eigenvalue using AST-splines and conventional finite elements ELFORM 2, 1, and -2. For coarse discretizations, AST-splines are closer to the converged result than conventional finite elements.} 
\label{bpillareigenvalue}
\end{figure}

First of all, we study how many elements are needed in the thickness direction to reach a solution independent of the number of elements employed in this direction. We check this for both AST-spline volumes and the conventional finite element formulations available in LS-DYNA through the options ELFORM 2, ELFORM 1, and ELFORM -2. ELFORM 2 uses trilinear hexahedral elements with eight quadrature points. ELFORM 1 uses trilinear hexahedral elements with one quadrature point to alleviate shear locking. ELFORM -2 uses trilinear hexadral elements with eight quadrature points and an assumed strain approach to alleviate shear locking. As shown in Table \ref{tablebpillar}, one element is enough for AST-splines while at least five elements are needed for conventional finite elements. We now study the resolution needed in the two surface directions to obtain a converged result with both AST-spline volumes (one element in the thickness direction) and conventional finite elements (five elements in the thickness direction). As shown in Fig. \ref{bpillareigenvalue}, AST-spline volumes reach a converged result with significantly fewer degrees of freedom than conventional finite element discretizations.

\begin{figure} [t!] 
 \centering
\subfigure[Geometry]{\includegraphics[scale=0.4]{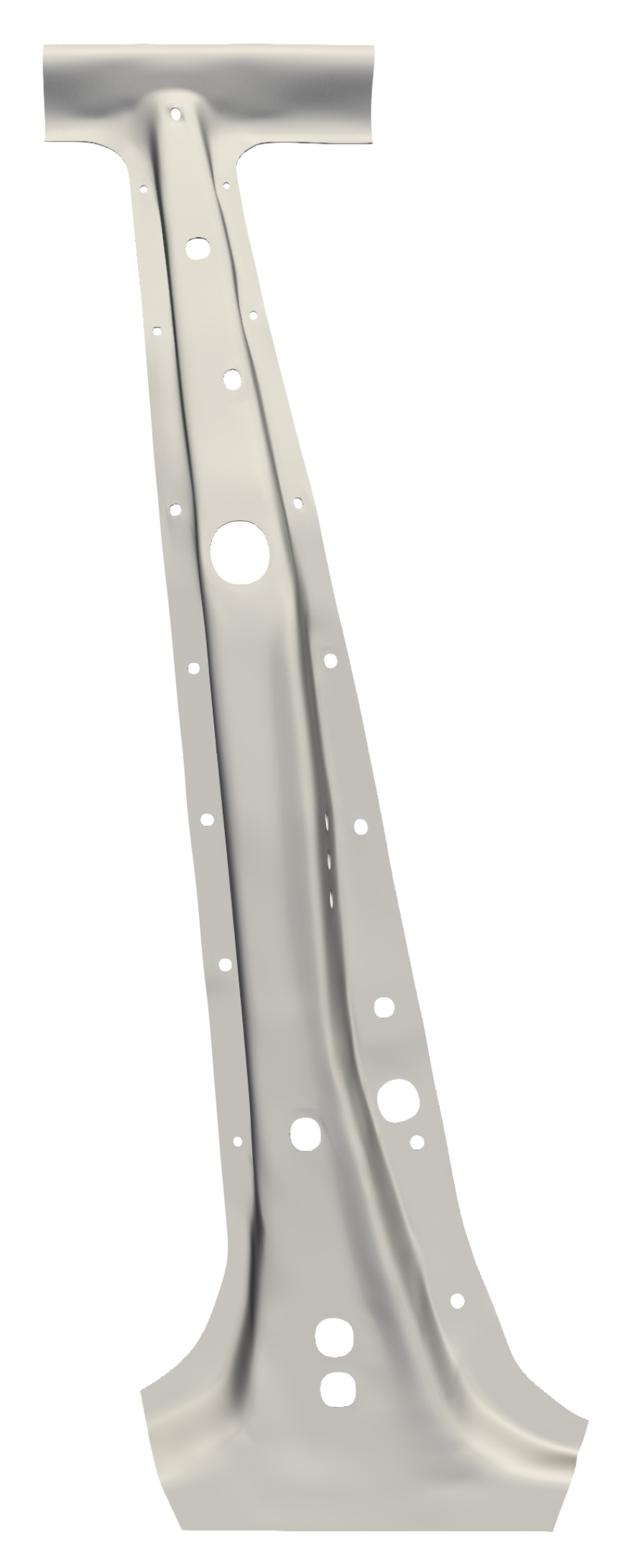}}
 \subfigure[IGA result]{\includegraphics[scale=0.4]{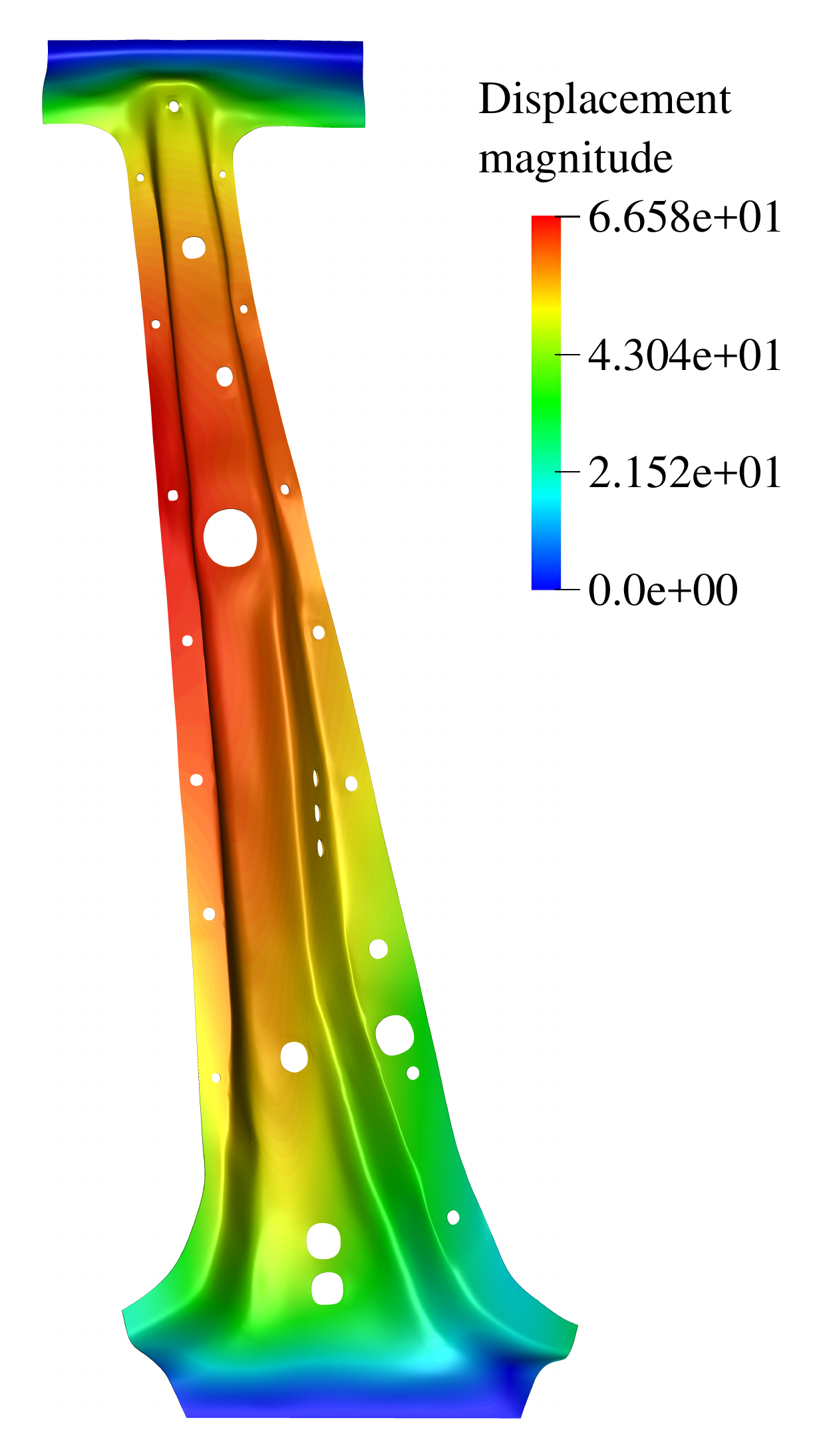}}
 \subfigure[FEM result]{\includegraphics[scale=0.4]{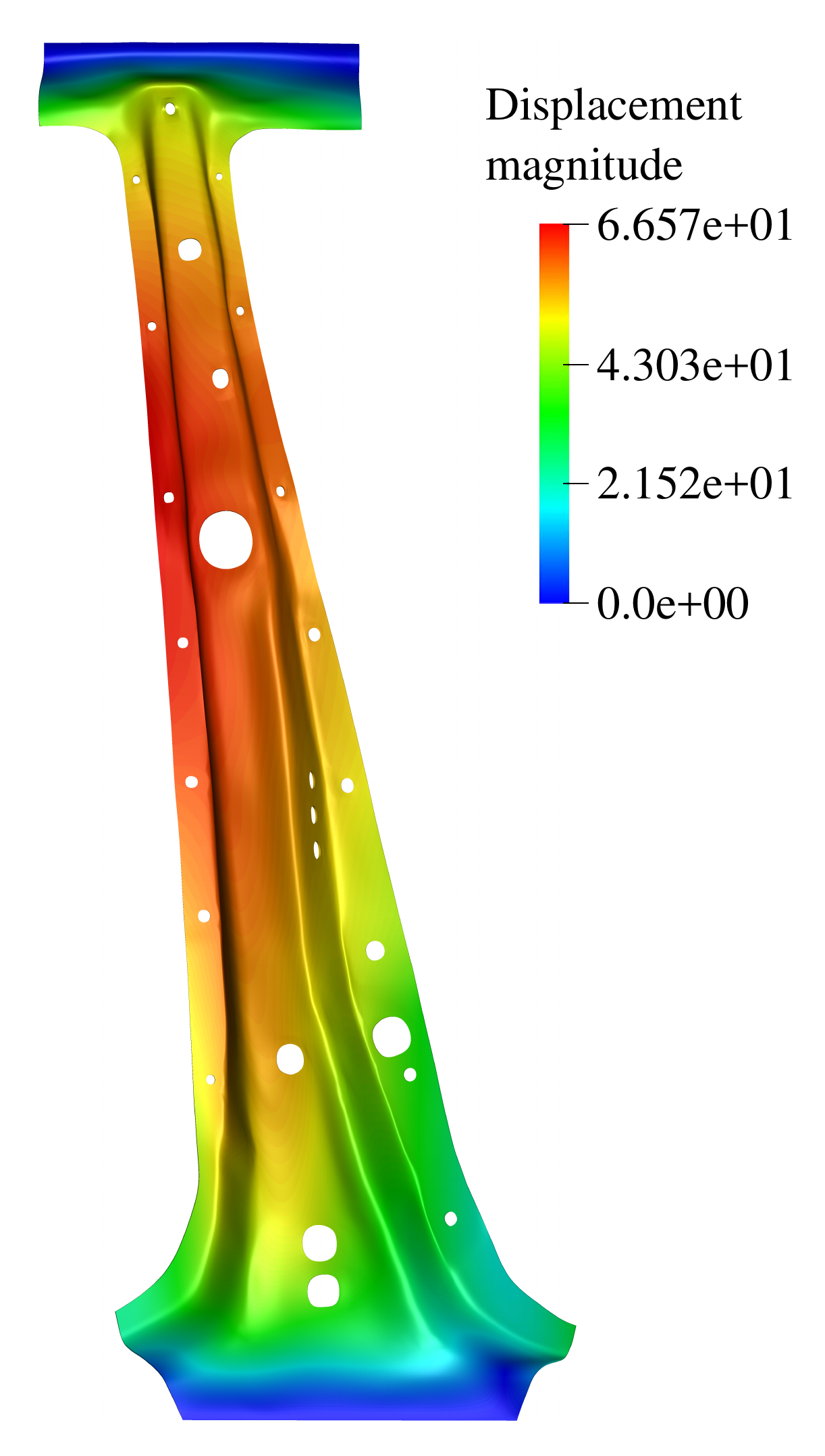}}
\caption{B-pillar. (a) Undeformed Geometry. (b) The first mode shape using AST-splines. (c) The first mode shape using conventional finite elements.}
\label{bpillarresults}
\end{figure}

In Figs. \ref{bpillarresults} (a)-(c), we plot front views of the underformed geometry, the first mode shape using AST-splines, and the first mode shape using trilinear hexahedral elements with eight quadrature points (ELFORM 2 in LS-DYNA), respectively.

\subsection{Side outer panel}

\begin{figure} [t!] 
 \centering
\includegraphics[scale=0.4]{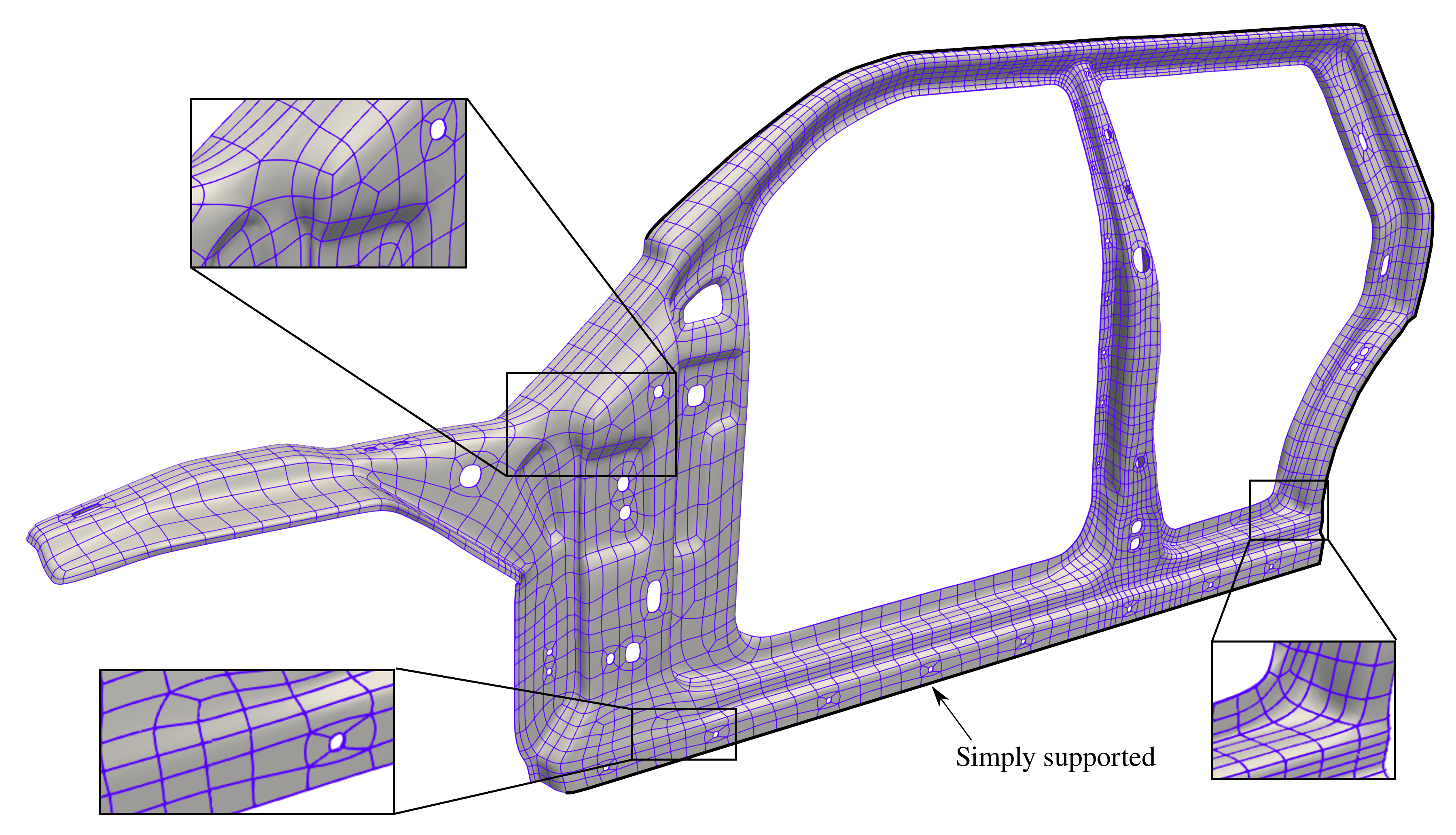}
\caption{Geometry and boundary conditions for the side outer panel. The AST-spline volume with the face boundaries colored in violet shows the positions of the EPs throughout the geometry.}
\label{sopbc}
\end{figure}

The side outer panel is the geometry considered in this example. The boundary conditions are indicated in Fig. \ref{sopbc}. The values of the parameters used in this problem are
\begin{equation}    \label{101}
t = 1.6 \text{,}  \quad  E = 2.1 \times 10^5  \text{,} \quad  \nu = 0.25  \text{,} 
\end{equation}
The geometry has 52 holes. The AST-spline surface contains 46 EPs with valence 3, 222 EPs with valence 5, and 19 EPs with valence 6. The AST-spline volume with one cubic B-spline element in the thickness direction has  31,588 ($7,897 \cdot 4$) control points. Fig. \ref{sopbc} shows the AST-spline volume with the face boundaries in violet color. Note that the 1-ring faces of EPs are divided into four B\'ezier elements, but we do not plot these additional lines in Fig. \ref{sopbc} so that the arrangement of the EPs throughout the geometry can be observed clearly.

\begin{figure} [t!] 
 \centering
 \subfigure[]{\includegraphics[scale=0.28]{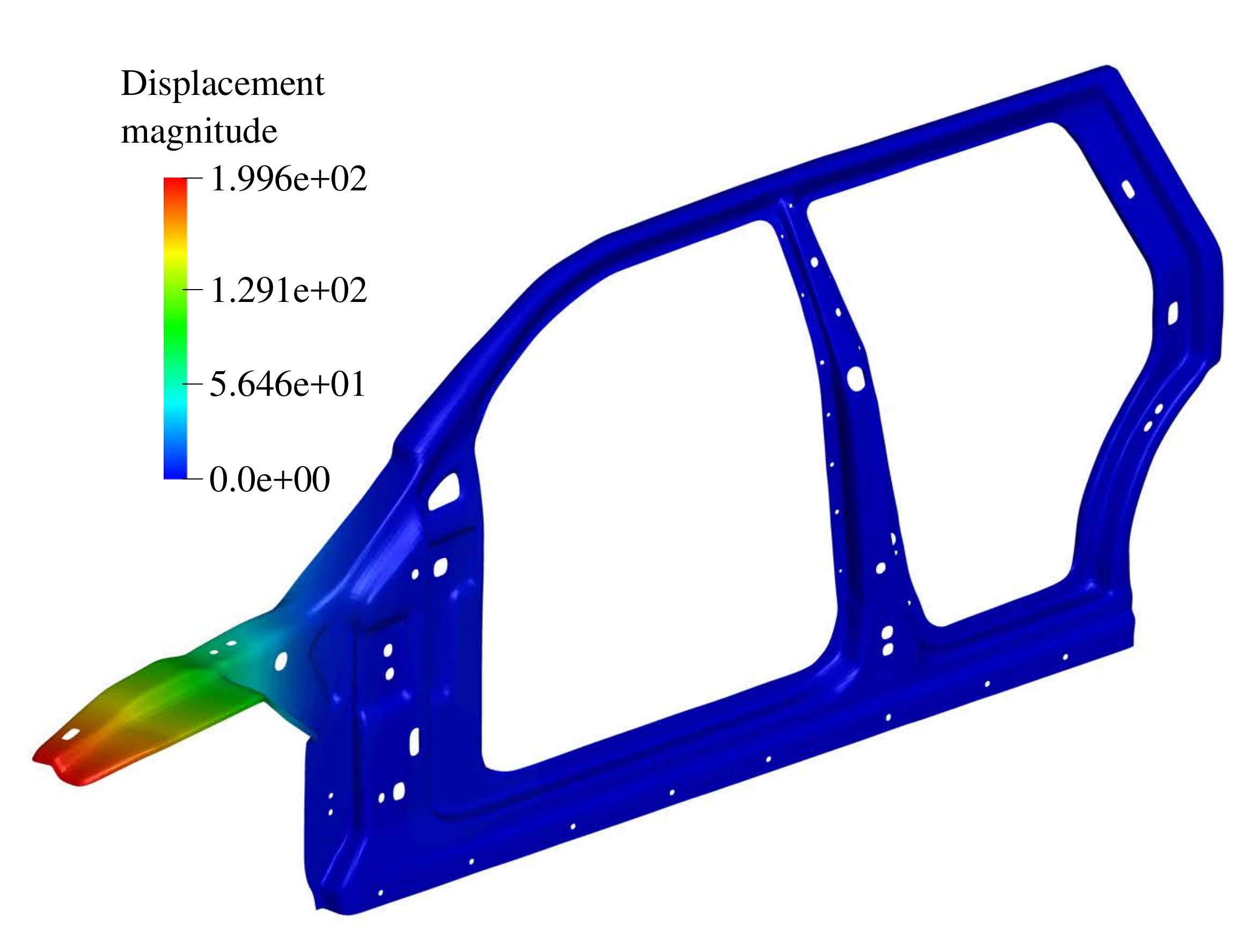}}
 \subfigure[]{\includegraphics[scale=0.28]{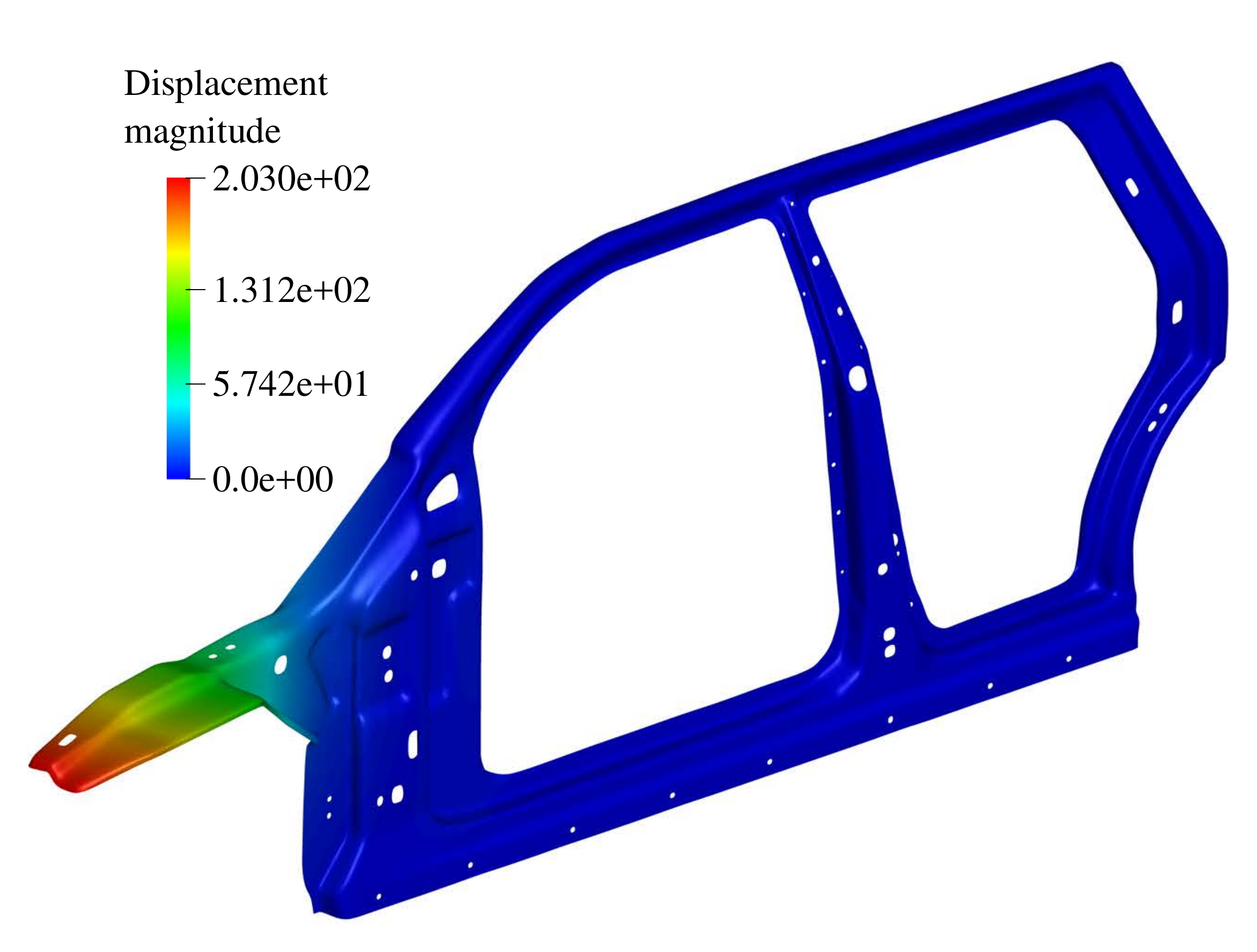}} \\
 \subfigure[]{\includegraphics[scale=0.28]{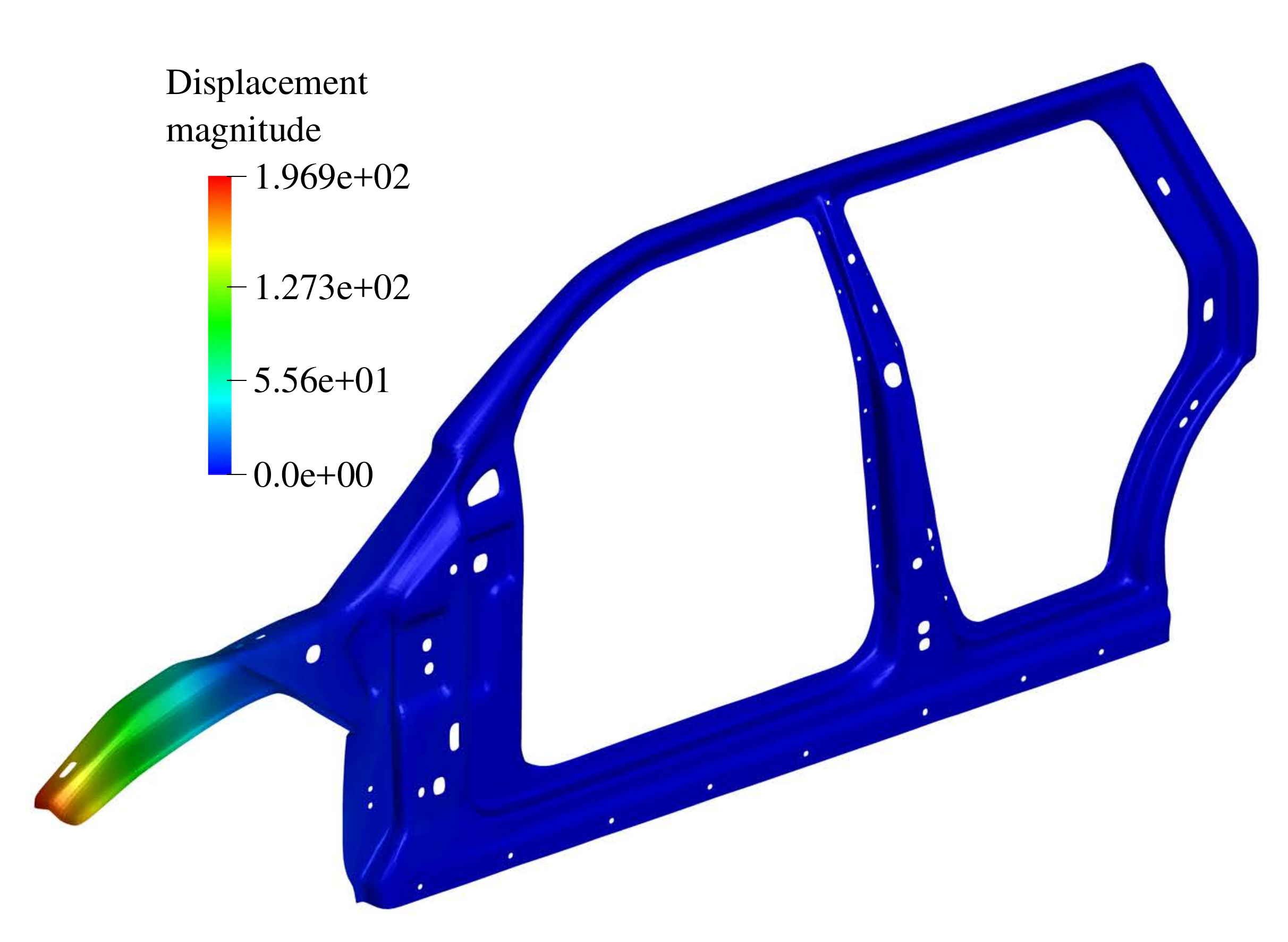}}
 \subfigure[]{\includegraphics[scale=0.28]{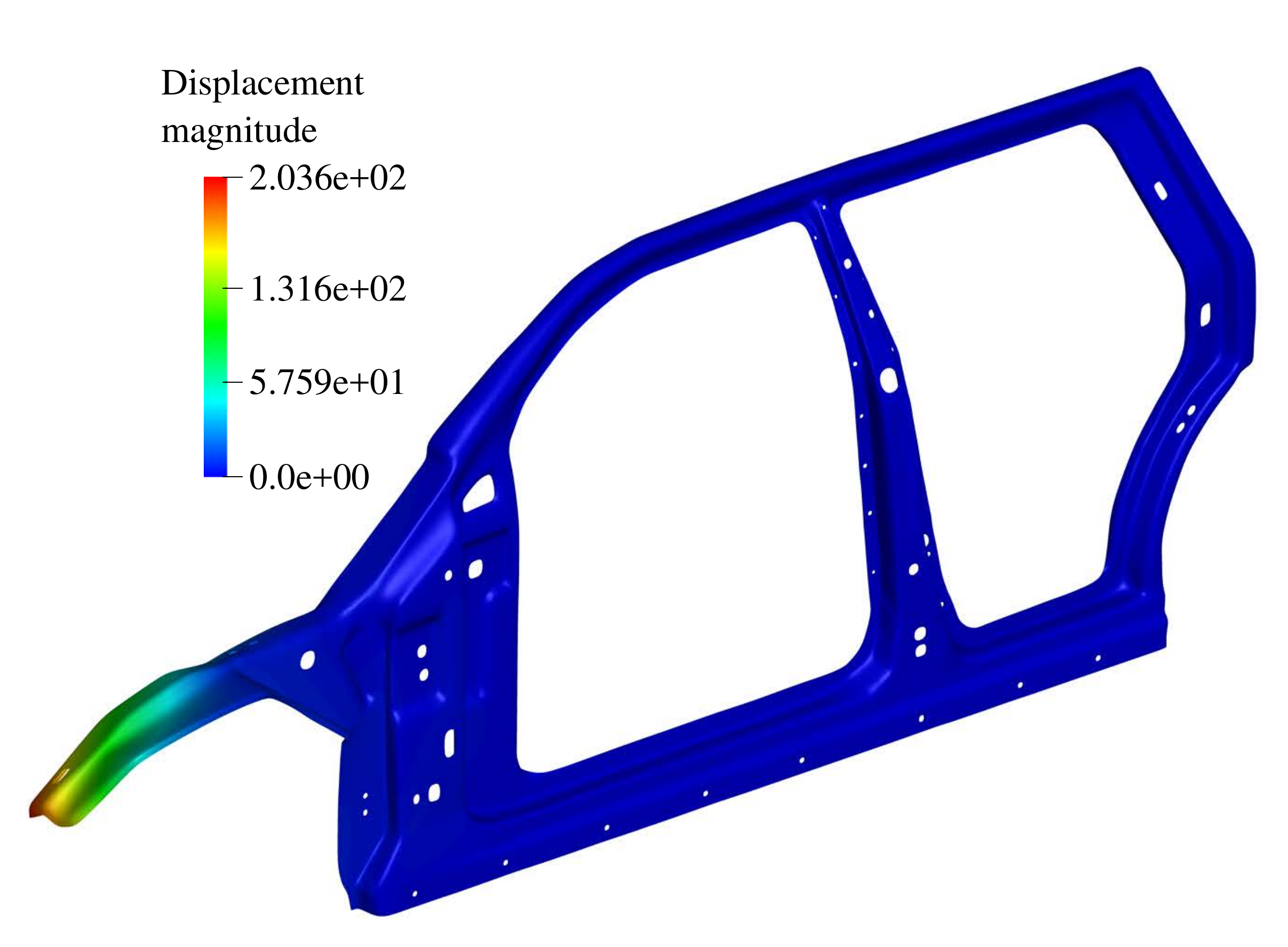}} \\
 \subfigure[]{\includegraphics[scale=0.28]{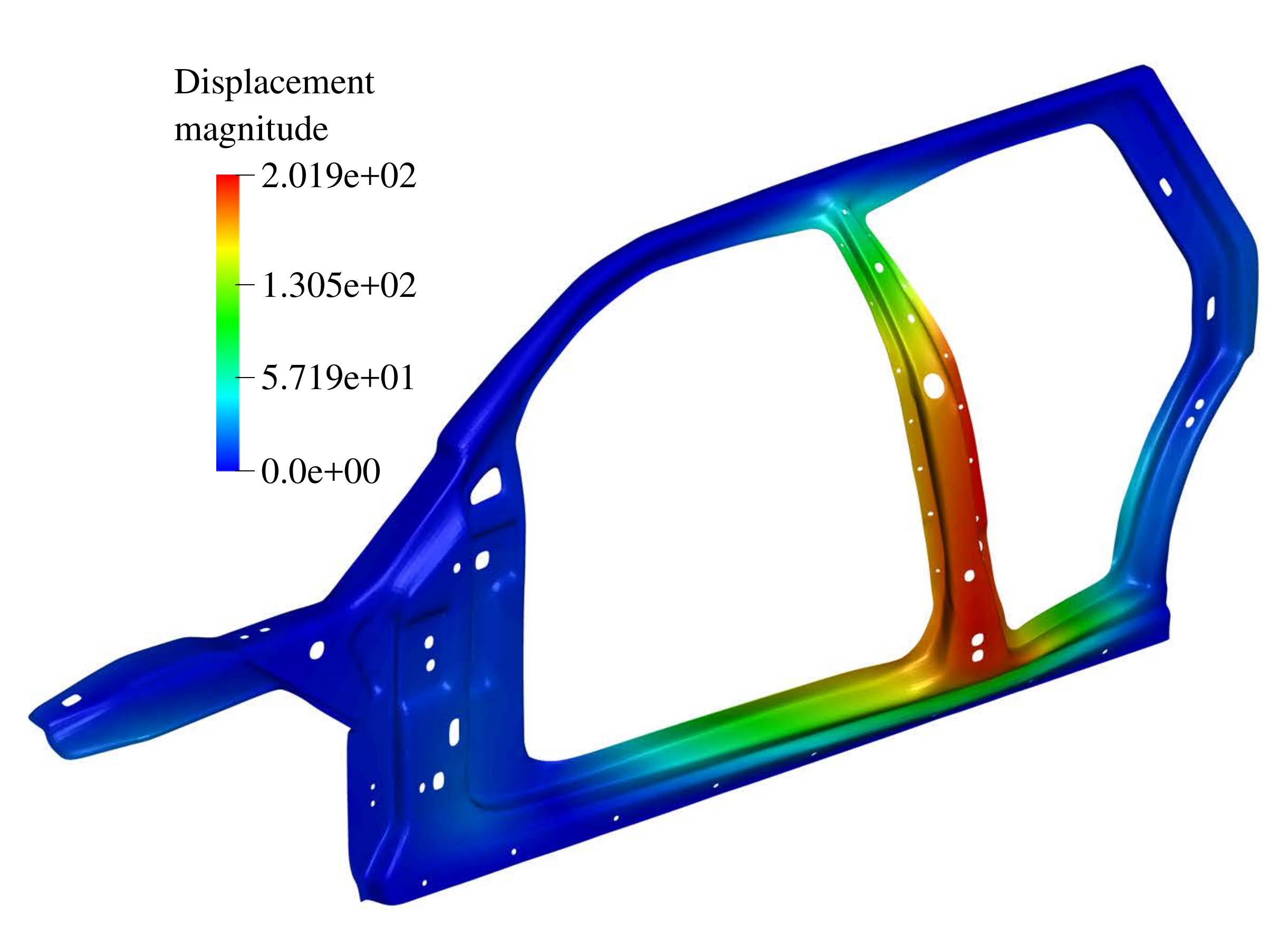}}
 \subfigure[]{\includegraphics[scale=0.28]{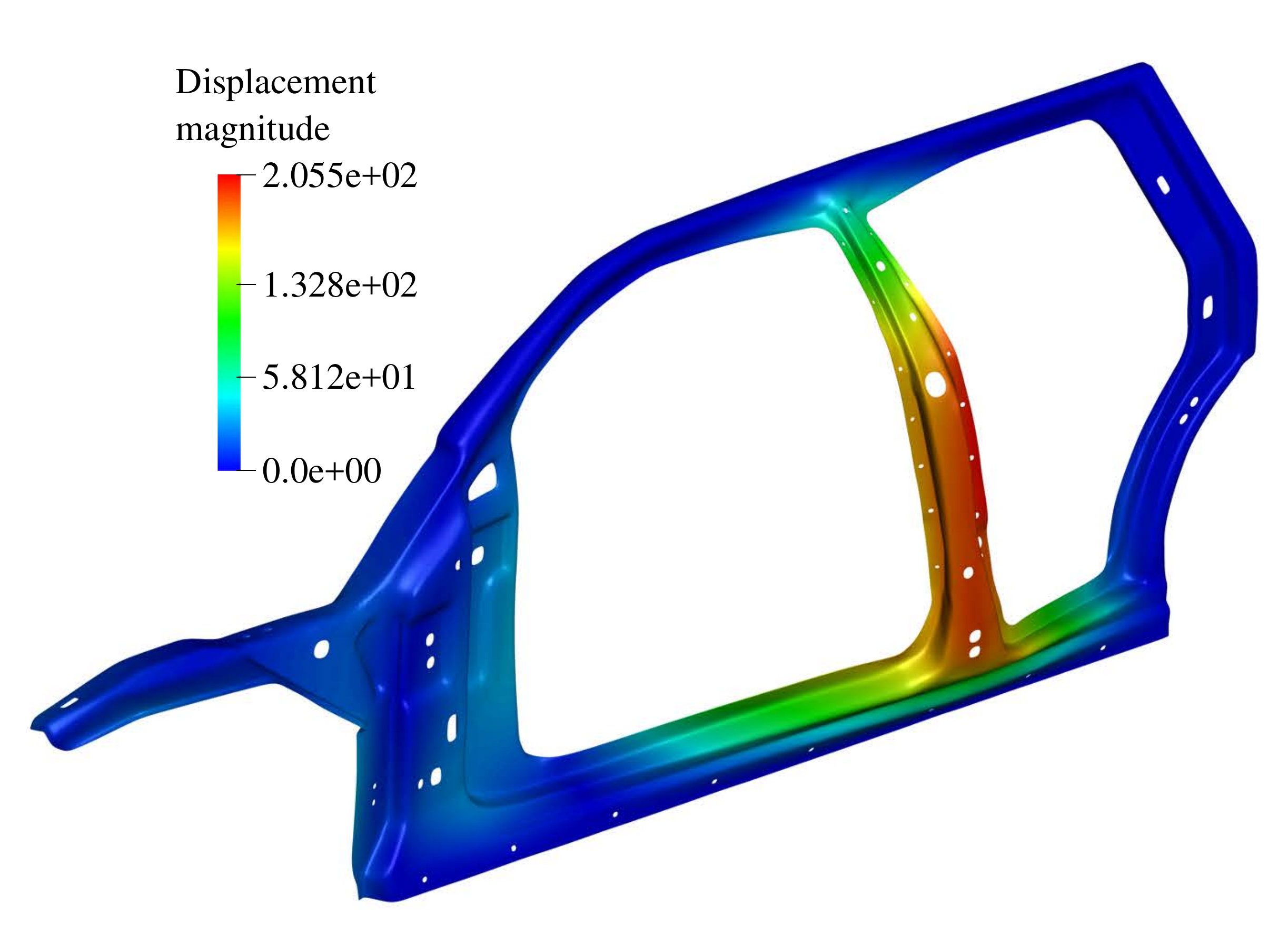}} \\
\caption{Side outer panel. (a), (c), and (e) show the first, second, and fifth mode shapes, respectively, using AST-splines. (b), (d), and (f) show the first, second, and fifth mode shapes, respectively, using conventional finite elements. The AST-spline volume has 94,764 degrees of freedom and the finite-element mesh has 4,033,128 degrees of freedom.}
\label{sopresults}
\end{figure}

Fig. \ref{sopresults} plots the first, the second, and the fifth mode shapes using AST-spline volumes and conventional finite elements. The finite element mesh has 5 elements in the thickness direction and a total of 4,033,128 degrees of freedom. ELFORM 2 is used to compute the mode shapes. As shown in Fig. \ref{sopresults}, good agreement between the coarse AST-spline volume and the overkill finite-element mesh is found.

\section{Conclusions and future work}

The subset of $C^1$-continuous non-negative bi-cubic AST-splines is significantly extended by allowing multiple EPs per face. We have mathematically proven that AST-splines with multiple EPs per face are linearly independent and form a non-negative partition of unity. We have numerically shown that optimal convergence rates for second- and fourth-order linear elliptic problems are obtained. We have built CAD geometries as complex as the side outer panel of a car using the T-spline capabilities of Autodesk Fusion360. Then, we have exported the control nets and combined them with our analysis-suitable basis functions to generate AST-spline surfaces. We have thickened these surfaces to obtain AST-spline volumes and imported them into LS-DYNA using B\'ezier extraction. We have solved eigenvalue problems using AST-spline volumes and conventional finite elements in LS-DYNA. Good agreement is found and AST-splines need significantly fewer degrees of freedom than conventional finite elements to reach a mesh-independent result.

Coming up with benchmark problems of increasing geometric complexity to evaluate the performance of different IGA techniques is a required task in the near future. These benchmarks would enable not only comparisons between different types of analysis-suitable splines that handle EPs \cite{toshniwal2017smooth, Wei2018, casquero2020seamless, majeed2017isogeometric, zhang2020manifold}, but also comparisons with non-boundary-fitted methods that deal with trimmed NURBS representations \cite{nagy2015numerical, breitenberger2015analysis, leidinger2019explicit,buffa2020minimal, antolin2019overlapping, wei2020immersed}.

\section*{Appendix A. Casteljau algorithm}

Given a cubic B\'ezier curve with B\'ezier control points $\mathbf{B}_1$, $\mathbf{B}_2$, $\mathbf{B}_3$, and $\mathbf{B}_4$ and parametric coordinate $u \in \left[ 0, a \right]$, the de Casteljau algorithm enables to refine the B\'ezier curve at $u = a/2$ as follows
\begin{equation}    \label{casteljaufirst}
\mathbf{B}_{1}^1 = \mathbf{B}_{1} \text{,} 
\end{equation}
\begin{equation}  
\mathbf{B}_{2}^1 =  \frac{\mathbf{B}_{1}}{2} + \frac{\mathbf{B}_{2}}{2}  \text{,} 
\end{equation}
\begin{equation}  
\mathbf{B}_{3}^1 =  \frac{\mathbf{B}_{1}}{4} + \frac{\mathbf{B}_{2}}{2}  + \frac{\mathbf{B}_{3}}{4} \text{,} 
\end{equation}
\begin{equation}  
\mathbf{B}_{4}^1 =  \mathbf{B}_{1}^2 = \frac{\mathbf{B}_{1}}{8} + \frac{3\mathbf{B}_{2}}{8}  + \frac{3\mathbf{B}_{3}}{8}   + \frac{\mathbf{B}_{4}}{8} \text{,} 
\end{equation}
\begin{equation}  
\mathbf{B}_{2}^2 =  \frac{\mathbf{B}_{2}}{4} + \frac{\mathbf{B}_{3}}{2}  + \frac{\mathbf{B}_{4}}{4} \text{,} 
\end{equation}
\begin{equation}  
\mathbf{B}_{3}^2 =  \frac{\mathbf{B}_{3}}{2} + \frac{\mathbf{B}_{4}}{2}  \text{,} 
\end{equation}
\begin{equation}    \label{casteljaulast}
\mathbf{B}_{4}^2 = \mathbf{B}_{4} \text{,} 
\end{equation}
\begin{figure} [t!] 
 \centering
 \subfigure[Before refinement]{\includegraphics[scale=0.4]{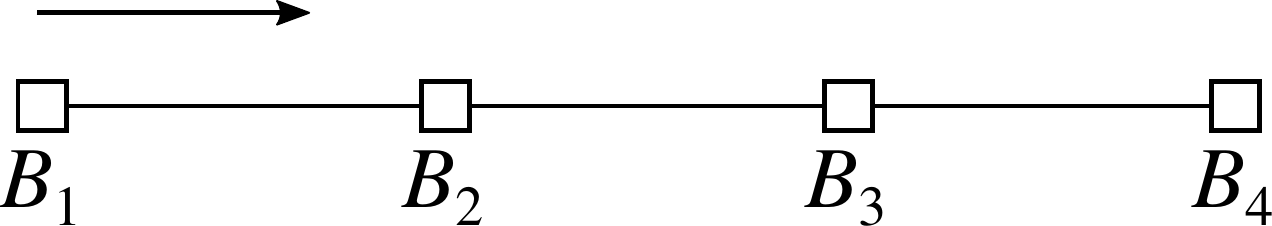}}
 \subfigure[After refinement]{\includegraphics[scale=0.4]{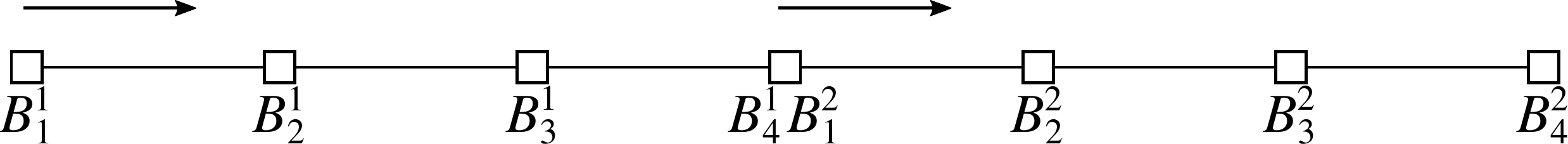}}
\caption{A cubic B\'ezier curve refined using the de Casteljau algorithm.}
\label{casteljau}
\end{figure}
where $\mathbf{B}_i^1$ and $\mathbf{B}_i^2$ with $i \in \{1,2,3,4\}$ are the B\'ezier control points of the two new cubic B\'ezier curves (see Fig. \ref{casteljau}). Given a cubic B\'ezier surface with B\'ezier control points $\mathbf{B}_i$ with $i \in \{1,2,...,16\}$ and parametric coordinates $u,v \in \left[ 0, a \right]$, the refinement of the B\'ezier surface at $u=v=a/2$ is obtained by tensor products of Eqs. \eqref{casteljaufirst}-\eqref{casteljaulast}.

\section*{Acknowledgements}

X. Wei was supported in part by the ERC AdG project CHANGE n. 694515 and the SNSF (Swiss National Science Foundation) project HOGAEMS n. 200021-188589. X. Li was supported by the National Key R\&D Program of China(2020YFB1708900), NSF of China (No.61872328) and the Youth Innovation Promotion Association CAS. Y. J. Zhang and K. Qian were partially supported by the NSF grants CMMI-1953323 and CBET-1804929 as well as Honda Motor Co., Ltd. T. J. R. Hughes was partially supported by the Office of Naval Research, USA (Grant Nos. N00014-17-1-2119 and N00014-13-1-0500). H. Casquero was partially supported by Honda Motor Co., Ltd. and Ansys Inc. This work used the Extreme Science and Engineering Discovery Environment (XSEDE), which is supported by National Science Foundation grant number OCI-1053575. Specifically, it used the Bridges system, which is supported by NSF award number ACI-1445606, at the Pittsburgh Supercomputing Center (PSC).





\bibliographystyle{elsarticle-num} 
\bibliography{./Bibliography}


\end{document}